\documentclass[11pt,english,oneside]{amsart}
\usepackage[T1]{fontenc}
\usepackage[latin9]{inputenc}
\usepackage{geometry}
\usepackage{amssymb}
\usepackage{color}

\usepackage{graphicx,color}
\usepackage{amsmath, amssymb, graphics}

\usepackage{graphicx}

\makeatletter
\numberwithin{equation}{section} %% Comment out for sequentially-numbered
\numberwithin{figure}{section} %% Comment out for sequentially-numbered
  \@ifundefined{theoremstyle}{\usepackage{amsthm}}{}
  \theoremstyle{plain}
  \newtheorem{thm}{Theorem}[section]
  \theoremstyle{plain}
  
  \theoremstyle{plain}
  \newtheorem{prop}[thm]{Proposition}
  \theoremstyle{Remark}
  \newtheorem{rem}[thm]{Remark}
  \theoremstyle{remark}
  
  \theoremstyle{plain}
  \newtheorem{lem}[thm]{Lemma}
  \newtheorem{mydef}{Definition}

\usepackage{geometry}

\newcommand{\R}{\mathbb{R}}

\def\bfR#1{{\bf R}^#1}

\def\com#1{ \hbox{#1}}

\def\inf{\hbox{\rm inf}}
\def\sup{\hbox{\rm sup}}
\def\max{\hbox{\rm max}}
\def\min{\hbox{\rm min}}

\smallskip
\def\<{{\langle }}
\def\>{{\rangle }}

\makeatother

\usepackage{babel}

\def\bfR#1{{\bf R}^#1}

\def\com#1{ \quad\hbox{#1}\quad}

\def\sf{\hbox{ $SF(n+1,k,a)$}}
\def\lf{\hbox{ $L(n+2,k,a)$}}

\def\R{\hbox{\bf R}}

\smallskip
\def\<{{\langle }}
\def\>{{\rangle }}

\makeatother

\begin{document}

\title[cmc hypersurfaces]{CMC hypersurfaces with two principal curvatures}

\author{ Oscar M. Perdomo }

\date{\today}

%\curraddr{Department of Mathematics\\
%Central Connecticut State University\\
%New Britain, CT 06050\\
%}

%\email{ perdomoosm@ccsu.edu}


\maketitle

\begin{abstract}

This paper explains the construction of all hypersurfaces with constant mean curvature -cmc- and exactly two principal curvatures on any space form endowed with a semi-riemannian metric. Here we will consider riemannian hypersurfaces as well as hypersurfaces with semi-riemannian metrics induced by the ambient space. We show explicit immersions for all these hypersurfaces. We will see that for any given ambient space, the family of cmc hypersurfaces with two principal curvatures depends on two parameters, $H$ and $C$. We describe the range for all possible $(H,C)$ associated with complete cmc hypersurfaces.   We end the paper by finding optimal bounds for the traceless second fundamental form in terms of $H$ and $n$.
\end{abstract}

\tableofcontents

\section{Introduction.}

Finding non trivial examples of constant mean curvature --cmc-- hypersurfaces contributes to a better understanding of this important family of submanifolds. In 1970 Otsuki, \cite{O}, discovered that for minimal hypersurfaces on the $(n+1)$-dimensional sphere, the conditions ``for every point in the hypersurface there are exactly two principal curvatures'' provides a family of hypersurfaces that could be understood by solving an ordinary differential equation.

%We will use a slightly different approach than the Bing-Ye Wu's approach.  

Bing-Ye Wu has made some major contributions constructing cmc riemannian hypersurfaces  with exactly two principal curvatures. He showed that these hypersurfaces  depend on two parameters, a constant $C$ and  the mean curvature $H$. He was able to find explicit parametrizations for most of the examples when $C\ne0$. When the ambient space is the sphere or the Euclidean space, Wu, \cite{W} and  \cite{W4}, found complete classifications. When the ambient space is the hyperbolic space or the de Sitter space, Wu \cite{W5} and \cite{W3} found some explicit parametrizations, but not all of them due to the fact that he did not find an expression for the parametrizations when $C=0$. Also, he has some small mistakes in the papers \cite{W1}, \cite{W2}, \cite{W3} and \cite{W5} due to the fact that he missed some solutions when $C$ reaches the  boundary values of the set of admissible $(H,C)$.  In particular, in  paper \cite{W2} the author claims to have all the solution but he is missing those with $C=0$. For the Lorentz-Minkowski space and the anti-de-Sitter space, Wu \cite{W2} and \cite{W1} claims to have a classification but again due to the same type of mistake, he is missing some solutions. For the Lorentz-Minkowski space $C=0$ turns out to be a boundary value for the admissible values of $(H,C)$ that generate cmc hypersurace examples with two principal curvatures. The last section in this paper shows some of the examples missed in Wu's papers.

The process of deciding if a solution is embedded is done by analyzing the profile curve. Due to the length of the present paper, the author is leaving this task for a future work. For the case when ambient space is the  three-dimensional  sphere, this process was started in \cite{P}  and completed in \cite{A-L}. For any ambient space, this process of deciding if the hypersurface is embedded or not, starts by finding a more explicit formulas for the parametrization, like those presented in \cite{P2} %or \cite{P5}.

%%%%%%%
%\rule{30mm}{2pt}
%%%%%%%

In \cite{P}, the author provided  a detailed study of the same type of hypersurfaces that we are considering here for the particular case when the ambient space is the $(n+1)$-dimensional sphere. The results in \cite{P}, contributed to a breakthrough in the understanding of embedded surfaces with cmc on $S^3$ by Li and Andrews \cite{A-L}. Andrews and Li showed that the only embedded tori with constant mean curvature in the three dimensional sphere are those that are rotational, which were very explicitly listed in the paper \cite{P}. Another similar breakthrough is the proof given by Brendle, \cite{B},  of the  Lawson conjecture that stated that the only embedded minimal tori on the three dimensional sphere is the Clifford torus. Regarding the existence of embedded examples, Otsuki, \cite{O}, proved that the only embedded minimal hypersurfaces on spheres with two principal curvatures are those with constant principal curvature. Brito and Leite, \cite{B-L} showed that the that there are embedded rotational non trivial hypersurface on the sphere with cmc  $H$, provided $H$ is close to zero. These examples has exactly two principal curvatures. In \cite{P1}, the author discuss a family of  cmc embedded hypersurfaces with two principal curvatures on the hyperbolic space.

We are considering all types of ambient spaces, in particular, the Lorentz-Minkowsky space, the de Sitter space and the anti de Sitter space. Cmc hypersurfaces in these ambient spaces, as explained in \cite{M1}, have interest in General relativity due to the fact that they are convenient initial data for the Cauchy problem of the Einstein equation and also they are suitable for studying propagation of gravitational ways. Some other papers with results on these type of hypersurfaces are \cite{A}, \cite{C-Y}, \cite{I}, \cite{M2}, \cite{N}, \cite{T}.  

There is a huge amount of literature with theorems showing that hypersurfaces with cmc $H$ and  two principal curvatures that  satisfy some bounds for the values of the norm of the shape operator must have constant principal curvatures. See for example \cite{C-W}, \cite{Ad1}, \cite{Ad2}, \cite{Ad3}, \cite{Ad4}, \cite{Ad5}, \cite{Ad6}, \cite{Ad7}, \cite{Ad8}, \cite{Ad9}, \cite{Ad10}, \cite{Ad11}, \cite{Ad12}, \cite{Ad13}, \cite{Ad14}, \cite{Ad15}, \cite{Ad16}, \cite{Ad17}, \cite{Ad18}, \cite{Ad19}, \cite{Ad20}. In section \ref{s2 results}  we provide similar characterizations an optimal bounds for the norm of the shape operator. In particular we generalize and extend Alias and Martinez's result \cite{A-M} to all the ambient spaces and to every type of hypersurface.

\section{Space Forms.}

For any non negative integer $k< m$, we will denote by $\R_k^{m}$ the space $\R^{m}$ endowed with the metric

$$\<v,w\>=-v_1w_1-\dots-v_kw_k+v_{k+1}w_{k+1}+\dots+v_{m}w_{m}$$

We also define

$$SF(n+1,k,-1)=\{x\in\bfR{{n+2}}_{k+1}: \<x,x\>=-1\} $$

$$ SF(n+1,k,1)=\{x\in\bfR{{n+2}}_k: \<x,x\>=1\}\com{and}\quad SF(n+1,k,0)=\bfR{{n+1}}_k$$
%and
%$$SF(n+1,k,0)=\bfR{{n+1}}_{k} $$

It is not difficult to check that $S(n+1,k,a)$ with the metric induced by $\<\, ,\, \>$ is either a Riemannian or semi-Riemannian metric. Notice that $SF(n+1,0,1)$ is the $n+1$-dimensional sphere, $SF(n+1,0,-1)$ is the $(n+1)$-dimensional hyperbolic space, $SF(n+1,0,0)$ is the $(n+1)$-dimensional Euclidean space, $SF(n+1,1,0)$ is the $(n+1)$-dimensional Minkowsky space, $SF(n+1,1,1)$ is the $(n+1)$-dimensional de-Sitter space and $SF(n+1,1,-1)$ is the $(n+1)$-dimensional anti-de Sitter space.

For notation sake, we will be denoting the metric on $SF(n+1,k,a)$ just as $\< \, ,\, \>$. The index $k$ will be understood by the ambient space. For any $x\in SF(n+1,k,a)$ we have that the tangent space at $x$ is given by,

$$ T_xSF(n+1,k,a)=\{v\in\bfR{{n+2}}: \<v,x\>=0\}\com{when $a=\pm1$ }$$

and,

$$ T_xSF(n+1,k,a)=\bfR{{n+1}} \com{when $a=0$} $$

\begin{rem} We will be referring to $SF(n+1,k,a)$ as an $(n+1)-$dimensional  space forms. We define

$$ L(n+2,k,a)=\begin{cases} \bfR{{n+2}}_{k} & \text{ if $a=1$}  \\  \bfR{{n+2}}_{k+1} & \text{ if $a=-1$}  \end{cases}$$

Notice that $\sf$ inherits the semi-riemannian metric from $\lf$
\end{rem}

\section{Hypersurfaces of space forms with two principal curvatures}

Let us assume that $M\subset SF(n+1,k,a)$ is a hypersurface, that is, $M$ is the image of an immersion $M_0\longrightarrow SF(n+1,k,a)$ where $M_0$ is an $n$-dimensional manifold.  We will assume that $M_0$ is connected. Usually we will not make reference to the manifold $M_0$, for example, the statement ``$M$ is connected'' must be interpreted as ``$M_0$ is connected''. 

\subsection{The Gauss map} A Gauss map on a hypersurface $M\subset \sf$ is a function $\nu:M\to \bfR{m}$, ($m=n+2$ when $a=\pm1$ and $m=n+1$ when $a=0$) such that $|\<\nu(x),\nu(x)\>|=1$,  $\nu(x)$ is perpendicular to $T_xM$ and $\nu(x)\in T_x\sf$ for all $x\in M$. Notice that the condition $\nu(x)\in T_x\sf$  reduces to $\<\nu(x),x\>=0$ when $a=\pm1$. A Gauss map essentially keeps track of the tangent space $T_xM$, because we have

$$T_xM=\{v\in T_x\sf: \<x,\nu(x)\>=0\}=\hbox{orthogonal complement of $\nu(x)$ on $T_x\sf$}$$

When the ambient space is Riemannian, it is always possible to locally define a Gauss map and, when the ambient space is not Riemannian we may not be able to, even locally, define a Gauss map. In this paper, we will be considering only hypersurfaces where we can globally define a Gauss map.  This condition on the existence of the Gauss map implies that the metric induced on $M$ by  $SF(n+1,k,a)$ is either Riemannian or  semi-Riemannian, that is, it implies that for any $x\in M$, there is not vector $v\in T_xM$ such that $\<v,w\>=0$ for all $w\in T_xM$.

\begin{rem}
One of the differences between hypersurfaces of Riemannian manifolds and semi-Riemannian maifolds is that in the semi-Riemannian case the existence of a Gauss map (even locally) is not guaranteed. The fact that the metric defined on the hypersurface $M\subset \sf$ is semi-Riemannian, allows us to talk about the tangential component on $T_xM$ of a vector in the vector space $T_x\sf$. 
\end{rem}

We will consider the following connections: $\bar{\nabla}$ will denote the Levi-Civita connection on $\bfR{{n+2}}$ (or on $\bfR{{n+1}}$ when $a=0$). Recall that for a vector field $X$ defined on $\bfR{m}_k$ and a vector $v$, we have

$$\bar{\nabla}_v X (x)=\frac{d X(\alpha(t))}{dt}\Big|_{t=0}$$

where $\alpha(t)$ is any curve such that $\alpha(0)=x$ and $\alpha^\prime(0)=v$.  We will also denote by $\nabla$ the Levi-Civita connection on $M$ induced by the ambient space $SF(n+1,k,a)$. In this case, if $X$ is a tangent vector field on $M$, and $v\in T_xM$, then 

$$\nabla_vX(x)=\hbox{tangential component of $\bar{\nabla}_vX(x)$ on $T_xM$}$$

\subsection{The shape operator} The shape operator of a hypersurface $M$  measures how the tangent spaces change. Notice that to measure how  the tangent spaces change  is equivalent to  measure how the Gauss map changes. For a given $x\in M$ all these changes are saved on a  linear transformation from $T_xM$ to $T_xM$ known as the shape operator which is denoted by $A$  and  defined as

$$A_x(v)=-\bar{\nabla}_v\nu\,  (x)\com{for any $v\in T_xM$ }$$

For any hypersurface $M\subset \sf$ we have that $\<A(v),w\>=\<v,A(w)\>$. The eigenvalues of the shape operator are called {\it principal curvatures}. When the metric on $M$ is Riemmanian we have that the  shape operator can be diagonalized, therefore in this case, we know that, counting multiplicities, there are $n$ principal curvatures. We cannot expect the same when the metric on $M$ is semi-Riemannian.  We will be studying hypersurfaces with two principal curvatures and since hypersurfaces whose induced metric is not riemannian may have a no-diagonalizable shape operator, we need the following definition in order to be precise above the set of hypersurfaces that we will be considering.

 \begin{mydef} \label{two principal curvatures general} We will say that a connected hypersurface $M\subset SF(n+1,k,a)$ has exactly two principal curvatures if
 \begin{itemize}
 \item
 it is possible to define a Gauss map on $M$.
 \item
 there exist two smooth functions  $\lambda,\mu:M\to {\bf R}$ such that $\lambda(x)\ne\mu(x)$ for all $x\in M$ and,
 \item
 there exists a positive integer $l\le n-1$ such that for every $x_0\in M$, it is possible to find a locally defined orthonormal frame $\{e_1,\dots,e_n\} $ near $x_0$, such that $ A(e_i)=\lambda e_i$ for $i=1,\dots,l$
and $A(e_j)=\mu e_j$ for $j=l+1,\dots,n$.
\item When $l=n-1$ we will assume that the vector field $e_n$ is globally defined and we will denote by $\Gamma(x)=\{v\in T_x M: A(v)=\lambda v\}$ the distribution associated with the eigenvalue $\lambda$.
\end{itemize}

For these hypersurfaces, we will define $b=\<\nu,\nu\>$ and, in the case that  $l=n-1$, that is, in the case that the multiplicity of $\mu$ is 1, we will define $d=\<e_n,e_n\>$. Recall that $b$ and $d$ are either $1$ or $-1$.

\end{mydef}

Using the same arguments as in \cite{P}, it can be shown that if the multiplicities of each principal curvature are both greater than one, that is, if $2 \le l\le n-2$ in Definition \ref{two principal curvatures general}, then the principal curvatures are constant, and therefore, the manifold can be easily described. For this reason we will assume that $M$ has one principal curvature with multiplicity $(n-1)$ and the other with multiplicity $1$; that is, we will assume that $l=n-1$. For a general hypersurface with a Gauss map defined on it, the mean curvature is given by 

$$H=\frac{1}{n} (\<A(w_1),w_1\>\<w_1,w_1\>+\dots+\<A(w_n),w_n\>\<w_n,w_n\>)$$

where $\{w_1,\dots,w_n\}$ is an orthonormal basis in the sense that $|\<w_i,w_j\>|=\delta_{i,j}$. As we mentioned in the abstract, we will be studying hypersurfaces with constant mean curvature. In our case, under the assumption that $M$ has two principal curvatures, one of them with multiplicity one, this condition reduces to

\begin{eqnarray}\label{Constant mth curvature equation}
(n-1) \lambda+\mu = nH \com{where $H$ is constant}
\end{eqnarray}

%\vskip.8cm

%\centerline{\bf Throughout all the sections in these notes, except the last one on Isoparametric}
%\centerline{\bf  hypersuraces, we will be assuming that  $M\subset \sf$ is a hypersurface }
%\centerline{\bf  with cmc and two principal curvatures, one of them with multiplicity one}

Throughout the rest of this paper we will be assuming that  $M\subset \sf$ is a hypersurface 
 with cmc and two non constant principal curvatures, one of them with multiplicity one.
% except the last one on Isoparametric hypersuraces,

\section{Structure of hypersurfaces with two principal curvature, one of them with multiplicity one}

Let us assume that $M\subset \sf$ is a connected hypersurface with two principal curvatures. We will also assume that one of the principal curvature has multiplicity one and that $M$ has constant mean curvature $H$. The functions $\lambda$ and $\mu$, the vector fields $e_n$ and $\nu$ and the distribution $\Gamma$ are defined as before, see Definition \ref{two principal curvatures general}. We will be assuming that $\lambda-\mu>0$ everywhere. We can assume this without loss of generality because we can change the Gauss map $\nu$ by $-\nu$ if necessary. In this section we will mention a list of theorems that will guide us to a better understanding of these hypersurfaces. The proof of the theorems will be developed throughout the paper. The goal of this section is to show the ``whole picture'' of the important properties of these
hyersurfaces.

The following definition will be useful

\begin{mydef} We define $w:M\to {\bf R}$ by   $w=\left(\frac{\lambda-\mu}{n}\right)^{-\frac{1}{n}}$. We also define the vector field   $\eta$ as

$$\eta(p) = - e_n(w)\,  e_n+ d b \lambda\,  w \nu -da w\, p$$

where $e_n(w)$ is the derivative of the function $w$ in the direction of the vector $e_n$.

\end{mydef}

In the same that way every non vertical line in $\bfR{2}$ has a slope, we have that every hypersurface $M\subset \sf$ has a constant $C$ associated with it. The following theorem allows us to define this constant $C$.

\begin{thm} The function $e_n(w)^2+dw^2\left(a+b\lambda^2  \right)$ is constant, that is, for some constant $C$ we have,

$$C=e_n(w)^2+dw^2\left(a+b\lambda^2  \right)$$ 

This number $C$ will be denoted as the constant associated to $M$.

\end{thm} 

%Beginning the case a\ne 0 and C\ne0

When  $C\ne0$, we can naturally define a planar curve called the profile curve of the hypersurface. Regardless if $C=0$ or not, we can define a two-dimensional plane associated to $M$. 

\begin{thm}\label{the plane} The vector field $\eta$ satisfies $\<\eta,\eta\>=C d$.  Moreover if  $a\ne0$, then
the two dimensional plane  given by

$$\Pi=\Pi_p=\{v=c_1 p+c_2 e_n+c_3\nu:\, c_i\in \mathbb{R}\quad \hbox{and}\quad  \<v,\eta\>=0\}$$

is independent of the point $p$. We also have that, if additionally $M$ is complete and $C\ne0$, then, the orthogonal projection on $\Pi$ of any integral curve of the vector field $e_n$ lies on a single curve  of $\Pi$ that we will call the profile curve of $M$. 
\end{thm}

The theorem above places us very close to a full classification. We would like to point out that this definition of profile curve is slightly different to the one defined in previous papers, for example the one defined in \cite{B-L}.

Here is one way to describe every possible $M$ when  $C\ne0$.

\begin{thm}\label{the immersions 1 intro}
Let $\sf$ be a space form with sectional curvature $a=1$, $a=-1$ or $a=0$ and let $M\subset \sf$ be a complete and connected hypersurface with constant mean curvature, two non-constant principal curvatures  and such that the constant $C$ associated to $M$ is not zero. If $x_0\in M$, then, $M$ agrees with the immersion 

$\phi:S\times \mathbb{R} \longrightarrow \sf$ given by 

$$\phi(y,t)= \frac{g(t)}{\sqrt{|C|}}\, y+ \gamma(t)+\frac{g(t)}{C}\, \rho_0 $$

where $\gamma(t)$ is an integral curve of the vector field $e_n$ that contains $x_0$, $\rho_0=\eta(x_0)$, $g(t)=w(\gamma(t))$ and
$$S=\{y\in \Gamma(x_0)+Span\{\eta(x_0)\}: \<y,y\> = d\frac{|C|}{C}\}.$$

\end{thm}

The case $C=0$ is a little different.  Regardless of the values of $a,b,d$ and $C$, Section \ref{explicitEx} provides explicit examples of immersions in terms of the solution of an ordinary differential equation. In Section \ref{section the immersions} we show that every complete cmc hypersurface  in $\sf$ with two principal curvatures, one of them with multiplicity one, is described by one of the explicit examples presented in Section  \ref{explicitEx}. 

%We also have that if additionally $M$ is complete, then the curve obtained by projecting on the plane $l$ an integral curve $\gamma_p$ defined in the whole set of real numbers is %independent of $p$. This curve in the plane $\Pi$ will be denoted as the profile curve.

%Fin the case a=0 and C\ne 0

In section \ref{mods}  we will be providing a precise relation between the values $H$ and $C$ for which there exists a complete hypersurface with cmc and two non constant principal curvatures. The completeness of $M$ turns out to be equivalent to the existence of a positive solution $g$ of the ordinary differential equation,

$$(g^\prime)^2+d\, g^2 \, (a+b(H+g^{-n})^2)=C $$

defined in the whole real line.  The solution of this ODE will be studied in section \ref{solvingode}.

%We will see that some of the complete immersions may not be properly immersed and on the other hand they may be embedded. by %explicitly solving the ode that defines $\gamma(t)$, we will be classifying values of $C$ and $H$ that provides embedded examples.

\section{Basic equations}\label{basic equations}

In this section we compute the Curvature tensor of $M$ and, we use the Codazzi equations to prove that the distribution $\Gamma$ is integrable and prove several more equations that will help us to understand the way  $M$ is immersed in $\sf$. To simplify writing let us define 

$$N=n+2\com{if} a=\pm1,\com{and} N=n+1\com{if} a=0$$

For any fixed vector $v\in \bfR{N}$, we define, $l_v:M\to {\bf R}$, $f_v:M\to {\bf R}$ and $V:M\to \bfR{N}$ by

$$l_v(x)=\<x,v\>\quad f_v(x)=\<\nu(x),v\> \com{and} V(x)=v-a\, l_v(x) x-b\, f_v(x) \nu(x)$$

A direct computation shows that $V$ is a tangent vector field and that

$$\nabla l_v=V\com{and} \nabla f_v=-A(V)$$

and

$$\nabla_wV=-a\, l_v(x) w + b\, f_v(x) A_x(w)\com{for any $w\in T_xM$}$$

Notice that if $v\in T_{x_0}M$ then $l_v(x_0)=0$, $f_v(x_0)=0$ and therefore, $\nabla_wV(x_0)$ vanishes for any $w\in T_{x_0}M$. Let us compute the Curvature tensor on $M$. If $v_i\in T_{x_0}M\subset \bfR{N}$ we have that,

\begin{eqnarray*}
\<R(v_1,v_2)v_3,v_4\>&=&\<\nabla_{v_2}\nabla_{V_1}V_3-\nabla_{v_1}\nabla_{V_2}V_3,v_4\>\cr
&=&\<\nabla_{v_2}(-a\, l_{v_3} V_1 + b f_{v_3} \, A(V_1))-\nabla_{v_1}(-a\, l_{v_3} V_2 + b f_{v_3} \, A(V_2)),v_4\>\cr
&=&a(\<v_1,v_3\>\<v_2,v_4\>-\<v_1,v_4\>\<v_2,v_3\>)+b(\<A(v_3),v_1\>\<A(v_2),v_4\>-\<A(v_3),v_2\>\<A(v_1),v_4\>)
\end{eqnarray*}

The equation above is the Gauss equation.

\begin{rem} A similar computation shows that the Curvature tensor of $\sf$ is given by 

$$R(v_1,v_2,v_3,v_4)=a(\<v_1,v_3\>\<v_2,v_4\>-\<v_1,v_4\>\<v_2,v_3\>)$$

Therefore, $\sf$ has constant sectional curvature $a$.

\end{rem}

Recall that for any three vector fields $X$, $Y$  and $Z$ on $M$,  $DA$, the covariant derivative of the shape operator, is defined by 

$$DA(X,Y,Z)=Z \<A(X),Y\> - \<A(\nabla_ZX),Y\> - \<A(X),\nabla_ZX \>$$

Since $\sf $ has constant sectional curvature, we have the following equations which are known as  the Codazzi equations

\begin{eqnarray}
\label{Codazzi equation}
DA(X,Y,Z)=DA(Z,Y,X)
\end{eqnarray}

The next theorem is essential in the classification of hypersurfaces with two principal curvatures, its proof is similar to the one made in \cite{P} for hypersurfaces of spheres.

%Thm begins
\begin{thm}
\label{equations}

 Let $M\subset SF(n+1,k,a)$ be a  hypersurface with constant mean curvature and two principal curvatures, one of them with multiplicity 1. If $n\ge 3$ and $\mu,\lambda:M\to{\bf R}$,  $b$, $d$ and $\{e_1,\dots,e_n\} $ are like in Definition \ref{two principal curvatures general}, then,

 \begin{eqnarray*}
v(\lambda)&=&0 \com{for any } v\in \Gamma= \hbox{Span}\{e_1,\dots,e_{n-1}\}\\
\nabla_v e_n&=&\frac{e_n(\lambda)}{\mu-\lambda}\, v\com{for any } v\in \Gamma \\
\nabla_{e_n}e_n& & \com{vanishes}  \\
ad+bd\, \lambda\mu&=&e_n(\frac{e_n(\lambda)}{\lambda-\mu})-(\frac{e_n(\lambda)}{\lambda-\mu})^2
 \end{eqnarray*}

Moreover we have that,

$$ [e_i,e_j] \com{and} [e_i,e_n]\com{ are in  $\Gamma$  for any } i,j\in\{1,\dots,n-1\}$$

and

$$\lambda, \, \mu \com{and} e_n(\lambda) \com{are constant along the integral submanifolds of $\Gamma$} $$

\end{thm}

%Thm ends

\begin{proof}
For any $i,j\in\{1,\dots,n-1\}$ with $i\ne j$ (here we are using the fact that the dimension of $M$ is greater than 2) and any $k\in\{1,\dots,n\}$, we have that,

 \begin{eqnarray*}
 DA(e_i,e_j,e_k)&=&e_k\<\, A(e_i),e_j\,\> -\<\, A(\nabla_{e_k}e_i),e_j\,\> -\<\, A(e_i),\nabla_{e_k}e_j\,\> \\
&=&e_k(\lambda \<\, e_i,e_j\,\> )-\<\, \nabla_{e_k}e_i,A(e_j)\,\> -\lambda\<\, e_i,\nabla_{e_k}e_j\,\> \\
&=& e_k( 0)-\lambda\<\, \nabla_{e_k}e_i,e_j\,\> -\lambda\<\, e_i,\nabla_{e_k}e_j\,\> \\
&=& 0-\lambda e_k(\<\, e_i,e_j\,\> )\\
&=& 0
\end{eqnarray*}

On the other hand,
\begin{eqnarray*}
DA(e_i,e_i,e_j)&=&e_j\<\, A(e_i),e_i\,\> -\<\, A(\nabla_{e_j}e_i),e_i\,\> -\<\, A(e_i),\nabla_{e_j}e_i\,\> \\
 &=&e_j(\lambda)\<e_i,e_i\>-\lambda e_j(\<\, e_i,e_i\,\> )\\
&=& e_j(\lambda)\<e_i,e_i\>
\end{eqnarray*}

By the Codazzi equations (\ref{Codazzi equation}), we get that $e_j(\lambda)=0$, for all $j\in\{1,\dots,n-1\}$ and therefore $v(\lambda)=0$ for any  $v\in \hbox{Span}\{e_1,\dots,e_{n-1}\}$. Now,
\begin{eqnarray*}
DA(e_i,e_n,e_j)&=&e_j\<\, A(e_i),e_n\,\> -\<\, A(\nabla_{e_j}e_i),e_n\,\> -\<\, A(e_i),\nabla_{e_j}e_n\,\> \\
&=&e_j(\lambda\<\, e_i,e_n\,\> )-\<\, \nabla_{e_j}e_i,A(e_n)\,\> -\lambda\<\, e_i,\nabla_{e_j}e_n\,\> \\
&=&e_j(0)-\mu\<\, \nabla_{e_j}e_i,e_n\,\> -\lambda\<\, e_i,\nabla_{e_j}e_n\,\> +( \lambda\<\, \nabla_{e_j}e_i,e_n\,\> -\lambda\<\, \nabla_{e_j}e_i,e_n\,\>)\\
&=&(\lambda-\mu)\<\, \nabla_{e_j}e_i,e_n\,\> -\lambda\ e_j(\<\, e_i,e_n\,\> )\\
&=& (\mu-\lambda)\<\, e_i,\nabla_{e_j}e_n\,\>
\end{eqnarray*}

Since $\mu-\lambda\ne0$, using Codazzi equations  we get that,

\begin{eqnarray}\label{den1}
\<\, e_i,\nabla_{e_j}e_n\,\> =0\com{for any $i,j\in\{1,\dots,n-1\}$ with $i\ne j$}
\end{eqnarray}

Now, for any $i\in\{1,\dots,n-1\}$, using the same type of computations as above we can prove that,

$$DA(e_i,e_i,e_n)=e_n(\lambda)\,\<e_i,e_i\> =\, DA(e_i,e_n,e_i)=(\mu-\lambda)\<\, e_i,\nabla_{e_i}e_n\, \>$$

and

$$DA(e_n,e_n,e_i)=d\, e_i(\mu)\,=\, DA(e_i,e_n,e_n)=(\mu-\lambda)\<\, e_i,\nabla_{e_n}e_n\, \>$$

Since we are assuming that $M$ has constant mean curvature and we already know that $v(\lambda)=0$ for any $v\in \Gamma$, then, we also have that $v(\mu)=0$ and therefore, $\<\nabla_{e_n}e_n,e_i\>=0$ for all $i\in\{1,\dots,n-1\}$. Since $e_n$ is a unit vector field, we have that $\<\,\nabla_{e_k}e_n,e_n\,\>=0$ for any $k$ and therefore $\nabla_{e_n}e_n$ vanishes. Using the fact that any vector $z\in T_xM$ can be written as $z=\sum_{k=1}^n \<z,e_k\> \<e_k,e_k\> \, e_k$ we obtain from the previous equations that for any $i\in\{1,\dots,n-1\}$,

\begin{eqnarray*}\label{den2}
\<\, {e_i},\nabla_{e_i}e_n\,\>=
 \frac{{e_n}(\lambda)}{\mu-\lambda}\<e_i,e_i\>\com{ and} \nabla_{e_i}e_n=\frac{e_n(\lambda)}{\mu-\lambda}\, {e_i} 
\end{eqnarray*}

using linearity we get that

\begin{eqnarray*}
\nabla_{v}e_n=\frac{e_n(\lambda)}{\mu-\lambda}\, {v}\com{for any $v\in \Gamma$. }
\end{eqnarray*}

Notice that for any $i,j\in\{1,\dots,n-1\}$ with $i\ne j$, Equation (\ref{den1}), give us that

$$\<\, [e_i,e_j],e_n\,\> =\<\, \nabla_{e_i}e_j-\nabla_{e_j}e_i,e_n\,\>=-\<\,e_j,\nabla_{e_i}e_n\,\> + \<\,e_i,\nabla_{e_j}e_n\,\>=0$$

likewise we have that

$$\<\, [e_i,e_n],e_n\,\> =\<\, \nabla_{e_i}e_n-\nabla_{e_n}e_i,e_n\,\>=-\<\,e_i,\nabla_{e_n}e_n\,\>+\<\,e_i,\nabla_{e_n}e_n\,\>=0$$

Therefore $[e_i,e_j]$ and $[e_i,e_n]$ are in $\hbox{Span}\{e_1,\dots,e_{n-1}\}$. The fact that $e_n(\lambda)$ is constant along the integral submanifolds of $\Gamma$ follows from the following equation,

$$e_i(e_n(\lambda))=[e_i,e_n] (\lambda)+e_n(e_i(\lambda))=0 \com{for $i=1,\dots,n-1$}$$

 Finally, we will use Gauss equation to prove the differential equation on $\lambda$. Since $[e_n,e_1]\in\hbox{Span}\{e_1,\dots,e_{n-1}\}$, using the Gauss equation we get,

\begin{eqnarray*}
(a+b \lambda\mu)\, d\, \<e_1,e_1\>  &=&\<\, R(e_n,e_1)e_n,e_1\,\> \\
&=&\<\,\nabla_{e_1}\nabla_{e_n}e_n-\nabla_{e_n}\nabla_{e_1}e_n+\nabla_{[e_n,e_1]}e_n,e_1\, \>\\
&=&\<\,0-\nabla_{e_n}(\frac{e_n(\lambda)}{\mu-\lambda}\, e_1)+ \frac{e_n(\lambda)}{\mu-\lambda}\, [e_n,e_1]   \, ,e_1 \>\\
&=& -e_n(\frac{e_n(\lambda)}{\mu-\lambda})\, \<e_1,e_1\> + \frac{e_n(\lambda)}{\mu-\lambda} \<\,\nabla_{e_n}e_1-\nabla_{e_1}e_n ,e_1  \, \>\\
&=&-e_n(\frac{e_n(\lambda)}{\mu-\lambda}) \, \<e_1,e_1\> -(\frac{e_n(\lambda)}{\mu-\lambda})^2\, \<e_1,e_1\>\\
&=&\big{(}e_n(\frac{e_n(\lambda)}{\lambda-\mu})-(\frac{e_n(\lambda)}{\lambda-\mu})^2\big{)} \, \<e_1,e_1\>
\end{eqnarray*}

\end{proof}

\begin{rem}\label{multiplicities} Using the same type of arguments as in the previous theorem we can easily show that if  $M\subset \sf$ is a hypersurface with constant mean curvature and with  two principal curvatures, both with multiplicity greater than 1, then the principal curvatures must be constant, that is, $M$ must be isoparametric.

\end{rem}

\section{Useful vector fields and first integrals}\label{vector fields}

As mentioned earlier, $M\subset \sf$ will denote a hypersurface with cmc $H$ and two principal curvatures, one of them with multiplicity one. By changing the Gauss map $\nu$ by $-\nu$, if necessary, we will assume that $\lambda -\mu$ is always positive and, since the case when $\lambda$ and $\mu$ are constant functions is relatively easy, we will also assume that $\lambda -\mu$ is not a constant function. Let us start this section defining some functions and vector fields.

\begin{mydef}\label{def function and vector fields}
We define the function $w:M\to {\bf R}$ as

$$w=\big{(}\frac{1}{n}(\lambda-\mu) \big{)}^{-\frac{1}{n}}=
(\lambda-H)^{-\frac{1}{n}}$$

and the vector fields $\eta$, $X$, $Y$ and $Z$

\begin{eqnarray*}
\eta&=&-e_n(w)\, e_n+d b \lambda\, w\, \nu-d a\, w\, x\\
X&=&\nu+\lambda \, x\\
Y&=& w \, e_n- e_n(w) \, x\\
Z&=& \lambda\, w\, e_n\, +\, e_n(w)\, \nu
\end{eqnarray*}

where $\nu$ is the Gauss map on $M$, $e_n(w)$ is the directional derivative of the function $w$ in the direction $e_n$ and $b$ and $d$ are like in Definition \ref{two principal curvatures general} and $x$ represents the point in $M$ where all these vector fields are being evaluated and also it is viewed as a vector in $\bfR{N}$.

\end{mydef}

\begin{rem}
The fact that $\frac{1}{n}(\lambda-\mu)=
\lambda-H$ easily follows form the cmc equation $(n-1)\lambda+\mu=nH$. We will be using both equations, $w^{-n}=\frac{1}{n}(\lambda-\mu)$ and $w^{-n}=\lambda-H$.
\end{rem}
\begin{rem}\label{X,Y,Z and eta}
The vector field $\eta$ defined above plays an important role in the construction of the immersions of the hypersurfaces with two principal curvatures. It looks like everything in these hypersurfaces has to do with this vector field. As an example, when $a\ne 0$, the vector fields $X$, $Y$ and $Z$ are perpendicular to the vector field $\eta$. Up to a factor, they are the only three vector fields that are perpendicular to $\eta$  that can be written as linear combination of two of the three vector fields $x$, $e_n$ and $\nu$.
\end{rem}

\begin{thm}\label{local relations} The function $w$ and the vector field $\eta$, $X$, $Y$ and $Z$ satisfy the following equations

\begin{eqnarray*}
& &\com{(i)} -\frac{e_n(w)}{w}=\frac{e_n(\lambda)}{\lambda-\mu} \com{or equivalently} we_n(\lambda)+\lambda e_n(w)=\mu e_n(w)\\
& &\com{(ii)} \bar{\nabla}_{v} e_n=\frac{e_n(w)}{w}\, v \com{for any $v\in \Gamma$}\\
& &\com{(iii)}e_ne_n(w)=-a d w-b d w\lambda\mu\\
& &\com{(iv)} (e_n(w))^2+d\, w^2 \, (a+b\lambda^2)=C\com{for some constant $C$}\\
& &\com{(v)}\bar{\nabla}_{e_n}x= e_n\qquad  \bar{\nabla}_{e_n}\nu= -\mu e_n \qquad \bar{\nabla}_{e_n}e_n= b d\mu \nu-a d x\\
& &\com{(vi)} \bar{\nabla}_{v}X, \qquad \bar{\nabla}_{v}Y\ \com{and}  \bar{\nabla}_{v}Z\com{vanish for all} v\in \Gamma\\
& &\com{(vii)}\bar{\nabla}_{e_n}\eta \com{vanishes}\\
& &\com{(viii)}\bar{\nabla}_{v}\eta =-\frac{C}{w} \, v\com{for any $v\in \Gamma$, where $C$ is defined in (iv)}\\
& &\com{(ix)} -w\, \eta+d b\,  \lambda \, w^2 \, X-e_n(w)\, Y=C \, x\\
& &\com{(x)} \bar{\nabla}_{e_n}Y=db w\mu X \\
& &\com{(xi)} \bar{\nabla}_{e_n}X=\frac{\lambda-\mu}{w} Y \\
& &\com{(xii)} \bar{\nabla}_{e_n}Z=-da w\, X \\
& &\com{(xiiii)} \<\eta,\eta\>=d C \com{where $C$ is defined in (iv)}\\
& &\com{(xiv)} \hbox{If} \com{$a=0$, then the vector $Z$ is constant.}
\end{eqnarray*}

\end{thm}

\begin{proof} Let us start by proving (i). Since $w>0$  we have that  $\lambda-H>0$.   By taking the derivative in the direction $e_n$ on the equation  $w^{-n}=\lambda-H$ we get,

$$ -nw^{-n-1}e_n(w)=e_n(\lambda) $$

Therefore, we have

\begin{eqnarray*}
e_n(\lambda)&=&-n\frac{e_n(w)}{w} w^{-n}\\
&=&-n \frac{e_n(w)}{w} \frac{1}{n}(\lambda-\mu)\\
&=&-\frac{e_n(w)}{w}(\lambda-\mu)
\end{eqnarray*}

This finishes the proof of (i). Part (ii) follows from the following three observations: the first one is that  $\nabla_{v} e_n=\frac{e_n(\lambda)}{\mu-\lambda}\, v=\frac{e_n(w)}{w}\, v $ by Theorem \ref{equations}. The second and third observations are that, if we take the derivative of $\<e_n,\nu\>=0$ with respect to $v\in \Gamma$ we get that $\<\bar{\nabla}_{v} e_n,\nu\>=0$, and likewise, when $a\ne0$ we get $\<\bar{\nabla}_{v} e_n,x\>=0$.   

Let us continue proving (iii). From Theorem \ref{equations} we have that

$$ad+bd\, \lambda\mu=e_n(\frac{e_n(\lambda)}{\lambda-\mu})-(\frac{e_n(\lambda)}{\lambda-\mu})^2$$

Using part (i) we get that

$$ ad+bd\, \lambda\mu=e_n(\frac{e_n(\lambda)}{\lambda-\mu})-(\frac{e_n(\lambda)}{\lambda-\mu})^2=
e_n(\frac{-e_n(w)}{w})-(-\frac{e_n(w)}{w})^2=-\frac{e_n(e_n(w))}{w}$$

Then (iii) follows. In order to prove (iv) we notice that by part (i) we have that $e_n(\lambda w)=e_n(w)\, \mu$, therefore

$$e_n(\lambda^2\, w^2)=2\lambda w e_n(\lambda w)=2\lambda w\mu e_n(w)$$

Using this last equation along with part  (iii)  we get

$$ e_n\big{(}\,  a d w^2+ b d \lambda^2 w^2 + (e_n(w))^2\, \big{)}=2 e_n(w) \big{(}\, a d w +b d w  \lambda \mu + e_n(e_n(w))\, \big{)}=2e_n(w) \, 0=0
$$

then (iv) follows. Recall that we already have shown (Theorem \ref{equations}) that $v(e_n(\lambda))=v(\lambda)=0$ for all $v\in \Gamma$. The first two equations in (v) follow by the definition of the vector field $x$ and the definition of shape operator and principal curvature. Since  $\nabla_{e_n}e_n$ vanishes (Theorem \ref{equations}), $\bar{\nabla}_{e_n}e_n$ is a linear combination of $\nu$ and $x$ when $a\ne0$ and, it is a multiple of $\nu$ when $a=0$. We can find these coefficients in the expression of $\bar{\nabla}_{e_n}e_n$ in terms of $\nu$ and $x$ by taking the derivative of the equality $\<e_n,\nu\>=0$ in the direction $e_n$. We get that $\<\bar{\nabla}_{e_n}e_n,\nu\>+\<e_n,\bar{\nabla}_{e_n}\nu\>=0$ or equivalently
$ \<\bar{\nabla}_{e_n}e_n,\nu\>-\mu \, \<e_n,e_n\>=0$ that is $ \<\bar{\nabla}_{e_n}e_n,\nu\>=\mu d$. Likewise, when $a\ne 0$, we have that $\<\bar{\nabla}_{e_n}e_n,x\>=-d$. These equations imply (v).   Let us prove (vi). We have that for any $v\in \Gamma$,

$$\bar{\nabla}_{v}X=\bar{\nabla}_{v}(\nu+\lambda x)=-\lambda v+v(\lambda)\, x + \,  \lambda\, v=-\lambda v+{\bf 0}+ \lambda\, v={\bf 0}$$

and using part (ii)

$$\bar{\nabla}_{v}Y=\bar{\nabla}_{v}(w e_n-e_n(w) x)=w \bar{\nabla}_{v}e_n -e_n(w)\, v = w\frac{e_n(w)}{w} v-e_n(w) v={\bf 0}$$

Finally we have that,

$$\bar{\nabla}_{v}Z=\bar{\nabla}_{v}(\lambda w e_n+e_n(w) \nu)=\lambda w \bar{\nabla}_{v} e_n+ e_n(w)\bar{\nabla}_v \nu = \lambda w\frac{e_n(w)}{w} v-\lambda e_n(w) v={\bf 0}$$

This completes the proof of part (vi).

Part (vii) is a direct computation using the previous parts.

\begin{eqnarray*}
\bar{\nabla}_{e_n}\eta&=& -e_n(e_n(w))e_n-e_n(w)\bar{\nabla}_{e_n}e_n+bde_n(\lambda w) \, \nu+bd\lambda w \bar{\nabla}_{e_n}\nu-ade_n(w) x-ad we_n\\
&=&d(aw+bw\lambda \mu)\, e_n-e_n(w)d(b\mu\nu-ax)+bd\mu e_n(w)\nu-bd\lambda \mu w e_n-ade_n(w) x-ad we_n\\
&=&0
\end{eqnarray*}

In order to prove (viii) we will use the fact that for any $v\in\Gamma$,
$v(e_n(\lambda))$ vanishes. We have

\begin{eqnarray*}
\bar{\nabla}_{v}\eta&=&-e_n(w)\bar{\nabla}_{v}e_n + db\lambda w\, \bar{\nabla}\nu-daw\, v\\
             &=& -e_n(w)\, \frac{e_n(w)}{w}\, v - db\lambda w \lambda \, v -daw\, v\\
             &=& -\frac{C}{w}\, v
\end{eqnarray*}

A direct computation shows that

$$-w\eta+d b \lambda w^2 X-e_n(w)\, Y=(e_n(w)^2+db\lambda^2 w^2+daw^2) x=Cx$$

Then (ix) follows.  Here is the proof of part (x).

\begin{eqnarray*}
\bar{\nabla}_{e_n} Y& =& \bar{\nabla}_{e_n} (w\, e_n-e_n(w)\, x) \cr
                    & =& e_n(w) e_n+ w \bar{\nabla}_{e_n} e_n - e_n(e_n(w))\, x-e_n(w)\, e_n \cr
                    & =& w\,  (bd\mu\, \nu-ad\, x) + (adw+bdw\lambda\mu)\, x \cr
                    & =& wbd\mu\, \nu +  bdw\lambda\mu\, x \cr
                    & =& bdw\mu ( \nu +  \lambda\, x )=bdw\mu X
\end{eqnarray*}

Here is the proof of (xi), it uses part (i)

\begin{eqnarray*}
\bar{\nabla}_{e_n} X& =& \bar{\nabla}_{e_n} (\nu+\lambda\, x) \cr
                    & =& -\mu\, e_n+e_n(\lambda) x+\lambda\, e_n \cr
                    & =& (\lambda-\mu) e_n-w^{-1} (\lambda-\mu) e_n(w)\, x\cr
                    & =& \frac{\lambda-\mu}{w} \, Y
\end{eqnarray*}

Here is the proof of (xii), it uses part (i), (iii) and (vi)

\begin{eqnarray*}
\bar{\nabla}_{e_n} Z& =& \bar{\nabla}_{e_n} (\lambda w\, e_n+e_n(w)\, \nu) \cr
                    & =& e_n(\lambda w)\, e_n +\lambda w \bar{\nabla}_{e_n} e_n +e_n(e_n(w))\, \nu-\mu e_n(w)\, e_n \cr
                    & =& \mu e_n(w) e_n+ \lambda w (bd\mu\, \nu-ad\, x) + (-adw-bdw\lambda\mu)\, \nu-\mu e_n(w)\, e_n    \cr
                    & =& -adw\, \nu-ad\lambda w\, x \cr
                    & =& -adw ( \nu +  \lambda\, x )=-adw\, X
\end{eqnarray*}

Part (xiii) is a direct computation that uses part (iv). Part  (xiv) follows form parts (vi) and (xii). 
\end{proof}

\begin{mydef}\label{def of C}
We will refer to the constant $C$ as the constant associated with the cmc-hypersurface with two principal curvatures $M$.
\end{mydef}

%MOVE THE PART TO AN APPENDIX

\begin{lem}\label{the plane--lemma}
Let $M\subset SF(n+1,k,a)$ be a cmc hypersurface with two principal curvatures, one of them with multiplicity one. Let us  assume that $a=\pm1$. For any $x$ in $M$ let us define,

$$\Pi(x)=\{v\in \hbox{Span}\{x,e_n,\nu\}\, :\, \<v,\eta\>=0\}=\hbox{span}\{X,Y\}$$

The plane $\Pi(x)$ does not depend on $x$. That is, for all $x\in M$

$$\Pi(x)=\Pi$$

where $\Pi$ is fixed.  Moreover, if $C\ne0$, the inner product $\<\, ,\, \>$ on $SF(k,n+1,a)$ restricted to $\Pi$ is not degenerated and for every $x\in M$, $\eta(x)$ is perpendicular to the plane $\Pi$. 

\end{lem}
\begin{proof}
It is not difficult to show that the vectors $X$ and $Y$ are linearly independent.  The plane $\Pi(x)$ is independent of $x$ 
because by Theorem \ref{local relations} part (vi), we have that the vector fields $X$ and $Y$ are constant on the integral submanifolds of $\Gamma$. Also, if we denote by $u_1=X(x_0)$ and $u_2=Y(x_0)$, then, we conclude that the span of the vectors $X$ and $Y$ is constant along the integral curve $\gamma(t)$ of the vector field $e_n$ satisfying $\gamma(0)=x_0$ because by the existence and uniqueness theorem of differential equations, the solution $X$ and $Y$ of the system  $\bar{\nabla}_{e_n}X=\frac{\lambda-\mu}{w} Y $ and $\bar{\nabla}_{e_n}Y=db w\mu X $ are linear combination of the  vectors $u_1$ and $u_2$.
\end{proof}

\begin{mydef}\label{the plane--def}
When  $a\ne0$, we will referring to the plane $\Pi$ in Lemma \ref{the plane--lemma} as the plane associated to the hypersurface $M$.
\end{mydef}

\begin{mydef}\label{def of zeta}
If the vector fields $x,\eta$, the function $w$ and the constant $C$ are defined as before, and $C\ne0$, then we define the vector field $\zeta$ as
$$\zeta=x+\frac{w}{C}\eta$$
\end{mydef}

\begin{lem}\label{zeta}
The vector field $\zeta$ defined in Definition \ref{def of zeta} satisfies that

\begin{eqnarray*}
& \bar{\nabla}_v\zeta \com{vanishes for all $v\in \Gamma$}\\
&\bar{\nabla}_{e_n}\zeta=e_n+\frac{e_n(w)}{C}\, \eta \\
&\zeta(x) \in \Pi \com{for all $x$ in $M$}
\end{eqnarray*}
\end{lem}

\begin{proof}
The first two parts follow from the fact $\bar{\nabla}_{e_n}\eta$ vanishes and the equation   $\bar{\nabla}_v\eta=-\frac{C}{w}\, v$ for any $v\in \Gamma$. The last part follows from the fact that 

$$\zeta = -\frac{1}{C}e_n(w)\, Y+ d b\lambda\frac{w^2}{C}\, X$$
\end{proof}

%Aqui comienza la seccion The immersions.
%
%
%
%EXPLICIT IMMERSION

\section{Explicit examples of immersions}\label{explicitEx}

In this section we define explicit immersions of cmc in $\sf$ in terms of a solution of a system of ordinary differential equations. Maybe the most important differential equation in the study of cmc hypersurfaces with two principal curvatures is the following

\begin{eqnarray}\label{theode}
(g^\prime(t))^2= C-d g(t)^2 (a + b (H+g(t)^{-n})^2)
\end{eqnarray}

Using separation of variables we have that a local solution near zero with $g(0)=v_0>0$ exists, if and only if, $C-d v_0^2 (a + b (H+v_0^{-n})^2)>0$. We will only consider positive solutions of the differential equation. Given a solution $g(t)$ of the differential equation (\ref{theode}) we define

\begin{eqnarray}\label{kappas}
\kappa_1(t)= H+g(t)^{-n}\, \text{and } \, \kappa_2(t)=H-(n-1)g(t)^{-n}
\end{eqnarray}

We have the following identities 

\begin{lem}\label{lk1k2} If $g(t)$ is a positive solution of (\ref{theode}) and $\kappa_1$ and $\kappa_2$ are defined as in (\ref{kappas}) then, the following relations are true

\begin{enumerate}
\item
$g^{\prime\prime}(t)=-dg(t)\left(a+b\kappa_1(t)\kappa_2(t) \right)$

\item
$(g(t)\kappa_1(t))^\prime=\kappa_2(t) g^\prime(t)$
\end{enumerate}

\end{lem}

\subsection{Immersion with $C=0$ and $a\ne 0$} In order to have a local immersion with cmc $H$, $\<e_n,e_n\>=d$ and $\<\nu,\nu\>=b$, we only need to have a positive solution of the differential equation \ref{theode}.

\begin{thm}\label{ceq0aneq0} Let us assume that  for some $a,b,d$  and $H$ with $|a|=|b|=|d|=1$, the function $g(t)>0$ is defined near $t=0$ and satisfies the differential equation  (\ref{theode}) with $C=0$. Let us  also assume that we can select orthogonal vectors $u_1$, $u_2$, $u_3$ in $\lf$ with $\<u_1,u_1\>=a$, $\<u_2,u_2\>=b$ and $\<u_3,u_3\>=d$. Let us consider $S=\{y\in \lf : \<y,u_1\>=\<y,u_2\>=\<y,u_3\>=0\}$
and $\alpha(t),\beta(t)\in \lf$ a solution of the system of ordinary differential equations

\begin{eqnarray*}
& &\alpha^{\prime\prime}(t)=bd\kappa_1(t)\beta(t)-ad\alpha(t)\, \quad \beta^\prime(t)=-\kappa_2(t)\alpha^\prime(t)\\
& & \alpha(0)=u_1,\quad \beta(0)=u_2,\quad \alpha^\prime(0)=u_3
\end{eqnarray*}

The immersion

$$\phi(y,t)=y+\alpha(t)+\frac{d}{2g(t)}\, \<y,y\> \rho_0$$

with 

$$\rho_0= -g^\prime(0) u_3+bd g(0)\kappa_1(0)u_2-da g(0) u_1$$

has  constant mean curvature $H$ and principal curvatures, $\kappa_1(t)$ with multiplicity $n-1$, and $\kappa_2(t)$ with multiplicity one. \end{thm}

\begin{proof}
 We can show that  the functions $\gamma_{11}=\<\alpha(t),\alpha(t)\>$, $\gamma_{12}=\<\alpha(t),\beta(t)\>$, $\gamma_{13}=\<\alpha(t),\alpha^\prime(t)\>$, $\gamma_{22}=\<\beta(t),\beta(t)\>$, $\gamma_{23}=\<\beta(t),\alpha^\prime(t)\>$ and $\gamma_{33}=\<\alpha^\prime(t),\alpha^\prime(t)\>$ satisfy a first order system of ordinary differential equations with $\gamma_{11}(t)=a$, $\gamma_{12}(t)=0$, $\gamma_{13}(t)=0$, $\gamma_{22}(t)=b$, $\gamma_{23}(t)=0$ and $\gamma_{33}(t)=d$ an equilibrium solution. Therefore, $\alpha(t)$, $\beta(t)$, $\alpha^\prime(t)$ are perpendicular and   $\<\alpha(t),\alpha(t)\>=a$, $\<\beta(t),\beta(t)\>=b$ and $\<\alpha^\prime(t),\alpha^\prime(t)\>=d$. Using Lemma \ref{lk1k2} we can check that if
 
  $$\rho(t)= -g^\prime(t) \alpha^\prime(t)+bd g(t)\kappa_1(t)\beta(t)-da g(t) \alpha(t)$$
  
  then, $\rho^\prime(t)=0$ and $\rho(t)=\rho(0)=\rho_0$. Therefore,

\begin{eqnarray}\label{inner products}
\<\rho_0,\alpha(t)\>=\<\rho(t),\alpha(t)\> = - d\, g(t), \quad 
 \<\rho_0,\beta(t)\>=d\,\kappa_1(t)\,  g(t),\quad \<\rho_0,\alpha\, ^\prime(t)\>=-d\, g^\prime(t)\, .
\end{eqnarray}

We have that 

$$\phi_t=-\frac{d g^\prime(t)}{2 g(t)^2} \, \<y,y\>\, \rho_0+\alpha\, ^\prime(t) $$

and, if we define the coordinates $y_1,\dots, y_{n-1}$ on $S$ by the equation $y=\sum_{i=1}^{n-1}y_iE_i$ where $E_1,\dots, E_{n-1}$ form an orthogonal basis of $S$, then,  for $i=1,\dots, n-1$ we have 

$$\frac{\partial \phi}{\partial y_i}(y,t)= \frac{d}{g(t)}\, \<y,E_i\>\, \rho_0+ E_i $$

Let us  show that  $\phi$ is an immersion.  By the uniqueness part of the existence and uniqueness theorem of ordinary differential equations, we have that
 $\alpha(t),\beta(t)$ are in the Span$\{u_1,u_2,u_3\}$. Therefore, $\alpha(t),\alpha^\prime(t), \beta(t)$ and
$\rho_0$ are perpendicular to all $E_i$'s. Since we are assuming that $C=0$ we have that $\<\rho_0,\rho_0\>=0$ and therefore, $\<\frac{\partial \phi}{\partial y_i},\frac{\partial \phi}{\partial y_j}\>=\<E_i,E_j\>$, and the vectors $\frac{\partial \phi}{\partial y_j}$ are perpendicular. In order to show that $\phi$ is an immersion, we notice that the vector $\frac{\partial \phi}{\partial t}$ cannot be written as a linear combination of the vectors $\frac{\partial \phi}{\partial y_i}$. If we define

$$\xi=\frac{\partial \phi}{\partial t}+\frac{g^\prime}{g}\, \sum_{i=1}^{n-1}y_i\, \frac{\partial \phi}{\partial y_i}\, =\, \frac{dg^\prime\<y,y\>}{2g^2}\, \rho_0+\alpha\, ^\prime+\frac{g^\prime}{g}\, y$$ 
 
%%%

then, a direct computation using the fact that $\<\rho_0,\alpha\, ^\prime\>=-dg^\prime$, give us that $\<\xi,\xi\>=d$ and $\<\xi,\frac{\partial \phi}{\partial y_i}\>=0$ for all $i=1,\dots,n-1$. Since 

$$\text{Span}\{\xi,\frac{\partial \phi}{\partial y_1},\dots, \frac{\partial \phi}{\partial y_{n-1}}\}=
\text{Span}\{\frac{\partial \phi}{\partial t},\frac{\partial \phi}{\partial y_1},\dots, \frac{\partial \phi}{\partial y_{n-1}}\}$$

and the vectors $\{\xi,\frac{\partial \phi}{\partial y_1},\dots, \frac{\partial \phi}{\partial y_{n-1}}\}$ are linearly independent (they are perpendicular to each other), then $\{\frac{\partial \phi}{\partial t},\frac{\partial \phi}{\partial y_1},\dots, \frac{\partial \phi}{\partial y_{n-1}}\}$ are linearly independent. This shows that $\phi$ in an immersion.

A direct computation using equations \ref{inner products}  shows that

$$\<\phi,\alpha\>=a-\frac{\<y,y\>}{2},\quad \<\phi,\alpha\, ^\prime\>=-\frac{\<y,y\>g^\prime}{2g}\, ,\quad \<\phi,\beta\>=\frac{\<y,y\>}{2}\, \kappa_1 $$

Therefore, we can check that if

$$\nu=\kappa_1(t)\,  \alpha(t)+\beta(t)-\kappa_1(t)\, \phi(y,t)=\beta+\frac{d\kappa_1(t)}{2g} \<y,y\>\, \rho_0-\kappa_1(t) y$$

then, $\<\nu,\nu\>=b$, $\<\nu,\phi\>=0$, $\<\nu,\frac{\partial \phi}{\partial t}\>=0$ and $\<\nu,\frac{\partial \phi}{\partial y_i}\>=0$ for $i=1,\dots, n-1$. Therefore, $\nu$ is a Gauss map of the immersion $\phi$. By the definition of $\nu$ it easily follows that  $\frac{\partial \nu}{\partial y_i} =-\kappa_1 \frac{\partial \phi}{\partial y_i}$, therefore $\kappa_1$ is a principal curvature of $\phi$ with multiplicity $n-1$. On the other hand we have that if $d\nu$ denotes the differential of $\nu$, then

\begin{eqnarray*}
d\nu(\xi) &=&   \frac{\partial \nu}{\partial t}+\frac{g^\prime}{g}\, \sum_{i=1}^{n-1}y_i\, \frac{\partial \nu}{\partial y_i}\\
&=& \frac{\partial\left( \beta-\frac{d\kappa_1\<y,y\>}{2g}\, \rho_0-\kappa_1 y \right)  }{\partial t}-\frac{\kappa_1g^\prime}{g}\,  \sum_{i=1}^{n-1}y_i\, \frac{\partial \phi}{\partial y_i}\\
&=&-\kappa_2 \, \alpha\, ^\prime - \frac{d}{2}\, \<y,y\> (\frac{\kappa_1^\prime g-g^\prime \kappa_1}{g^2})\, \rho_0   -  \kappa_1^\prime y-\frac{\kappa_1g^\prime}{g}\,  \left(\frac{d}{g}\<y,y\>\rho_0 + y \right)\\
&=&-\kappa_2 \, \alpha\, ^\prime - \frac{d}{2}\, \<y,y\> (\frac{\kappa_2 g^\prime-2g^\prime \kappa_1}{g^2})\, \rho_0 - \kappa_1^\prime y-\frac{\kappa_1g^\prime}{g}\,  \left(\frac{d}{g}\<y,y\>\rho_0 + y \right)\\
&=&-\kappa_2 \, \alpha\, ^\prime - \frac{d\kappa_2 g^\prime}{2 g^2}\, \<y,y\> \, \rho_0 - (\frac{\kappa_1^\prime g+\kappa_1g^\prime}{g})\,   y \\
&=&-\kappa_2 \, \alpha\, ^\prime - \frac{d\kappa_2 g^\prime}{2 g^2}\, \<y,y\> \, \rho_0 - \frac{\kappa_2g^\prime}{g}\,   y \\
&=& -\kappa_2\, \xi
\end{eqnarray*}

Therefore, $\kappa_2$ is a principal curvature with multiplicity 1. Since $(n-1) \kappa_1+\kappa_2=n H$ we conclude that $\phi$ has constant mean curvature and the theorem follows.

\end{proof}

%%
%%Case a=0, C=0
%%

\subsection{Immersion with $C=0$ and $a= 0$} The  case $C=0$ and $a=0$ is almost exactly the same. The only difference is that since $\sf$ is a linear space in this case, then we do not need to use the the space $\lf$ and without loss of generality we can use $u_1={\bf 0}=(0,\dots , 0)$.

\begin{thm} \label{ceq0aeq0} Let us assume that $a=0$ and for some $b,d$  and $H$ with $|b|=|d|=1$, the function $g(t)>0$ is defined near $t=0$ and satisfies the differential equation  (\ref{theode}) with $C=0$. Let us  also assume that we can select orthogonal vectors  $u_2$, $u_3$ in $\sf$ with  $\<u_2,u_2\>=b$ and $\<u_3,u_3\>=d$. Let us consider $S=\{y\in \sf :  \<y,u_2\>=\<y,u_3\>=0\}$
and $\alpha(t),\beta(t)\in \sf$ a solution of the system of ordinary differential equations

\begin{eqnarray*}
& &\alpha^{\prime\prime}(t)=bd\kappa_1(t)\beta(t)\, \quad \beta^\prime(t)=-\kappa_2(t)\alpha^\prime(t)\\
& & \alpha(0)={\bf 0},\quad \beta(0)=u_2,\quad \alpha^\prime(0)=u_3
\end{eqnarray*}

The immersion

$$\phi(y,t)=y+\alpha(t)+\frac{d}{2g(t)}\, \<y,y\> \rho_0$$

with 

$$\rho_0= -g^\prime(0) u_3+bd g(0)\kappa_1(0)u_2$$

has constant mean curvature $H$ and
 principal curvatures $\kappa_1(t)$ with multiplicity $n-1$ and $\kappa_2(t)$ with multiplicity $1$.
 
 \end{thm}

%proof case a=0, C=0
\begin{proof}
 We can show that  the functions $\gamma_{22}=\<\beta(t),\beta(t)\>$, $\gamma_{23}=\<\beta(t),\alpha^\prime(t)\>$ and $\gamma_{33}=\<\alpha^\prime(t),\alpha^\prime(t)\>$ satisfy a first order system of ordinary differential equations with $\gamma_{22}(t)=b$, $\gamma_{23}(t)=0$ and $\gamma_{33}(t)=d$ an equilibrium solution. Therefore, $\beta(t)$, $\alpha^\prime(t)$ are perpendicular and   $\<\beta(t),\beta(t)\>=b$ and $\<\alpha^\prime(t),\alpha^\prime(t)\>=d$. A direct computation shows that if 
 
 $$\rho(t)= -g^\prime(t) \alpha^\prime(t)+bd g(t)\kappa_1(t)\beta(t)$$
 
 then $\rho^\prime(t)$ vanishes and $\rho(t)=\rho_0$. Therefore,

\begin{eqnarray}\label{inner products}
 \<\rho_0,\beta(t)\>= \<\rho(t),\beta(t)\>=d\,\kappa_1(t)\,  g(t),\quad \<\rho_0,\alpha\, ^\prime(t)\>=\<\rho(t),\alpha\, ^\prime(t)\>=-d\, g^\prime(t)
\end{eqnarray}

We have that 

$$\phi_t=-\frac{d g^\prime(t)}{2 g(t)^2} \, \<y,y\>\, \rho_0+\alpha\, ^\prime(t) $$

and, if we define the coordinates $y_1,\dots, y_{n-1}$ on $S$ by the equation $y=\sum_{i=1}^{n-1}y_iE_i$ where $E_1,\dots, E_{n-1}$ form an orthogonal basis of $S$, then,  for $i=1,\dots, n-1$ we have 

$$\frac{\partial \phi}{\partial y_i}(y,t)= \frac{d}{g(t)}\, \<y,E_i\>\, \rho_0+ E_i $$

As in the proof of the previous theorem, we can show that   $\phi$ is an immersion.  Also we can show that if 

$$\xi=\frac{\partial \phi}{\partial t}+\frac{g^\prime}{g}\, \sum_{i=1}^{n-1}y_i\, \frac{\partial \phi}{\partial y_i}\, =\, \frac{dg^\prime\<y,y\>}{2g^2}\, \rho_0+\alpha\, ^\prime+\frac{g^\prime}{g}\, y$$ 
 
 and
 
$$\nu=\kappa_1(t)\,  \alpha(t)+\beta(t)-\kappa_1(t)\, \phi(y,t)=\beta+\frac{d\kappa_1(t)}{2g} \<y,y\>\, \rho_0-\kappa_1(t) y$$

then, $\<\nu,\nu\>=b$, $\<\nu,\phi\>=0$, $\<\nu,\frac{\partial \phi}{\partial t}\>=0$ and $\<\nu,\frac{\partial \phi}{\partial y_i}\>=0$ for $i=1,\dots, n-1$. Therefore, $\nu$ is a Gauss map of the immersion $\phi$. By the definition of $\nu$ it easily follows that  $\frac{\partial \nu}{\partial y_i} =-\kappa_1 \frac{\partial \phi}{\partial y_i}$, therefore $\kappa_1$ is a principal curvature of $\phi$ with multiplicity $n-1$. On the other hand arguing in the same way as  in the previous theorem we can show that $d\nu(\xi)=-\kappa_2\, \xi$. Therefore, $\kappa_2$ is a principal curvature with multiplicity 1. Since $(n-1) \kappa_1+\kappa_2=n H$ we conclude that $\phi$ has constant mean curvature and the theorem follows.

\end{proof}

%
% case C\ne 0 and a\ne 0
%

\subsection{Immersion with $C\ne0$ and $a\ne 0$} This is the theorem when $C\ne 0$ and $a\ne0$.

\begin{thm}\label{cneq0aneq0} Let us assume that  for some $a,b,d$  and $H$ with $|a|=|b|=|d|=1$, the function $g(t)>0$ is defined near $t=0$ and satisfies the differential equation  (\ref{theode}) with $C\ne0$. Let us  also assume that we can select orthogonal vectors $u_1$, $u_2$, $u_3$ in $\lf$ with $\<u_1,u_1\>=a$, $\<u_2,u_2\>=b$ and $\<u_3,u_3\>=d$. Let us consider $S= \{y=y_1E_1+\dots y_{n} E_{n}: \<y,y\> = d\frac{|C|}{C}\}$ where 

$$E_n=\rho_0= -g^\prime(0) u_3+bd g(0)\kappa_1(0)u_2-da g(0) u_1$$

and $\{E_1\dots E_{n-1}\}$ form an orthogonal basis of $\{u:\<u,u_1\>=\<u,u_2\>=\<u,u_3\>=0\}$,
and let us consider $\alpha(t),\beta(t)\in \lf$

a solution of the system of ordinary differential equations

\begin{eqnarray*}
& &\alpha^{\prime\prime}(t)=bd\kappa_1(t)\beta(t)-ad\alpha(t)\, \quad \beta^\prime(t)=-\kappa_2(t)\alpha^\prime(t)\\
& & \alpha(0)=u_1,\quad \beta(0)=u_2,\quad \alpha^\prime(0)=u_3
\end{eqnarray*}

The immersion $\phi:S\times I\longrightarrow \sf$ given by

$$\phi(y,t)= \frac{g(t)}{\sqrt{|C|}}\, y+ \alpha(t)+\frac{g(t)}{C}\, \rho_0 $$

has constant mean curvature $H$ and
 principal curvatures $\kappa_1(t)$ with multiplicity $n-1$ and $\kappa_2(t)$ with multiplicity $1$.
 
\end{thm}

\begin{proof} 
A direct computation shows that $\<\rho_0,\rho_0\>=d C$. Using an argument as in Theorem (\ref{ceq0aneq0}), we can show that $\alpha(t)$, $\alpha\, ^\prime(t)$,  and $\beta(t)$ are perpendicular and $\<\alpha(t),\alpha(t)\>=a$, $\<\alpha^\prime(t),\alpha^\prime(t)\>=d$ and $\<\beta(t),\beta(t)\>=b$. We have that 

$$\phi_t=\frac{g^\prime(t)}{\sqrt{|C|}} \, y+\alpha^\prime(t)+\frac{g^\prime(t)}{C}\, \rho_0$$

The theorem will follow by showing first, that

$$\nu=\frac{-\kappa_1 g}{\sqrt{|C|}}\, y+ \beta-\frac{\kappa_1 g}{C}\,  \rho_0$$

is a Gauss map of the immersion $\phi$, second, that $\nu_t+\kappa_2 \phi_t$ vanishes and third, that for every $v\in T_yS$ we have that 
$d\nu(v)+\kappa_1 v$ also vanishes where $d\nu$ denotes the differential of the map $\nu$. Therefore we would  have that $\kappa_1$ is a principal curvature of $\phi$ with multiplicity $n-1$ and $\kappa_2$ is a 
principal curvature of $\phi$ with multiplicity $1$.

Let us work all the details. As in Theorem \ref{ceq0aneq0} we have that for all $t$,

$$\rho(t)=-g^\prime(t) \, \alpha\, ^\prime(t)+b d \kappa_1(t) g(t)\, \beta(t)- d a g(t)\, \alpha(t)=\rho_0$$

Therefore, we  have that

$$\<\rho_0,\alpha(t)\>=\<\rho(t),\alpha(t)\>=-d\, g(t)\com{and} \<\rho_0,\alpha\, ^\prime(t)\>=\<\rho(t),\alpha\, ^\prime(t)\>=-d\, g^\prime(t)$$

Let us define $\Pi=\{v\in \hbox{Span}(u_1,u_2,u_3)\, : \, \<v,\rho_0\>=0\}$ and $\xi(t)= \alpha(t)+\frac{g(t)}{C}\, \rho_0$. Since $\alpha(t)$, $\alpha^\prime(t)$ and $\beta(t)$ are contained in $Span(u_1,u_2,u_3)$ then $\xi(t)$ is contain in  $Span(u_1,u_2,u_3)$. Since $\<\xi(t),\rho_0\>=0$ then, 
$\xi(t)$ is contained in $\Pi$ and moreover, this plane is perpendicular to the 
$n$-dimensional space spanned by $\{E_1,\dots,E_n\}$.
Since $\phi= \frac{g(t)}{\sqrt{|C|}}\, y+\xi(t)$ then 

$$\<\phi,\phi\>=\<\frac{g(t)}{\sqrt{|C|}}\, y,\frac{g(t)}{\sqrt{|C|}}\, y\>+\<\xi,\xi\>=\frac{g^2(t)}{|C|}d \frac{|C|}{C}+a+\frac{g^2(t)}{C^2}d C-2\frac{g(t)}{C} \, d\, g(t) \ =a$$

 Therefore, we indeed have that $\phi\in \sf$. Since $\xi(t)\in \Pi$ then $\xi^\prime(t)\in\Pi$. A similar computation shows that $\<\phi_t,\phi_t\>=d$. For any $v\in T_yS$ we have that  $d\phi(v)=\frac{g(t)}{\sqrt{|C|}}\, v$, where $d\phi(v)$ is  the differential  of $\phi$ in the direction of $v$. Clearly $\<d\phi(v),\phi_t\>=0$. Recall if $v\in T_yS$, then $\<y,v\>=0$. Therefore, $\phi$ is an immersion. Notice that $\nu=\frac{-\kappa_1 g}{\sqrt{|C|}}\, y+ \varphi(t)$ where 
 $\varphi= \beta-\frac{\kappa_1 g}{C}\,  \rho_0$. A direct computation shows that 
 
 $$\<\varphi(t),\rho_0\>=\<\beta(t)-\frac{\kappa_1(t) g(t)}{C}\, \rho(t) ,\rho(t)\>=d\, \kappa_1(t)\, g(t)- \frac{\kappa_1(t) g(t)}{C}\, d\, C =0$$
 
 Therefore $\varphi(t)$ is the plane $\Pi$.  In the same way we show that $\<\phi,\phi\>=a$, we can show that $\<\nu,\nu\>=b$. A direct computation shows that 
$\<\varphi,\xi\>=\frac{\kappa_1\, d\, g^2}{C}$ and $\<\varphi,\xi^\prime\>=\frac{\kappa_1\, d\, g\, g^\prime}{C}$. Since we know that $\varphi$ and $\xi$ are in $\Pi$ and $S$ is the orthogonal complement of $\Pi$, then it is not difficult to check that $\<\nu,\phi\>=0$ and $\<\nu, \phi_t\>=0$. By noticing that 
for any $v\in T_yS$, we have that  $\<\nu,v\>=0$, we conclude that $\nu$ is a Gauss map of the immersion $\phi$. We can easily see that for any $v\in T_yS$ we have that

$$d\nu(v)=-\frac{\kappa_1 g}{\sqrt{|C|}}\, v=-\kappa_1d\phi(v)$$

Moreover, Lemma \ref{lk1k2} give us that $(\kappa_1 g)^\prime=\kappa_2 g^\prime$, therefore,

$$ \nu_t=-\frac{\kappa_2\, g^\prime}{\sqrt{|C|}} \, y+\beta^\prime-\frac{\kappa_2 g}{C}\,  \rho_0=-\frac{\kappa_2\, g^\prime}{\sqrt{|C|}} \, y-\kappa_2\alpha\, ^\prime-\frac{\kappa_2 g}{C}\,  \rho_0=-\kappa_2\, \phi_t$$

The previous computations shows that  $\kappa_1$ is a principal curvature with multiplicity $n-1$ and $\kappa_2$ is a principal curvature with multiplicity 1. Since $(n-1) \kappa_1+\kappa_2=n H$ we conclude that $\phi$ has constant mean curvature and the theorem follows.

\end{proof}

%
%
%Case C\ne0 and a=0
%
%

\subsection{Immersion with $C\ne0$ and $a=0$} 
The  case $C\ne0$ and $a=0$ is almost exactly the same as the case $C\ne0$ and $a\ne0$. The only difference is that since $\sf$ is a linear space in this case, then we do not need to use the space $\lf$ and without loss of generality we can use $u_1={\bf 0}=(0,\dots , 0)$.

\begin{thm}\label{cneq0aeq0} Let us assume that  $a=0$ and for some  $b,d$  and $H$ with $|b|=|d|=1$, the function $g(t)>0$ is defined near $t=0$ and satisfies the differential equation  (\ref{theode}) with $C\ne0$. Let us  also assume that we can select orthogonal vectors  $u_2$, $u_3$ in $\sf$ with  $\<u_2,u_2\>=b$ and $\<u_3,u_3\>=d$. Let us consider $S= \{y=y_1E_1+\dots y_{n} E_{n}: \<y,y\> = d\frac{|C|}{C}\}$ where 

$$E_n=\rho_0= -g^\prime(0) u_3+bd g(0)\kappa_1(0)u_2-da g(0) u_1$$

and $\{E_1\dots E_{n-1}\}$ form an orthogonal basis of $\{u:\<u,u_2\>=\<u,u_3\>=0\}$,
and let us consider $\alpha(t),\beta(t)\in \lf$  a solution of the system of ordinary differential equations

\begin{eqnarray*}
& &\alpha^{\prime\prime}(t)=bd\kappa_1(t)\beta(t) \quad \beta^\prime(t)=-\kappa_2(t)\alpha^\prime(t)\\
& & \alpha(0)={\bf 0},\quad \beta(0)=u_3,\quad \alpha^\prime(0)=u_2
\end{eqnarray*}

The immersion $\phi:S\times I\longrightarrow \sf$ given by

$$\phi(y,t)= \frac{g(t)}{\sqrt{|C|}}\, y+ \alpha(t)+\frac{g(t)}{C}\, \rho_0 $$

has constant mean curvature $H$ and
 principal curvatures $\kappa_1(t)$ with multiplicity $n-1$ and $\kappa_2(t)$ with multiplicity $1$.
 
\end{thm}

\begin{proof} The proof follows the lines of those in the proof of Theorem \ref{cneq0aneq0}. We can show that:  (i)

$$\nu=\frac{-\kappa_1 g}{\sqrt{|C|}}\, y+ \beta-\frac{\kappa_1 g}{C}\,  \rho_0$$

 is a Gauss map of the immersion and $\<\nu,\nu\>=b$ (ii)  $d\nu(\phi_t)=-\kappa_2(t) \phi_t$ and (iii) for any $v\in T_yS$, $d\nu(d\phi_{(y,t)}(v))=-\kappa_1d\phi_{(y,t)}(v)$.

\end{proof}

\subsection{Examples with constant principal curvatures}

\begin{thm} Let $f(v)=C-dv^2\left(a+b(H+v^{-n})^2\right)$. If $v_0>0$ satisfies that $f(v_0)=f^\prime(v_0)=0$, then the immersions given in Theorems \ref{ceq0aeq0}, \ref{ceq0aeq0}, \ref{ceq0aneq0} and \ref{cneq0aneq0} obtained by replacing $g(t)$ with $v_0$ define immersions with constant principal curvature $\kappa_1=H+v_0^{-n}$ with multiplicity $n-1$ and $\kappa_2=H-(n-1)v_0^{-n}$ with multiplicity 1.
\end{thm}

\begin{proof} We have that the equations for  $g(t)$:

\begin{eqnarray*}
(g^\prime(t))^2&=&C-dv^2\left(a+b(H+g(t)^{-n})^2\right) \\
g^{\prime\prime}(t)&=&-ad g(t)-bd g(t)\left(H+g^{-n}(t)\right) \left(H-(n-1)g^{-n}(t)\right)
\end{eqnarray*}

hold true for $g(t)=v_0$  because the first one can be written as $(g^\prime(t))^2=f(g(t))$  and the second one can be written as $g^{\prime\prime}(t)=\frac{1}{2}\, f^\prime(g(t))$.  Since the equations $(g^\prime(t))^2=f(g(t))$ and $g^{\prime\prime}(t)=\frac{1}{2}\, f^\prime(g(t))$ are the only equations for $g$ used the  in the proofs of Theorems \ref{ceq0aeq0},\ref{ceq0aeq0}, \ref{ceq0aneq0},\ref{cneq0aneq0}, then this theorem follows.
\end{proof}

%%%%%
%%%%%
%Aqui comienza la seccion The immersions.
%%%%%%
%%%%%%

\section{The immersions}\label{section the immersions}

In the previous section we explicitly show immersion in $\sf$ with constant mean curvature and two principal curvatures. In this section we show that those examples are essentially all possible examples. The form of the hypersurfaces changes for $C\ne0$ and $C=0$. This is due to the fact that when $C\ne0$, the vector field $\eta$ has nonzero norm.

%{\color{red} add a comment} 

\subsection{The case $C\ne0$}

\begin{thm}\label{the immersions 1 intro}
Let $\sf$ be a space form with sectional curvature $a=1$, $a=-1$ or $a=0$ and let $M\subset \sf$ be a complete connected hypersurface with constant mean curvature, two non-constant principal curvatures  and such that the constant $C$ associated to $M$ is not zero. Assume that $n>2$. If $x_0\in M$, then,  $M$ agrees with the immersion  $\phi:S\times  \mathbb{R} \longrightarrow \sf$ given by 

$$\phi(y,t)= \frac{g(t)}{\sqrt{|C|}}\, y+ \gamma(t)+\frac{g(t)}{C}\, \rho_0 $$

where $\gamma(t)$ is an integral curve of the vector field $e_n$ satisfying $\gamma(0)=x_0$, $\rho_0=\eta(x_0)$, $g(t)=w(\gamma(t))$ and
$$S=\{y\in \Gamma(x_0)+Span\{\eta(x_0)\}: \<y,y\> = d\frac{|C|}{C}\}$$

\end{thm}

\begin{proof}
Since $M$ is  non-isoparametric because the principal curvatures are not constant then, $M$ must have a principal curvatures $\mu$ with multiplicity 1 and another principal curvature $\lambda$ with multiplicity $n-1$. See Remark \ref{multiplicities}. By definition of hypersurface with two principal curvatures we have that $\lambda\ne \mu$. We will assume without loss of generality that $\lambda>\mu$.  In this proof we will use the notation introduced in section  \ref{vector fields}. Recall that $x_0$ is a point in $M$. Let $y_0=e_n(x_0)$ and $z_0=\nu(x_0)$ and let $\gamma:R\to M$ be the integral curve of the vector field $e_n$ that satisfies $\gamma(0)=x_0$.  Let us work first the case that $a=\pm1$. Notice that if we consider the plane associate to the hypersurface $M$, see Definition \ref{the plane--def}, then the hypersurface $S$ used in the definition of this immersion $\phi$ is given by 
$S=\{y\, :\, \hbox{$y$ is perpendicular to $\Pi$ and }\, \<y,y\>\, =\, \frac{|C| d}{C}\, \}$. By theorem \ref{equations}, we have the distribution $\Gamma=\{v\in T_xM\, :\, \<v,e_n(x)\>=0\}$ is integrable. For any $t\in \mathbb{R}$, let us denote by $M_t$ the $(n-1)-$dimensional submanifold that integrates the distribution $\Gamma$ and contains $\gamma(t)$. Since $M$ is connected, we have that $M$ is the union of all $M_t$. For any $x_1\in M$, let us consider $t_1$ such that $x_1\in M_{t_1}$ and let us consider a smooth curve $\beta:[0,1]\to M_{t_1}$ such that $\beta(0)=x_1$ and $\beta(1)=\gamma(t_1)$. Since we are assuming that $C$ is not zero,  we can define the vector field
 $\zeta=x+\frac{w}{C}\, \eta$. By theorem \ref{zeta}, we have that $\zeta(\beta(s))$ is constant. Also, by Theorem \ref{equations} we have that $w(\beta(s))$ is constant. Therefore, the equations $\zeta(\beta(0))=\zeta(\beta(1))$ and
 $w(\beta(0))=w(\beta(1))$ can be written as 
 
 $$ x_1+\frac{w(x_1)}{C} \eta(x_1)=\gamma(t_1)+\frac{g(t_1)}{C} \rho_0 \com{and} w(x_1)=g(t_1)$$
 
 In the equation above we are using the fact that  $\eta(\gamma(t_1))=\eta(\gamma(0))=\rho_0$. If we define 
 
 $$y=\frac{\sqrt{|C|}}{g(t_1)}\,\left( x_1-\left(\gamma(t_1)+\frac{g(t_1)}{C} \rho_0\right)\right)=-\frac{\sqrt{|C|}}{g(t_1)}\, \frac{w(x_1)}{C}\,  \eta(x_1)= -\frac{\sqrt{|C|}}{C}\,  \eta(x_1) $$
 
 then, using the fact that $\<\eta,\eta\>=d C$  (see Theorem \ref{local relations}), we get that $\<y,y\>=\frac{|C|d}{C}$. Also, since $y$ is a multiple of $\eta(x_1)$, then $y$ is perpendicular to the plane $\Pi$. It follows that 
 $y\in S$, and from the equation $\frac{g(t_1)}{\sqrt{|C|}}\, y=x_1-(\gamma(t_1)+\frac{g(t_1)}{C} \rho_0)$ we conclude that $x_1=\phi(y,t_1)$. Since $x_1$ was any point in $M$ we conclude that $M$ agrees with the immersion $\phi$. %Recall that a direct computation shows that the immersion $\phi$ has CMC and with two different principal curvatures. 

 Let us now consider the case $a=0$.
Notice that if we define $Z=\lambda w\, e_n+e_n(w)\, \nu$, then by Theorem \ref{local relations} we have that $Z$ is a constant vector, that is, it is independent of the point $x\in M$. In particular $Z=Z(x_0)$ and therefore, the hypersurface $S$ used in the definition of the immersion $\phi$ is given by 
$S=\{y\, :\, \hbox{$y$ is perpendicular to $Z$ and }\, \<y,y\>\, =\, \frac{|C| d}{C}\, \}$. Recall that $Z$ is perpendicular to $\eta(x)$ for all $x$, see Remark \ref{X,Y,Z and eta}. By theorem \ref{equations}, we have the distribution $\Gamma=\{v\in T_xM\, :\, \<v,e_n(x)\>=0\}$ is integrable. For any $t\in {\bf R}$, let us denote by $M_t$ the $(n-1)-$dimensional submanifold that integrates the distribution $\Gamma$ and contains $\gamma(t)$. Since $M$ is complete and connected, we have that $M$ is the union of all $M_t$. For any $x_1\in M$, let us consider $t_1$ such that $x_1\in M_{t_1}$ and let us consider a smooth curve $\beta:[0,1]\to M_{t_1}$ such that $\beta(0)=x_1$ and $\beta(1)=\gamma(t_1)$. Since we are assuming that $C$ is not zero,  we can define the vector field
 $\zeta=x+\frac{w}{C}\, \eta$. By theorem \ref{zeta}, we have that $\zeta(\beta(s))$ is constant. Also, by Theorem \ref{equations} we have that $w(\beta(s))$ is constant. Therefore, the equations $\zeta(\beta(0))=\zeta(\beta(1))$ and
 $w(\beta(0))=w(\beta(1))$ can be written as 
 
 $$ x_1+\frac{w(x_1)}{C} \eta(x_1)=\gamma(t_1)+\frac{g(t_1)}{C} \rho_0 \com{and} w(x_1)=g(t_1)$$
 
 In the equation above we are using the fact that  $\eta(\gamma(t_1))=\eta(\gamma(0))=\rho_0$. If we define

 $$y=\frac{\sqrt{|C|}}{g(t_1)}\,\left( x_1-\left(\gamma(t_1)+\frac{g(t_1)}{C} \rho_0\right)\right)=-\frac{\sqrt{|C|}}{g(t_1)}\, \frac{w(x_1)}{C}\,  \eta(x_1)= -\frac{\sqrt{|C|}}{C}\,  \eta(x_1) $$
 
 then, using the fact that $\<\eta,\eta\>=d C$, we get that $\<y,y\>=\frac{|C|d}{C}$. Also, since $y$ is a multiple of $\eta(x_1)$, then $y$ is perpendicular to the vector  $Z$. It follows that 
 $y\in S$, and from the equation $\frac{g(t_1)}{\sqrt{|C|}}\, y=x_1-(\gamma(t_1)+\frac{g(t_1)}{C} \rho_0)$ we conclude that $x_1=\phi(y,t_1)$. Since $x_1$ was any point in $M$ we conclude that $M$ agrees with the immersion $\phi$.

\end{proof}

\subsection{Case $C=0$ .}

%Let us assume that $M$ satisfies the conditions of this case and let us consider the function $g(t)$, $\kappa_1(t)$, $\kappa_2(t)$,  $\gamma(t)$ and $N(t)$ defined at the end of the previous section. 

This section characterizes all  immersion with cmc on space forms with associated constant $C=0$. 

\begin{thm}\label{the immersions c=0}
Let $\sf$ be a space form with sectional curvature $a=1$, $a=-1$ or $a=0$ and let $M\subset \sf$ be a complete connected hypersurface with constant mean curvature, two non-constant principal curvatures  and such that the constant $C$ associated to $M$ is  zero. Assume that $n>2$. If $x_0\in M$, then,  $M$ agrees with the immersion 
$\phi:S\times  \mathbb{R} \longrightarrow \sf$ given by 

$$\phi(y,t)= y+ \gamma(t)+\frac{d}{2g(t)}\,\<y,y\>\,  \rho_0 $$

where $\gamma(t)$ is an integral curve of the vector field $e_n$ satisfying $\gamma(0)=x_0$, $\rho_0=\eta(x_0)$, $g(t)=w(\gamma(t))$ and
$S= \Gamma(x_0)$.

\end{thm}

\begin{proof} 
Since $C=0$, using Theorem \ref{local relations} we deduce that the vector field $\eta$ is constant on all $M$. By theorem \ref{equations}, we have the distribution $\Gamma=\{v\in T_xM\, :\, \<v,e_n(x)\>=0\}$ is integrable. For any $t\in \mathbb{R}$, let us denote by $M_t$ the $(n-1)-$dimensional submanifold that integrates the distribution $\Gamma$ and contains $\gamma(t)$. Since $M$ is connected, we have that $M$ is the union of all $M_t$. For any $x_1\in M$, let us consider $t_1$ such that $x_1\in M_{t_1}$ and let us consider a smooth curve $\beta:[0,1]\to M_{t_1}$ such that $\beta(0)=x_1$ and $\beta(1)=\gamma(t_1)$. Using part (v) from Theorem \ref{local relations} we have that $\hbox{Span}\{ \gamma(t), \gamma^\prime(t),\nu(\gamma(t))\}=\hbox{Span}\{\gamma(0),\gamma^\prime(0),\nu(\gamma(0))\}$.

Let consider first the case where $a=\pm 1$. We have that for some $y\in S$ and some constants $c_1$, $c_2$ and  $c_3$ we have 

$$x_1=y+c_1 \gamma(t_1)+c_2\gamma^\prime(t_1)+c_3\nu(\gamma(t_1))$$

Doing the dot product of $x_1$ with itself we obtain that 
\begin{eqnarray}\label{eqx1}
\<x_1,x_1\>=a=\<y,y\>+a c_1^2+dc_2^2+bc_3^2
\end{eqnarray}

Since the vector fields $Z$ and $X$ (see Definition \ref{def function and vector fields}) do not change on $M_{t_1}$ we obtain that 

\begin{eqnarray*}
0&=&\<Z(x_1),x_1\>=\<Z(\gamma(t_1)),x_1\>=dc_2\lambda(\gamma(t_1)) g(t_1)+bc_3g^\prime(t_1)\\
a \lambda(x_1)&=&\<X(x_1),x_1\>=\<X(\gamma(t_1)),x_1\>=ac_1\lambda(\gamma(t_1)) + b c_3
\end{eqnarray*}

From the two equations above, using the fact that $\lambda(\gamma(t_1))=\lambda(x_1)$ we obtain that 

$$  \lambda(\gamma(t_1)) c_1=\lambda(\gamma(t_1))-abc_3\com{and} \lambda(\gamma(t_1)) c_2=-bdc_3\frac{g^\prime(t_1)}{g(t_1)} $$

Let us denote by $ \lambda(\gamma(t_1))$ by $\kappa_1(t_1)$. Multiplying Equation (\ref{eqx1}) by $\kappa(t_1)^2$ and, using  the expressions for $\kappa_1(t_1)c_1$ and $\kappa_1(t_1)c_2$ in terms of $c_3$ that we just found, we obtain,

\begin{eqnarray*} 
\kappa_1(t_1)^2 a&=& \kappa_1(t_1)^2\<y,y\>+a ( \kappa_1(t_1) c_1)^2+d( \kappa_1(t_1) c_2) ^2+b ( \kappa_1(t_1) c_3)^2\\
    &=& \kappa_1^2(t_1) \<y,y\>+a\left(\kappa_1^2(t_1)-2 a b c_3 \kappa_1(t_1)+c_3^2\right)+\frac{dg^\prime(t_1)^2 c_3^2}{g^2(t_1)}+b\kappa_1^2(t_1)c_3^2\\
  &=&   \kappa_1^2(t_1) \<y,y\>+a\kappa_1^2(t_1)-2 b c_3 \kappa_1(t_1)+ \frac{dc_3^2}{g^2(t_1)}\left(adg^2(t_1) +g^\prime(t_1)^2 +bd\kappa_1^2(t_1)g^2(t_1)\right)\\
    &=&   \kappa_1^2(t_1) \<y,y\>+a\kappa_1^2(t_1)-2 b c_3 \kappa_1(t_1)
\end{eqnarray*}

Therefore, $ \kappa_1^2(t_1) \<y,y\>-2 b c_3 \kappa_1(t_1)=0$. Therefore $c_3=\frac{b\kappa_1(t_1) \<y,y\>}{2}$ anytime $\kappa_1(t_1)\ne0$. Using a continuity argument we conclude that $c_3=\frac{b\kappa_1(t_1) \<y,y\>}{2}$  also holds true for points $x_1$ with $\kappa_1(t_1)=0$. Then,

$$c_1=1-a\frac{\<y,y\>}{2},\quad c_2=-d\<y,y\>\frac{g^\prime(t_1)}{2 g(t_1)},\quad c_3=b\kappa_1(t_1)\frac{\<y,y\>}{2}$$

and therefore 

$$x_1= y+ \gamma(t_1)+\frac{d}{2g(t_1)}\,\<y,y\>\,  \eta(\gamma(t_1))= y+ \gamma(t_1)+\frac{d}{2g(t_1)}\,\<y,y\>\,  \rho_0 =\phi(y,t_1)$$

This finish the proof in the case $a=\pm 1$. 

The argument for $a=0$ follows is similar. Let us start by pointing out that when $a=0$ we have that for any $u\in\bfR{{n+1}}$ the hypersurface $\tilde{M}=\{x+u:x\in M\}=u+M$ is also a hypersurface with two principal curvatures and constant mean curvature. For this reason without loss of generality we may assume that $x_0\in W$ where

$$W=\hbox{Span}\{\gamma^\prime(0),\nu(\gamma(0))\}$$

Using part (v) from Theorem \ref{local relations} we have that 

$$\hbox{Span}\{ \gamma(t), \gamma^\prime(t),\nu(\gamma(t))\}=\hbox{Span}\{\gamma(0),\gamma^\prime(0),\nu(\gamma(0))\}=W$$

For any $x_1\in M$, let us consider $t_1$ such that $x_1\in M_{t_1}$ and let us consider a smooth curve $\beta:[0,1]\to M_{t_1}$ such that $\beta(0)=x_1$ and $\beta(1)=\gamma(t_1)$. Since the sum of the vectors spaces $S$ and $W$ is the whole $\bfR{{n+1}}$ we can write
 
$$x_1=y+a_1 \gamma^\prime(t_1)+a_2 \nu(\gamma(t_1))$$

As pointed out before, $\gamma(t)$ is contained in the the subspace $W$, therefore we can write

$$\gamma(t_1) = b_1 \gamma^\prime(t_1)+b_2 \nu(\gamma(t_1))$$

Since the vector fields $X$ and $Y$ are constant along the curve $\beta$ and the functions $w$, $\lambda$ and $e_n(w)$ are also constant along the curve $\beta$, after denoting $\lambda=\lambda(x_1)=\lambda(\gamma(t_1))$, $N=\nu (\gamma(t_1))$, we have that 

\begin{eqnarray*}
\nu(x_1)&=&-\lambda y+\left(-\lambda a_1+\lambda b_1\right) \gamma^\prime(t_1)+\left(-\lambda a_2+1+\lambda b_2 \right)N\\
e_n(x_1)&=&\frac{g^\prime(t_1)}{g(t_1)} y+\left(1+\frac{g^\prime(t_1)}{g(t_1)} a_1-\frac{g^\prime(t_1)}{g(t_1)} b_1\right) \gamma^\prime(t_1)+\left(\frac{g^\prime(t_1)}{g(t_1)} a_2-\frac{g^\prime(t_1)}{g(t_1)} b_2\right)N\\
\end{eqnarray*}

Let us denote by $\xi_1=a_1-b_1$ and $\xi_2=a_2-b_2$. Since $\<\nu(x_1),\nu(x_1)\>=b$ and $\<e_n(x_1),e_n(x_1)\>=d$, we have that

\begin{eqnarray*} 
b&=&\lambda^2 \<y,y\> + d \lambda^2\xi_1^2 + (1-\lambda \xi_2)^2b\\
d&=&\left(\frac{g^\prime(t_1)}{g(t_1)} \right)^2 \<y,y\> + d \left(1+\frac{g^\prime(t_1)}{g(t_1)} \xi_1\right)^2 +b \left(\frac{g^\prime(t_1)}{g(t_1)} \right)^2 \xi_2^2
\end{eqnarray*}

 By multiplying the first equation above by $db$ and adding it to the third equation and using the fact that 
 
 \begin{eqnarray} \label{egCeq0}
 g^\prime(t_1)^2+db \lambda^2 g(t_1)^2=0,\quad b\frac{g^\prime(t_1)^2}{g(t_1)^2}+d \lambda^2=0,\quad 
  d\frac{g^\prime(t_1)^2}{g(t_1)^2}+b \lambda^2=0
 \end{eqnarray}
 
 we obtain 
 
 $$ \frac{g^\prime(t_1)}{g(t_1)} \xi_1-\lambda \xi_2=0 $$
 
Since $\<\nu(x_1),e_n(x_1)\>=0$, 

\begin{eqnarray*}
0&=&-\lambda \frac{g^\prime(t_1)}{g(t_1)} \<y,y\> - d \lambda \xi_1 (1+\frac{g^\prime(t_1)}{g(t_1)} \xi_1)+b \frac{g^\prime(t_1)}{g(t_1)} \xi_2 (1-\lambda \xi_2)
\end{eqnarray*}

Combining the previous equation with the relations (\ref{egCeq0}) and $ \frac{g^\prime(t_1)}{g(t_1)} \xi_1-\lambda \xi_2=0 $ lead us to:
 
\begin{eqnarray}\label{expxis}
 \xi_1=-\frac{d g^\prime(t_1)}{2 g(t_1)}\<y,y\> \com{and}   \xi_2=\frac{b\lambda}{2} \<y,y\> 
 \end{eqnarray}
 
Recall that,  since $\eta$ is constant on all the manifold, we have that $\eta=-g^\prime(t_1) \gamma^\prime(t)+bd\lambda g(t_1) N$. A direct computation, using the expressions for $\xi_1$ and $\xi_2$ in Equation (\ref{expxis}),  shows that 
 
\begin{eqnarray*}
x_1 &=&y+a_1\gamma^\prime(t_1)+a_2 N \\
&=& y+(b_1+\xi_1) \gamma^\prime(t_1)+(b_2+\xi_2)N\\
&=& y+b_1 \gamma^\prime(t_1)-\frac{dg^\prime(t_1)}{2g(t_1)}\<y,y\> \gamma^\prime(t_1)+b_2N+\frac{b\lambda}{2}N\\
& =&y+\gamma(t_1)+\frac{d\<y,y\>}{2 g(t_1)}\eta
\end{eqnarray*}

  This finishes the proof.

\end{proof}

\section{Solving the ODE}\label{solvingode}

In this section we study the solutions of the ODE that provides the principal curvatures. The reason for doing this is to provide an exact description of the values of $C$ and $H$ that produce complete hypersurfaces.

\begin{mydef}\label{def of f, g, and kappas}
For any real numbers $n,a,b,d,C$ and $H$ such that $n,a,b$ and $d$ are integers satisfying $n \ge 2$, $|b|=|d|=1$ and $|a|\le1$, we define:

\begin{itemize}

\item
The function $f:\mathbb{R}\longrightarrow \mathbb{R}$ as

$$f=C-d v^2 (a + b (H+v^{-n})^2) $$

\item

Fixing $n,a,b$ and $d$ integers satisfying $n \ge 2$, $|b|=|d|=1$ and $|a|\le1$, we say that $(H,C)$\underline{ define a $CMC(n,a,b,d)$} if there exists a positive solution $g:\mathbb{R}\longrightarrow \mathbb{R}$ of the ODE

$$(g^\prime(t))^2=f(g(t)) $$

\item

If $(H,C)$ define a $CMC(n,a,b,d)$, and $g:\mathbb{R}\longrightarrow \mathbb{R}$ satisfies $(g^\prime(t))^2=f(g(t)) $ then, the functions $\kappa_1$ and $\kappa_2$ given by

$$ \kappa_1(t)=H+g(t)^{-n}\com{and} \kappa_2(t)=nH-(n-1) \kappa_1(t)$$

will be called the \underline{curvatures functions of the solution $g$}

\end{itemize}

\end{mydef}

\begin{rem}\label{cmc existence} If $M\subset \sf$ is a complete connected hypersurface with constant mean curvature $H$, with two non constant principal curvatures, and $C$ is the constant associated with $M$, then $(H,C)$ define a $CMC(n,a,b,d)$. This is because for any integral curve of the vector field $e_n$ defines a geodesic $\gamma(t)$ that is defined in the whole real line by the completeness of $M$ and, by Theorem \ref{local relations} part (iv), the function $g(t)=\omega(\gamma(t))$ satisfies the differential equation $(g^\prime(t))^2=f(g(t)) $.

\end{rem}

\begin{rem}\label{rem ODE for g}
It is evident that this ordinary differential equation can be solved by separation of variables. More precisely, when $f$ is positive between $v_1$ and $v_2$
with  $0\le v_1< v_2\le \infty$, then, if we define $F:(v_1,v_2)\to {\bf R}$ by $F(v)=\int_{v_0}^v\frac{1}{\sqrt{f(u)}}\, du$ (where $v_0$ is a number between $v_1$ and $v_2$), $J=F((v_1,v_2))=(T_1,T_2)$ and  $G:(T_1,T_2)\to (v_1,v_2)$  the inverse of $F$, then, for any constant $k$, the functions $y(t)=G(t-k)$ and $y(t)=G(k-t)$ are solutions of the ODE $y^\prime(t)^2=f(y(t))$. 

\end{rem}

\begin{prop}\label{typers} We can divide the set $\{g: g \hbox{ is a solution of ODE} (n,a,b,d,C,H) \}$ into 5 types.

\begin{enumerate}
\item
Periodic solution:  In this case the function $g$ takes values between $v_1$ and $v_2$ and the maximum and minimum of $g$ are achieved infinitely many times

\item
Unbounded with no minimum: In this case the function $g$ takes values from $v_1$ to  $\infty$, $g$ is one to one and the value $v_1$ is never reached.
\item

Unbounded with a minimum: in this case the function $g$ takes values from $v_1$ to  $\infty$,  $v_1$ is reached exactly one and all other values in the range of $g$ are reached exactly twice.

\item
bounded and not periodic with a minimum:  in this case  the function $g$ takes values between $v_1$ to  $v_2$,  $v_1$ is reached exactly one, $v_2$ is never reached, and all other values in the range of $g$ are reached exactly twice.

\item
bounded and not periodic with a maximum:  in this case the function $g$ takes values between $v_1$ to  $v_2$,  $v_2$ is reached exactly one, $v_1$ is never reached, and all other values in the range of $g$ are reached exactly twice.

\end{enumerate}

\end{prop}

\begin{proof}
The classification follows from considering all the possible intervals $I$'s in $\mathbb{R}_+=\{t:t>0\}$ such that $f$ vanishes at the boundary points of $I$ and $f$ is positive on $I$. Since $v=0$ is a vertical asymptote of $f$, then $0$ cannot be the initial point of the interval $I$. A direct computation shows that the positive critical points of $f$ satisfy an equation of the form

$$b (n-1)(v^{-n})^2+b H (n-2) v^{-n} -a-b H^2=0$$

which we can be easily solved for $v^{-n}$. Then, the function $f$ can have a most two positive critical points, and moreover, in case it has two positive critical points it cannot vanish in both of them. Therefore the possible cases for $I$ are: (i) $I=[v_1,v_2]$ where neither $v_1$ or $v_2$ is a critical point. In this case  the solution $g$ of $(g^\prime)^2=f(g)$ is periodic and its range is $I$. Figures \ref{aeqn1beq1ex4} and  \ref{aeq1beqn1ex2}  show examples of values of $a,b,d, H$ and $C$ that produce this type of solutions. (ii) $I=[v_1,\infty)$ with $v_1$ a critical  point of $f$. In this case  the solution $g$ of $(g^\prime)^2=f(g)$ is unbounded with no minimum and its range is the open interval $(v_1,\infty)$. Figures \ref{aeqn1beq1ex3}, \ref{aeqn1beq1ex1}, \ref{aeq1beqn1ex1} \ref{aeqn1beqn1ex1} and  \ref{aeq0beqn1ex1}  show examples of values of $a,b,d, H$ and $C$ that produce this type of solutions. (iii) $I=[v_1,\infty)$ with $v_1$ not a critical  point of $f$. In this case  the solution $g$ of $(g^\prime)^2=f(g)$ is unbounded with a minimum and its range is the interval $I$. (iv) $I=[v_1,v_2]$ with $v_2$ a critical  point of $f$. In this case  the solution $g$ of $(g^\prime)^2=f(g)$ is bounded not periodic and its range is the interval $[v_1,v_2)$ Figure \ref{aeqn1beq1ex2}  shows an example of values of $a,b,d, H$ and $C$ that produce this type of solutions. (v)  $ I=[v_1,v_2]$ with $v_1$   a critical  point of $f$. In this case  the solution $g$ of $(g^\prime)^2=f(g)$ is bounded, not periodic and  with no minimum and its range is the interval $(v_1,v_2]$. Figure \ref{aeq1beqn1ex3}  shows an example of values of $a,b,d, H$ and $C$ that produce this type of solutions.  After solving the differential equation $(g^\prime)^2=f(g)$ (see Remark  \ref{rem ODE for g}) we can see that case (i) provides a solution of type (1), case (ii) provides a solution of type (2) and so on.

The following figures show possible examples of values of $n,a,b,d,C$ and $H$ that produce graphs of the function $f$ for all possible types.

\begin{figure}[hbtp]
\begin{centering}\includegraphics[width=.3\textwidth]{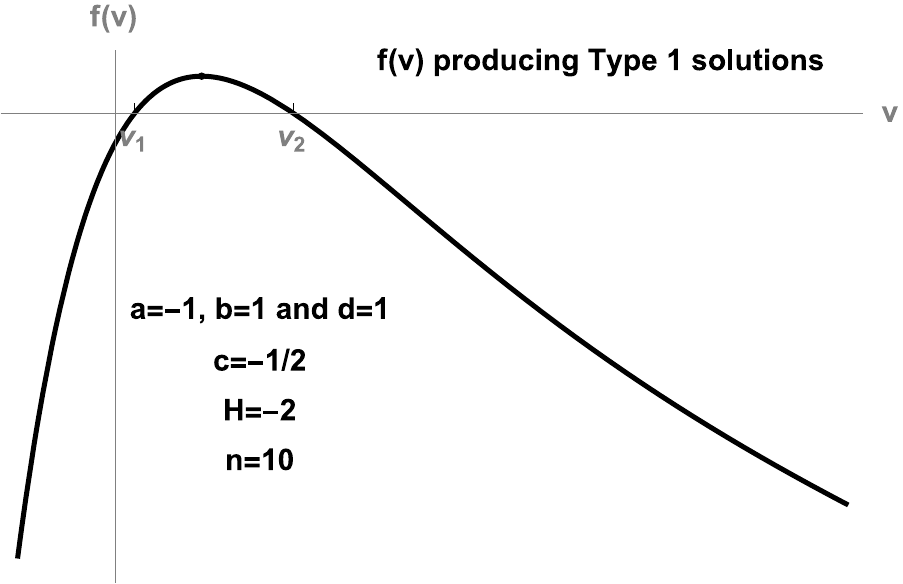} \includegraphics[width=.3\textwidth]{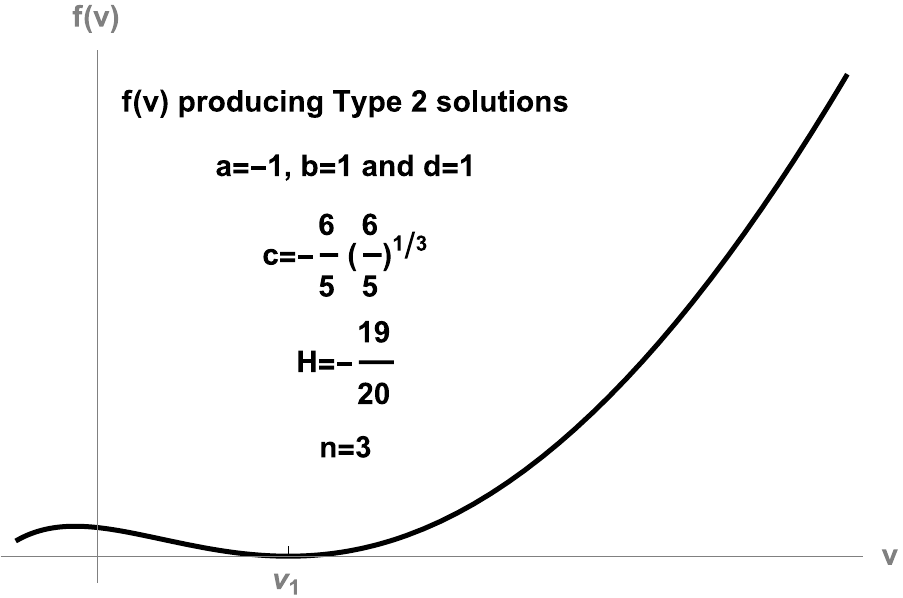} \includegraphics[width=.3\textwidth]{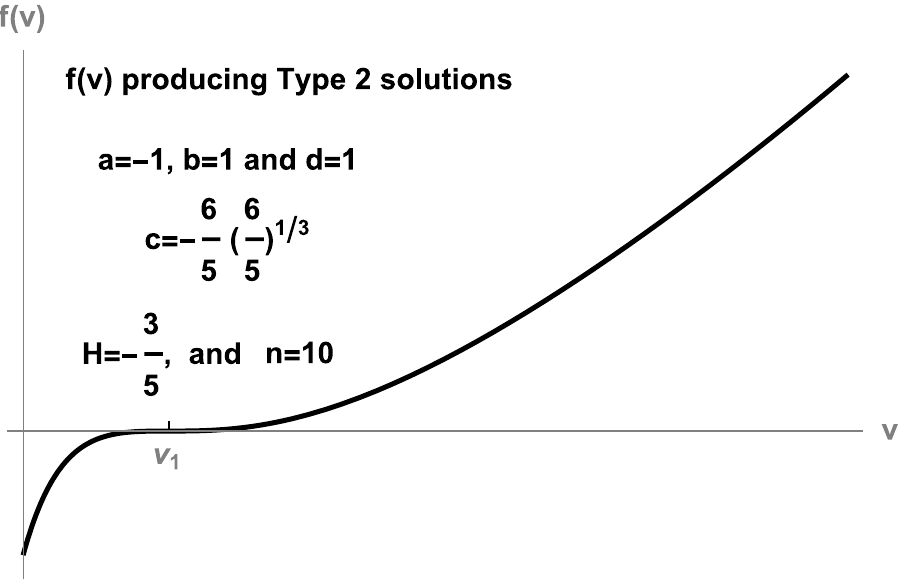}
\end{centering}
\caption{Type 1 solutions provide periodic solutions of the equation $(g^\prime)^2=f(g)$ while Type 2 produce unbounded solution whose range is the open interval $(v_1,\infty)$. The zero of $f$ at $v_1$ can have multiplicity two or three.}\label{fig1}
\end{figure}

\begin{figure}[hbtp]
\begin{centering}\includegraphics[width=.3\textwidth]{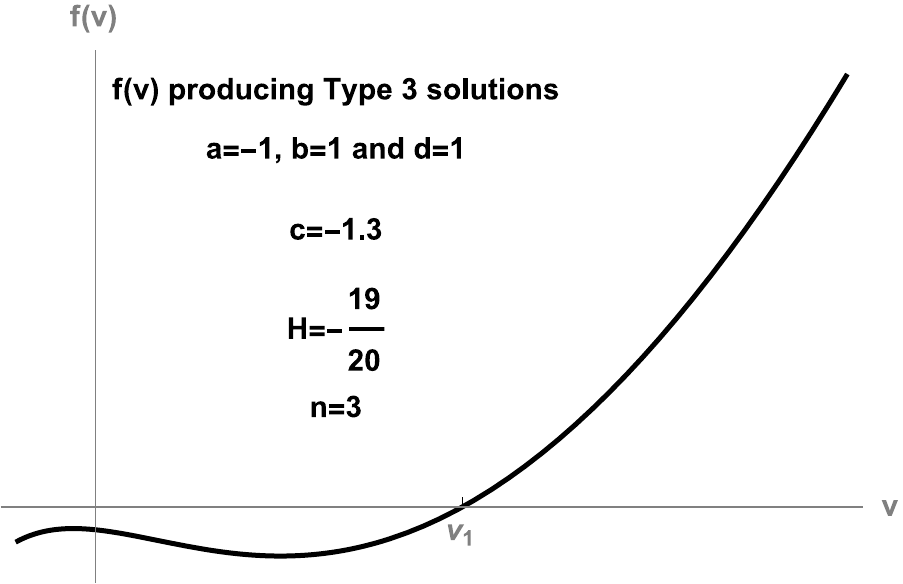} \includegraphics[width=.3\textwidth]{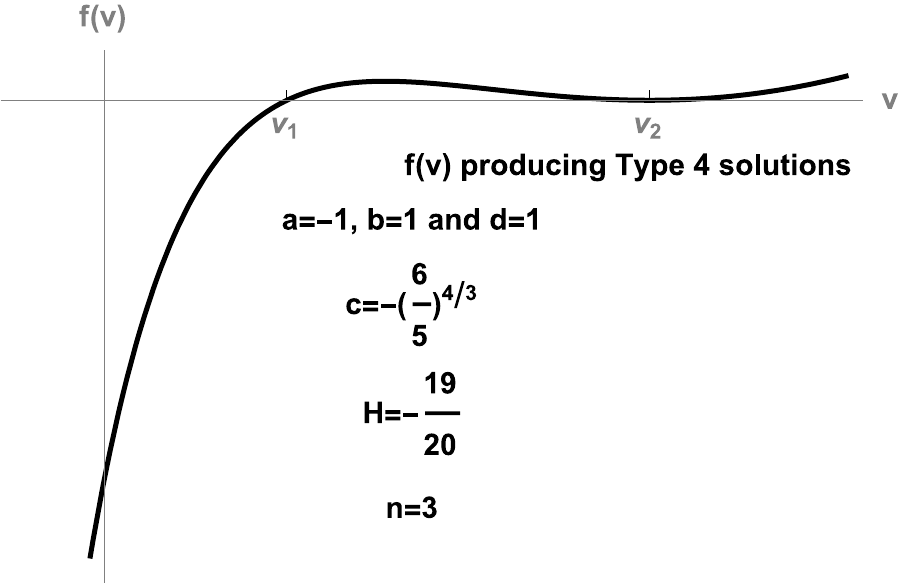} \includegraphics[width=.3\textwidth]{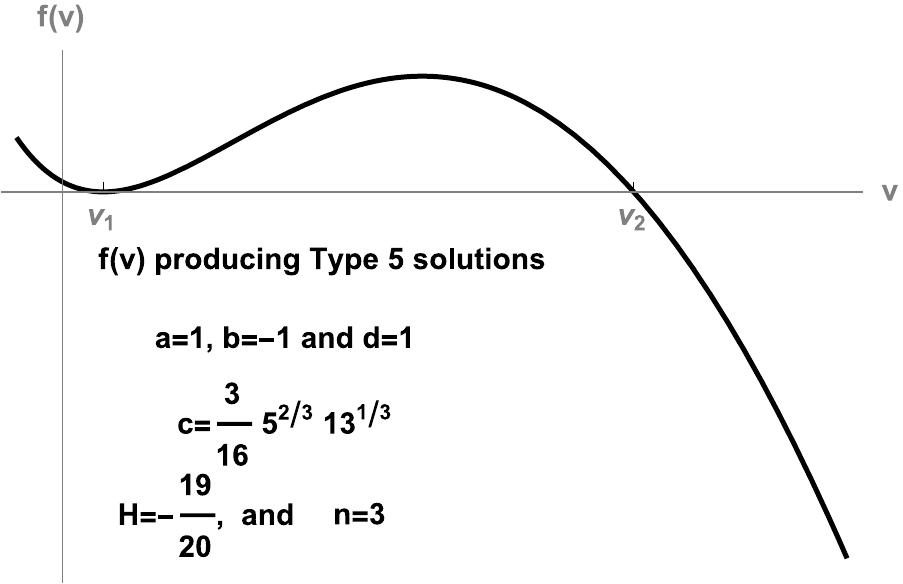}
\end{centering}
\caption{Type 3 solutions provide unbounded solutions of the equation $(g^\prime)^2=f(g)$  with a minimum.  Type 4 solutions produce bounded solutions with a minimum and not a maximum and Type 5 solutions produce bounded solutions with a maximum but not a minimum}\label{fig2}
\end{figure}
%%%%

\end{proof}

\section{Module space of complete immersions}\label{mods}

In this section we describe the posible values of $H$ and $C$ that can be achieved as values of a cmc hypersurfaces and with two principal curvatures on a space form $\sf$ under the assumption that the principal curvatures are not constant and  $n>2$.

\subsection{Space of solutions in case: $ad=-1$ and $bd=1$}  In this section we describe the values of $(H,C)$ that provide complete immersions in $\sf$ when $ad=-1$ and $bd=1$. This case includes cmc hypersurfaces in the Hyperbolic $(n+1)-$dimensional space.
In this case the differential equation for $g$ reduces to

$$g'(t)^2=f(g)=C + g^2 - H^2 g^2 - g^{2 - 2 n} - 2 H g^{2 - n}$$

where,
 $$f:(0,\infty)\to \mathbb{R} \com{(we are only considering positive inputs for the function $f$)}$$
 Let us define the following two functions: 

$$q_1,r_1:(-\infty,-\frac{2 \sqrt{n-1}}{n}]\cup [1,\infty)\to \R\com{and} r_2:[-1,-\frac{2 \sqrt{n-1}}{n}]\to \R \com{and} q_2:(-1,-\frac{2 \sqrt{n-1}}{n}]\to \R  $$

given by

$$ r_1(H)=\frac{n \left(h \sqrt{h^2 n^2-4 n+4}+\left(h^2-2\right) n+2\right)}{2^{\frac{n-2}{n}} (n-1)^{\frac{2 (n-1)}{n}} \left(\sqrt{h^2 n^2-4 n+4}-h (n-2)\right)^{2/n}}$$
and

$$ r_2(H)=\frac{n \left(h \left(-\sqrt{h^2 n^2-4 n+4}\right)+\left(h^2-2\right) n+2\right)}{2^{\frac{n-2}{n}} (n-1)^{\frac{2 (n-1)}{n}} \left(-\sqrt{h^2 n^2-4 n+4}-h (n-2)\right)^{2/n}}$$

and 

$$ q_1=\left(\frac{\sqrt{h^2 n^2-4 n+4}-h (n-2)}{2 (n-1)}\right)^{-1/n}\com{and} q_2=\left(\frac{-\sqrt{h^2 n^2-4 n+4}-h (n-2)}{2 (n-1)}\right)^{-1/n} $$

\begin{figure}[hbtp]
\begin{centering}\includegraphics[width=.3\textwidth]{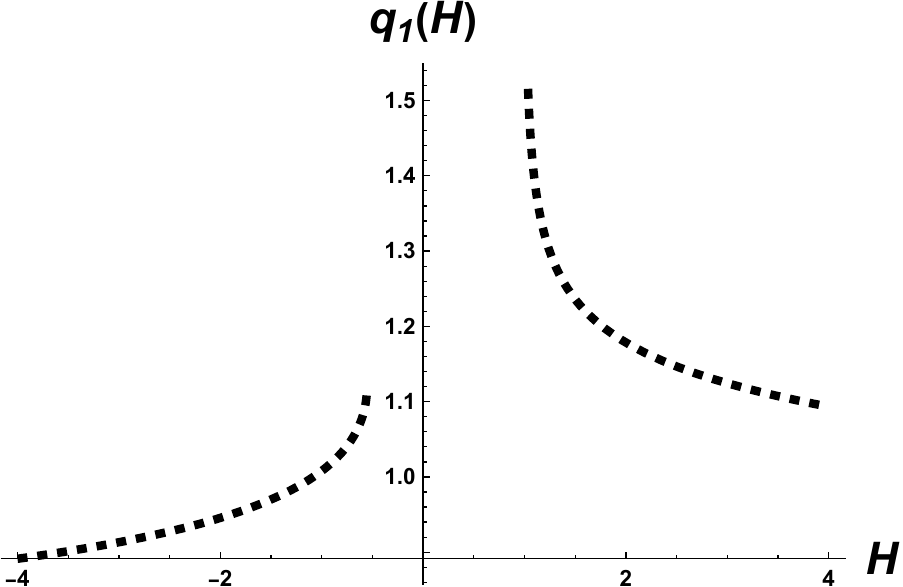} \hskip.5cm \includegraphics[width=.3\textwidth]{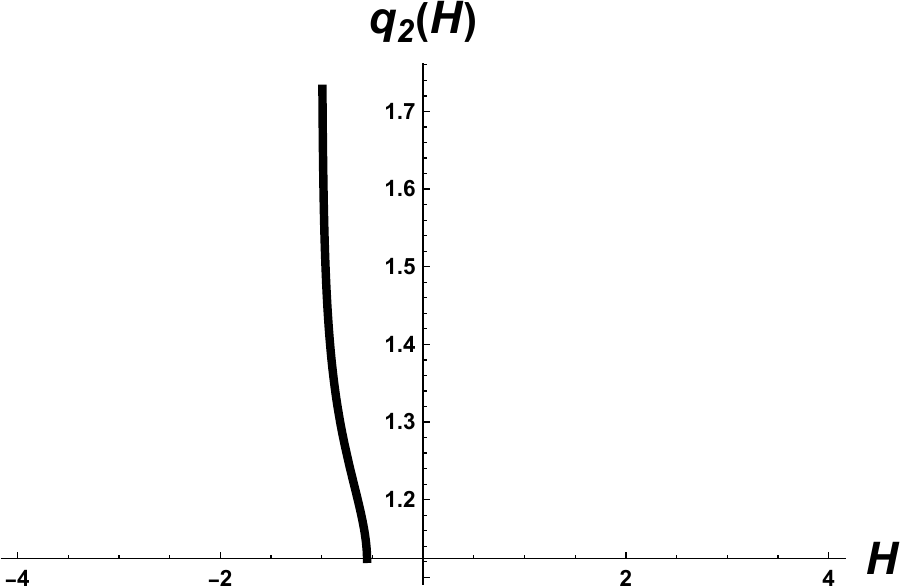} \hskip0.5cm \includegraphics[width=.3\textwidth]{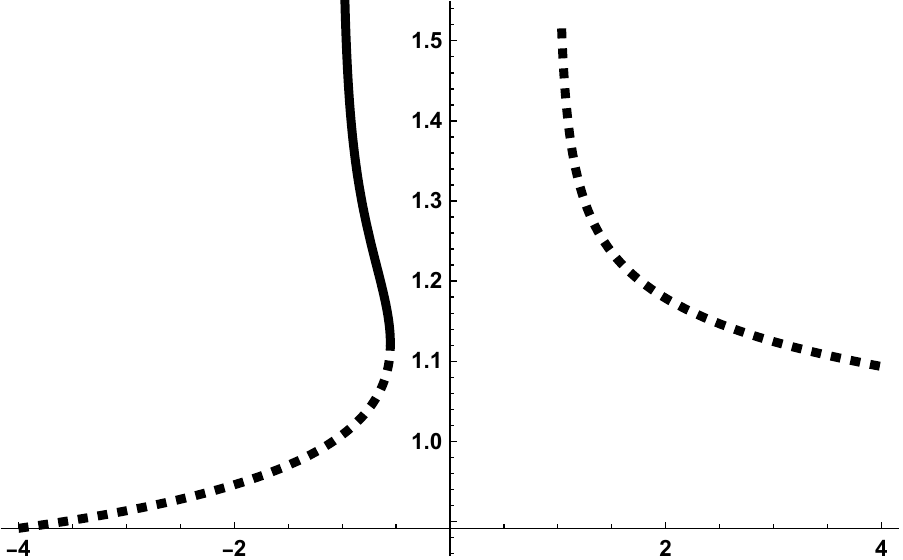} 

\end{centering}
\caption{Graph of the functions $q_1$ and $q_2$ when $n=12$. Both graphs have the same value at $H=-\frac{2 \sqrt{n-1}}{n}$. The image on the right shows both graphs in the same cartesian plane.}\label{abn1bd1qs}
\end{figure}

\begin{figure}[hbtp]
\begin{centering}\includegraphics[width=.3\textwidth]{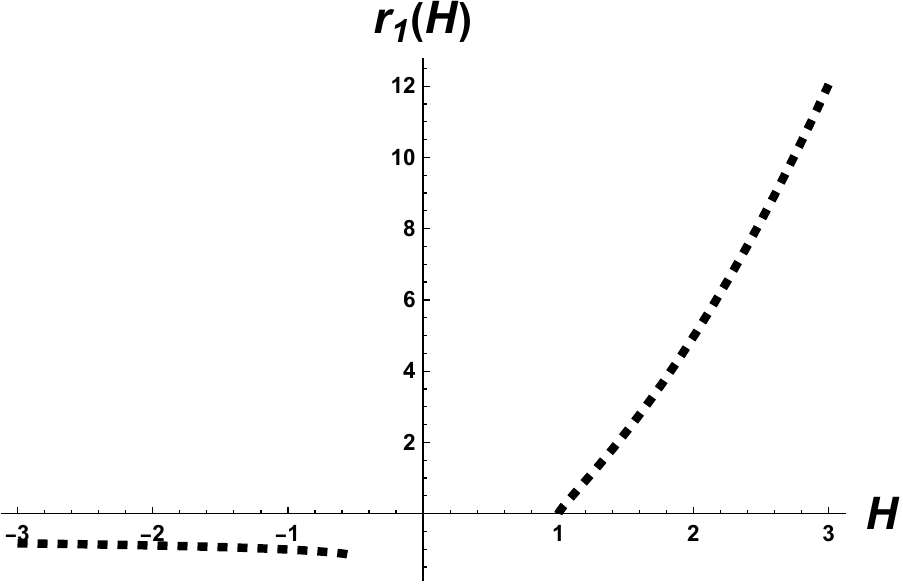} \hskip.5cm \includegraphics[width=.3\textwidth]{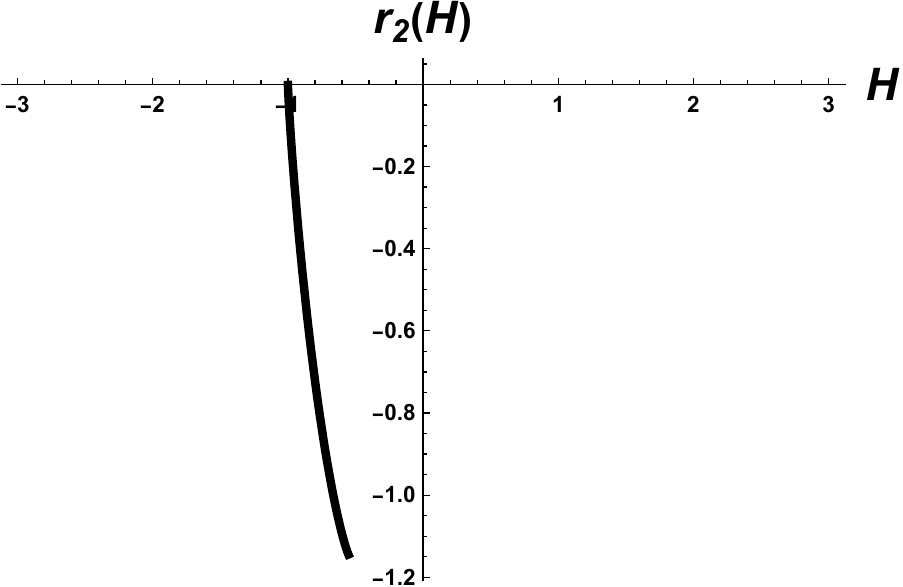} \hskip0.5cm \includegraphics[width=.3\textwidth]{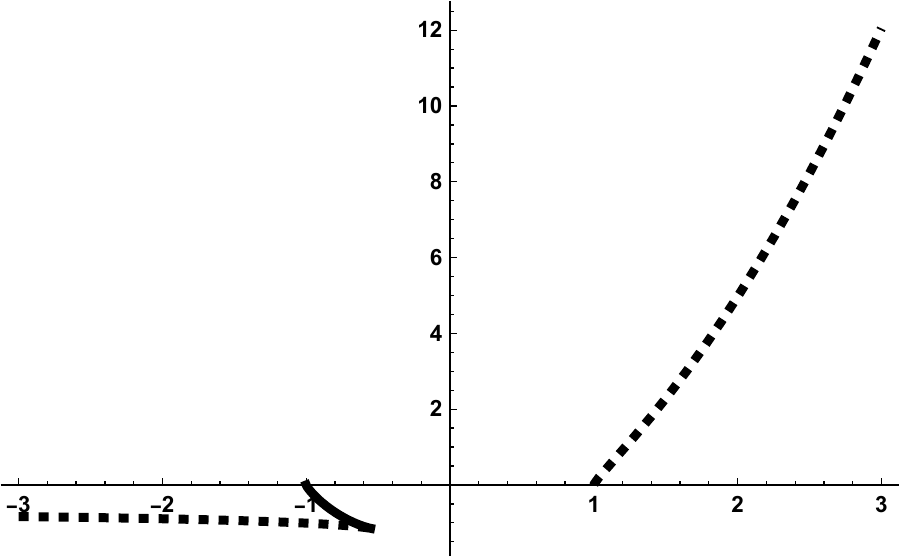} 

\end{centering}
\caption{Graph of the functions $r_1$ and $r_2$ when $n=12$. Both graphs have the same value at $H=-\frac{2 \sqrt{n-1}}{n}$. The image on the right shows both graphs in the same cartesian plane..}\label{abn1bd1rs}
\end{figure}

Let us explain the relation between the functions $r_1$ and $r_2$ and the graph of the function  $f$. It is clear that the limit when $v$ goes to $0$ from the right of the function $f(v)$ is negative infinity. A direct computation shows that the derivative of $f$ is given by 

$$ f'(v)=v \left(2 - 2 h^2 + 2 (-1 + n) v^{-2 n} + 2 h (-2 + n) v^{-n}\right) $$

Therefore we have that the positive critical points of $f$ satisfies the quadratic equation $2 - 2 h^2 + 2 (-1 + n) x^2 + 2 h (-2 + n) x=0$ with $x=v^{-n}$. Solving this quadratic equation produces the values of the functions $q_1$ and $q_2$ defined above.  The expressions for $r_i(H)$ come from evaluating the function  $-f$ when $C=0$ and $v$ is $q_i$.  The following theorem summarizes the values of $H$ and $C$ that are achieved as values of non trivial complete hypersurfaces with two principal curvatures and cmc.

%\label{hit}
\begin{thm} \label{hyp} Let $M\subset \sf$ be a complete hypersurface with two principal curvatures. Assume that $M$ has constant mean curvature $H$, that $C$ is the constant associated with $M$ and that the principal curvatures of $M$  are not constant. Also assume that the sign of $H$ is determined by the condition $\lambda-\mu>0$. If $n>2$, $ad=-1$ and $bd=1$, then the value of $H$  and $C$ satisfy the following relations

\begin{enumerate}
\item
$H$ can be any real number.
\item
When $-\infty<H<-1$, then $C>r_1(H)$ and  $g(t)$ is a {\it type 1} solution.
\item
When $H=-1$, then   $C>r_1(-1)$ and either $g(t)$ is a {\it type 1} solution when $r_1(-1)<C<0$ or $g(t)$ is a {\it type 3} solution when $C\ge 0$.
\item
When $-1<H<-\frac{2\sqrt{n-1}}{n} $ then $C$ can be any real number. Moreover,

\begin{itemize}
\item
 if $C \le r_1(H)$ then $g(t)$ is a  {\it type 3} solution.
\item
if $r_1(H)<C<r_2(H)$ then there are two types of immersions, one with  $g(t)$ of {\it type 1}  and the other immersion with $g(t)$ a {\it type 3} solution.

\item
if $C=r_2(H)$ then there are two types of immersions, one with $g(t)$ of {\it type 4}  and the other immersion with $g(t)$ a {\it type 2} solution.
\item
 if $C > r_2(H)$ then $g(t)$ is a  {\it type 3} solution.

 \end{itemize}
 \item
When $H=-\frac{2\sqrt{n-1}}{n}$, then $C$ can be any number and $g(t)$ is of {\it type 2} if $C= r_1(-\frac{2\sqrt{n-1}}{n})$, otherwise $g(t)$ is a  {\it type 3} solution.
 \item
When $-\frac{2\sqrt{n-1}}{n}<H<-1$, then $C$ can be any number and $g(t)$ is a  {\it type 3} solution.

 \item
When $H=1$, then  $C>r_1(1)=0$ and  $g(t)$ is a {\it type 3} solution.

 \item
When $H>1$, then  $C>r_1(H)$ and  $g(t)$ is a {\it type 1} solution.
\end{enumerate} 

\end{thm}

\begin{proof}
Notice that regardless of the values of $C$ and $H$ we have that the limit when $v\rightarrow 0^+$ is negative infinity. The theorem is a direct application of remark \ref{rem ODE for g} and the fact that:

\begin{enumerate}

\item
When $-\infty<H<-1$, $q_1(H)$ is the only positive critical point of the function $f(v)$ and it corresponds to a maximum. Moreover: (i) when $C>r_1(H)$, $f(v)$ has exactly two roots $v_1$ and $v_2$ and it is  only positive on an interval $(v_1,v_2)$ that contains the critical point $q_1(H)$. (ii) when $C\le r_1(H)$, $f$ is never positive.
\item
When $H=-1$, the function  $f$ has only one critical point which is a maximum. Moreover: (i) if  $C\le r_1(-1)$ then $f$ is never positive. (ii) If  $r_1(-1)<C<0$, then, $f(v)$ has exactly two roots $v_1$ and $v_2$ and it is  only positive on an interval $(v_1,v_2)$ that contains the critical point $q_1(H)$, and (iii) if $C\ge 0$, then $f(v)$ has only one root $v_1$ and it is positive on the interval $(v_1,\infty)$.

\item
When $-1<H<-\frac{2\sqrt{n-1}}{n} $ then the function $f$ has two critical points, a local maximum $q_1$ and, a local minimum  $q_2$. Moreover,

\begin{itemize}
\item
 if $C < r_1(H)$ the function $f$ has only one root $v_1>q_2(H)$ and $f$ is positive on the interval $(v_1,\infty)$.
\item
 if $C = r_1(H)$ the function $f$ has two roots $v_1$ and $v_2$ with $v_1<v_2$. $v_1$ is a local maximum and  $f$ is positive on the interval $(v_2,\infty)$
 
\item
if $r_1(H)<C<r_2(H)$ then $f$ has three roots $v_1, v_2$ and $v_3$, and $f$ is positive on the intervals $(v_1,v_2)$ and on the interval $(v_3,\infty)$.  Figure \ref{aeqn1beq1ex4}  shows particular choices for the parameters that produces an example of a function $f$ for this case. 

\item
if $C=r_2(H)$ then $f$ has two roots $v_1$ and  $v_2$ with  $v_1<v_2$. $v_2$ is a local minimum and $f$ is positive on the intervals $(v_1,v_2)$ and in the interval $(v_2,\infty)$. Figure \ref{aeqn1beq1ex3} (or Fiugre \ref{aeqn1beq1ex2}) shows particular choices for the parameters that produces an example of a function $f$ for this case. %add here
\item
 if $C > r_2(H)$ then  $f$ has only one root $v_1<q_1(H)$ and $f$ is positive on the interval $(v_1,\infty)$.

 \end{itemize}
 \item
When $H=-\frac{2\sqrt{n-1}}{n}$, then the function $f$ has only one critical point which is also an inflection point. When  $C=r_1(-\frac{2\sqrt{n-1}}{n})$ then the function $f$ has only one root $v_1$ with multiplicity three and it is positive on the interval $(v_1,\infty)$. In this case, $f:(v_1,\infty)\to \R$ produces a solution $g(t)$  on the whole line of {\it type 2}. Figure \ref{aeqn1beq1ex1}  shows particular choices for the parameters that produces an example of a function $f$ for this case.  When $C\ne r_1(-\frac{2\sqrt{n-1}}{n})$, then the function $f$ has only one root with multiplicity one and it is positive on the interval $(v_1,\infty)$. In this case there is a solution $g(t)$ of {\it type 3}.
 \item
When $-\frac{2\sqrt{n-1}}{n}<H<1$, then for any value of $C$, the function $f$ has not critical points, only one root $v_1$ and it is positive on the interval $(v_1,\infty)$.

 \item
When $H=1$, then the function $f$ does not have critical points. When $C\le0$, then $f$ is never positive. When  $C>0$ then $f$ has only one root $v_1$   and it is positive on the interval $(v_1,\infty)$.

 \item
When $H>1$, then the function $f$ has only one critical point that is a maximum. Moreover when $C\le r_1(H)$, the function $f$ is never positive. When $C>r_1(H)$, the function $f$ has only two roots $v_1$ and $v_2$ and  it is positive on the interval $(v_1,v_2)$.
\end{enumerate} 

Recall that the existence of a smooth positive solution $g(t)$ on the whole real line guarantees the existence of the  immersion due to fact that the differential equation for the maps $\alpha(t)$ and $\beta(t)$ defined in  Theorem \ref{cneq0aneq0} is a linear system of ordinary differential equations.
\end{proof}

%% SECTION ad=1 and bd=-1

\subsection{Space of solutions in case: $ad=1$ and $bd=-1$} In this section we describe the values of $(H,C)$ that provide complete immersions in $\sf$ when $ad=1$ and $bd=-1$. This case includes space-like cmc hypersurfaces in the $(n+1)-$dimensional de Sitter space.
In this case the differential equation for $g$ reduces to

$$g'(t)^2=f(g)=C - g^2 + H^2 g^2 + g^{2 - 2 n} + 2 H g^{2 - n}$$

Now the corresponding functions $r_1,r_2$ and $q_1$ and $q_2$ are: 

$$q_1:(-\infty,-\frac{2 \sqrt{n-1}}{n}]\cup (1,\infty)\to \R \com{and}  q_2:(-1,-\frac{2 \sqrt{n-1}}{n}]\to \R  $$

and 

$$  r_1:(-\infty,-\frac{2 \sqrt{n-1}}{n}]\cup [1,\infty)\to \R \com{and} r_2:[-1,-\frac{2 \sqrt{n-1}}{n}]\to \R$$

given by

$$ r_1(H)=\frac{-n \left(h \sqrt{h^2 n^2-4 n+4}+\left(h^2-2\right) n+2\right)}{2^{\frac{n-2}{n}} (n-1)^{\frac{2 (n-1)}{n}} \left(\sqrt{h^2 n^2-4 n+4}-h (n-2)\right)^{2/n}}$$

$$ r_2(H)=\frac{-n \left(h \left(-\sqrt{h^2 n^2-4 n+4}\right)+\left(h^2-2\right) n+2\right)}{2^{\frac{n-2}{n}} (n-1)^{\frac{2 (n-1)}{n}} \left(-\sqrt{h^2 n^2-4 n+4}-h (n-2)\right)^{2/n}}$$

and 

$$ q_1=\left(\frac{\sqrt{h^2 n^2-4 n+4}-h (n-2)}{2 (n-1)}\right)^{-1/n}\com{and} q_2=\left(\frac{-\sqrt{h^2 n^2-4 n+4}-h (n-2)}{2 (n-1)}\right)^{-1/n} $$

\begin{figure}[hbtp]
\begin{centering}\includegraphics[width=.5\textwidth]{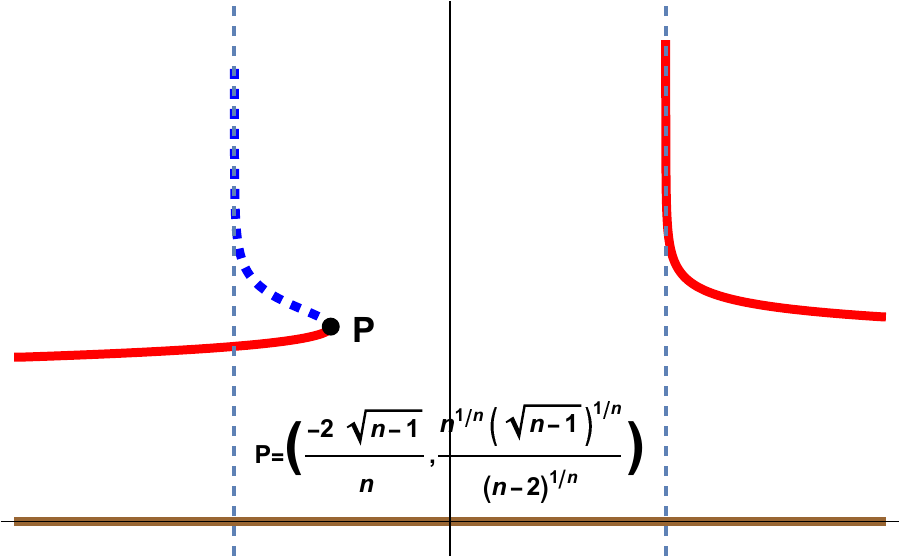} 
\end{centering}
\caption{The dashed curve is the graph of the function $q_2(H)$ and the solid curve is the graph of the function $q_1(H)$. Both graphs have the same value at $H=-\frac{2 \sqrt{n-1}}{n}$. }\label{adeq1bdeqn1qs.}
\end{figure}

\begin{figure}[hbtp]
\begin{centering}\includegraphics[width=.4\textwidth]{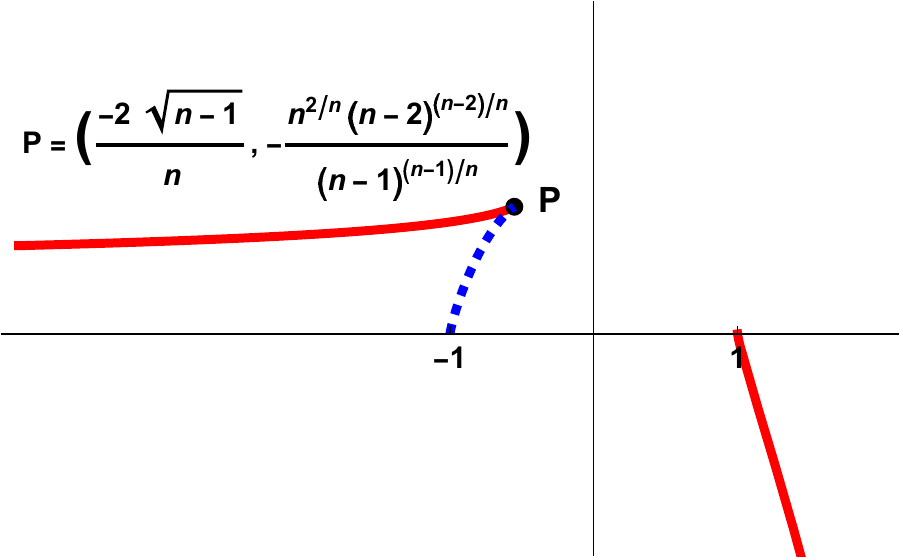} 
\end{centering}
\caption{The dashed curve is the graph of the function $r_2(H)$ and the solid curve is the graph of the function $r_1(H)$. Both graphs have the same value at $H=-\frac{2 \sqrt{n-1}}{n}$.}\label{adeq1bdeqn1rs.}
\end{figure}

\vfil
\eject

The following theorem summarizes the values of $H$ and $C$ that are achieved as values of non trivial complete hypersurfaces with two principal curvatures and cmc.

%\label{deSit}
\begin{thm}\label{deSit}Let $M\subset \sf$ be a complete hypersurface with two principal curvatures. Assume that $M$ has constant mean curvature $H$, that $C$ is the constant associated with $M$ and that the principal curvatures of $M$  are not constant.  Also assume that the sign of $H$ is determined by the condition $\lambda-\mu>0$.  If $n>2$, $ad=1$ and $bd=-1$, then the value of $H$  and $C$ satisfy the following relations

\begin{enumerate}
\item
$H$ cannot take values in the interval $[-\frac{2\sqrt{n-1}}{n},1]$, any other value of $H$ represents the mean curvature of one of these immersions. 
\item
When $-\infty<H<-1$, then $C\le r_1(H)$ and  $g(t)$ is a {\it type 2} solution when $C=r_1(H)$ and a {\it type 3} solution when $C<r_1(H)$.
\item
When $H=-1$, then   $0<C\le r_1(-1)$ and  $g(t)$ is a {\it type 3} solution when $0<C<r_1(-1)$ and a {\it type 2} solution when $C= r_1(-1)$. 
\item
When $-1<H<-\frac{2\sqrt{n-1}}{n} $ then $r_2(H)<C\le r_1(H)$ and $g(t)$ is a {\it type 5} solution when $C=r_1(H)$ and a {\it type 1} solution when $r_2(H)<C<r_1(H)$

 \item
When $H>1$, then  $C\le r_1(H)$ and  $g(t)$ is a {\it type 2} solution when $C=r_1(H)$ and a {\it type 3} solution when $C<r_1(H)$.
\end{enumerate} 

\end{thm}

\begin{proof}
Notice that regardless of the values of $C$ and $H$ we have that the limit when $v\rightarrow 0^+$ is positive infinity. Moreover, the behavior of $f$ near zero does not allow to have a positive solution $g(t)$ with range containing an interval of the form $(0,v_1)$. % in this case, every solution $g(t)$ will reach the value zero with no derivative. 
The proof of the theorem follows the same arguments as those on Theorem \ref{hyp}

\begin{enumerate}

\item
When $-\infty<H<-1$, $q_1(H)$ is the only positive critical point of the function $f$ and it corresponds to a minimum. Moreover: (i) when $C>r_1(H)$, $f$ has not roots, it is  always positive. (ii) when $C=r_1(H)$, $f$ has only one root $v_1$ with multiplicity 2 and $f$ is positive on the intervals $(0,v_1)$ and the interval $(v_1,\infty)$. 
Figure \ref{aeq1beqn1ex1} shows particular choices for the parameters that produces an example of a function $f$ for this case.
The function $f$ on the $(0,v_1)$ does not provide a solution $g$ define on the whole real line. The function $f$ on the interval $(v_1,\infty)$ provides a {\it type 2} solution $g(t)$.  (ii) If  $C<r_1(H)$, then, $f$ has exactly two roots $v_1$ and $v_2$ and it is  only positive on the intervals $(0,v_1)$ and $(v_2,\infty)$.
\item
When $H=-1$, the function  $f$ has only one critical point which is a minimum. Moreover: (i) if  $C> r_1(-1)$ then $f$ is always positive. (ii) If  $C=r_1(-1)$, then, $f$ has exactly one root $v_1$ with multiplicity 2 and it is  positive on the intervals $(0,v_1)$ and $(v_1,\infty)$. (iii) If  $0<C<r_1(-1)$, then, $f$ has exactly two roots $v_1$ and $v_2$ and it is  only positive on the  intervals $(0,v_1)$ and $(v_1,\infty)$. (iv) If $C<0$ then $f$ has only one root $v_1$ and it is positive on the interval $(0,v_1)$.
\item
When $-1<H<-\frac{2\sqrt{n-1}}{n} $ then the function $f$ has two critical points, a local minimum $q_1(H)$ and, a local maximum  $q_2(H)$. Moreover, (i) If $C>r_1(H)$, then, the function $f$ only has one root $v_1$ and it is positive on the interval $(0,v_1)$ (ii) If $C=r_1(H)$, then $f$ has two roots $v_1$ and $v_2$ with $v_1$ a local minimum and $f$ is positive on the intervals $(0,v_1)$ and $(v_1,v_2)$. Figure \ref{aeq1beqn1ex3} shows particular choices for the parameters that produces an example of a function $f$ for this case. (iii) If $r_2(H)<C<r_1(H)$, then $f$ has three roots $v_1$,  $v_2$ and $v_3$ and $f$ is positive on the intervals $(0,v_1)$ and $(v_2,v_3)$.  Figure \ref{aeq1beqn1ex2} shows particular choices for the parameters that produces an example of a function $f$ for this case. (iv) If $C=r_2(H)$ then $f$ has two roots $v_1$ and $v_2$ with $v_2$ a local maximum. In this case $f$ is positive on the interval $(0,v_1)$. (v) If $C<r_2(H)$, then $f$ has only one root $v_1$ and it is positive on the interval $(0,v_1)$.

 \item
When $-\frac{2\sqrt{n-1}}{n}\le H<1$, then, for any $C$, $f$ has only one root $v_1$, and it is positive on the interval $(0,v_1)$.

 \item
When $H=1$, then the function $f$ does not have critical points. When $C\ge0$, then $f$ is always positive. When  $C<0$ then $f$ has only one root $v_1$ and it is positive on the interval $(0,v_1)$.

 \item
When $H>1$, then the function $f$ has only one critical point $q_1(H)$ that is a minimum. Moreover when $C> r_1(H)$, the function $f$ is always positive. When  $C=r_1(1)$, the function $f$ has one root $v_1=q_1(H)$, in this case $f$ is positive on the intervals $(0,v_1)$ and $(v_1,\infty)$. When $C<r_1(H)$, the function $f$ has only two roots $v_1$ and $v_2$  and it is positive on the intervals $(0,v_1)$ and  $(v_2,\infty)$.
\end{enumerate}

\end{proof}

%% SECTION ad=1 and bd=1

\subsection{Space of solutions in case: $ad=1$ and $bd=1$} In this section we describe the values of $(H,C)$ that provide complete immersions in $\sf$ when $ad=1$ and $bd=1$. This case includes  cmc hypersurfaces in the $(n+1)-$dimensional unit sphere.
In this case the differential equation for $g$ reduces to

$$g'(t)^2=f(g)=C - g^2 - H^2 g^2 - g^{2 - 2 n} - 2 H g^{2 - n}$$

In this case, for any $H$, the function $f$ has only one critical point $q_1(H)$. We have,

$$q_1,r_1:(-\infty,\infty)\to \R  $$

given by

$$ q_1=\left(\frac{\sqrt{h^2 n^2+4 n+4}-h (n-2)}{2 (n-1)}\right)^{-1/n}\com{and} r_1=\frac{n \left(h \sqrt{h^2 n^2+4 n-4}+h^2 n+2 n-2\right)}{2^{\frac{n-2}{n}} (n-1)^{\frac{2 (n-1)}{n}} \left(\sqrt{h^2 n^2+4 n-4}-h (n-2)\right)^{2/n}} $$

\begin{figure}[hbtp]
\begin{centering}\includegraphics[width=.4\textwidth]{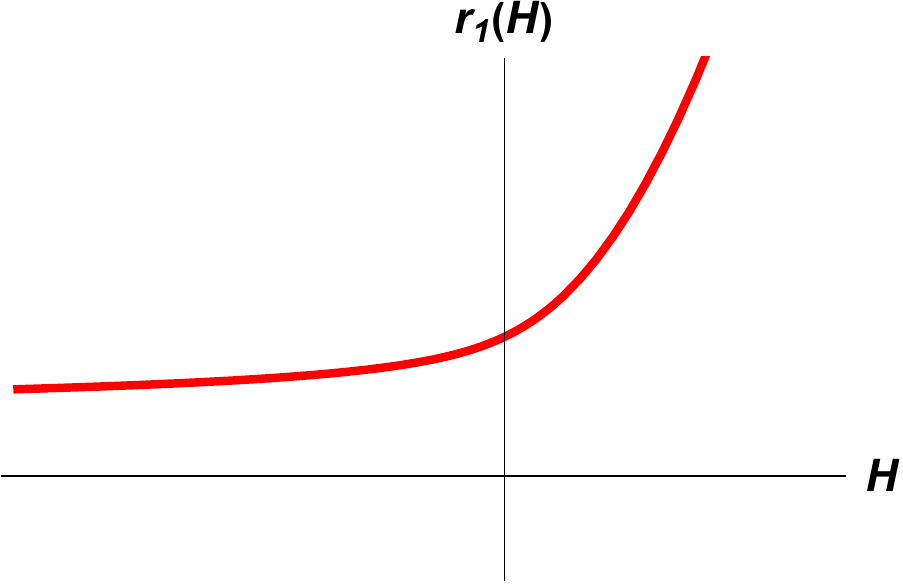} 
\end{centering}
\caption{Graph of the function $r_1(H)$.}\label{adeq1bedq1r}
\end{figure}

The following theorem summarizes the values of $H$ and $C$ that are achieved as values of non trivial complete hypersurfaces with two principal curvatures and cmc.

%Theorem sphere
\begin{thm} \label{Sph} Let $M\subset \sf$ be a complete hypersurface with two principal curvatures. Assume that $M$ has constant mean curvature $H$, that $C$ is the constant associated with $M$ and that the principal curvatures of $M$  are not constant.  Also assume that the sign of $H$ is determined by the condition $\lambda-\mu>0$.  If $n>2$, $ad=1$ and $bd=1$, then  
$H$ can be any real number, $C$ needs to be greater than $r_1(H)$ and $g(t)$ is a {\it type 1} solution.

\end{thm}

\begin{proof}
Once again the proof follows the same arguments as those on Theorem \ref{hyp}. In this case we have that the limit of the function $f$ when $v\to 0^+$ and $v\to \infty$ is negative infinity. We also have that $f$ has only one critical point $q_1(H)$ and also, we have that if $C\le r_1(H)$ then $f$ is never positive and when $C>r_1(H)$ then $f$ has two roots $v_1$ and $v_2$ and it is positive on the interval $(v_1,v_2)$.

\end{proof}

%% SECTION ad=-1 and bd=-1

\subsection{Space of solutions in case: $ad=-1$ and $bd=-1$} In this section we describe the values of $(H,C)$ that provide complete immersions in $\sf$ when $ad=-1$ and $bd=-1$. This case includes  space-like cmc hypersurfaces in the $(n+1)-$dimensional Anti de Sitter space.
In this case the differential equation for $g$ reduces to

$$g'(t)^2=f(g)=C + g^2 + H^2 g^2 + g^{2 - 2 n} + 2 H g^{2 - n}$$

In this case, for any $H$, the function $f$ has only one critical point $q_1(H)$. We have,

$$q_1,r_1:(-\infty,\infty)\to \R  $$

given by

$$ q_1=\left(\frac{\sqrt{h^2 n^2+4 n+4}-h (n-2)}{2 (n-1)}\right)^{-1/n}\com{and} r_1=\frac{-n \left(h \sqrt{h^2 n^2+4 n-4}+h^2 n+2 n-2\right)}{2^{\frac{n-2}{n}} (n-1)^{\frac{2 (n-1)}{n}} \left(\sqrt{h^2 n^2+4 n-4}-h (n-2)\right)^{2/n}} $$

\begin{figure}[hbtp]
\begin{centering}\includegraphics[width=.4\textwidth]{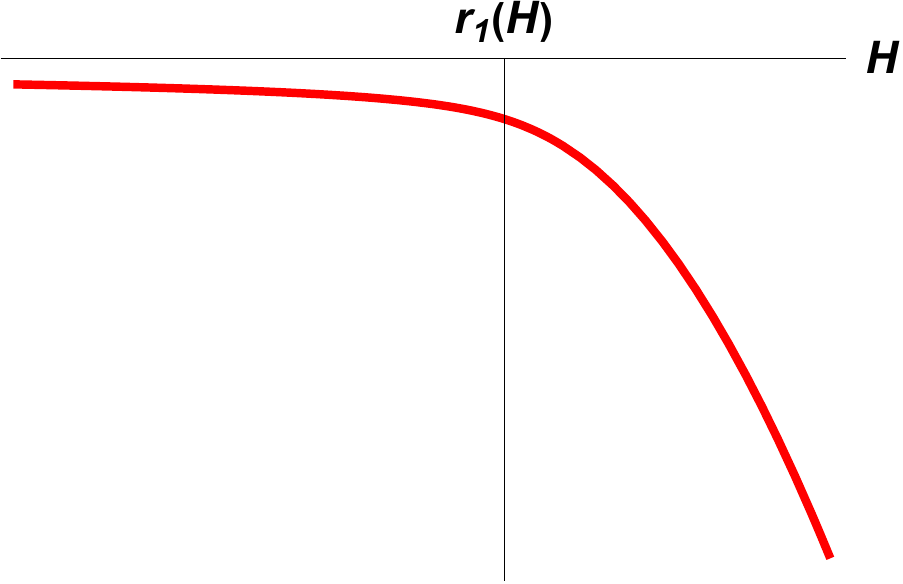} 
\end{centering}
\caption{Graph of the function $r_1(H)$. For this graph $n=12$}\label{adeqn1bdeqn1r}
\end{figure}

The following theorem summarizes the values of $H$ and $C$ that are achieved as values of non trivial complete hypersurfaces with two principal curvatures and cmc.

\begin{thm}\label{AntideSit} Let $M\subset \sf$ be a complete hypersurface with two principal curvatures. Assume that $M$ has constant mean curvature $H$, that $C$ is the constant associated with $M$ and that the principal curvatures of $M$  are not constant.  Also assume that the sign of $H$ is determined by the condition $\lambda-\mu>0$.  If $n>2$, $ad=-1$ and $bd=-1$, then  
$H$ can be any real number. Moreover  $C\le r_1(H)$ and $g(t)$ is a {\it type 2} solution when $C=r_1(H)$, otherwise $g(t)$ is a {\it type 3} solution.

\end{thm}

\begin{proof}
Once again the proof follows the same arguments as those on Theorem \ref{hyp}. In this case we have that the limit of the function $f$ when $v\to 0^+$ and $v\to \infty$ is positive infinity. We also have that $f$ has only one critical point $q_1(H)$ and also, we have that if $C> r_1(H)$ then $f$ is always positive, when $C=r_1(H)$ then $f$ one root $v_1$ which is a minimum and it is positive on the intervals $(0,v_1)$ and $(v_1,\infty)$.  Figure \ref{aeqn1beqn1ex1} shows particular choices for the parameters that produces an example of a function $f$ for this case. Finally, when $C<r_1(H)$. $f$ has two roots $v_1$ and $v_2$ and it is positive on the intervasl $(0,v_1)$ and $(v_2,\infty)$.

\end{proof}

%% SECTION a=0 and bd=-1

\subsection{Space of solutions in case: $a=0$ and $bd=-1$} In this section we describe the values of $(H,C)$ that provide complete immersions in $\sf$ when $a=0$ and $bd=-1$. This case includes  space-like cmc hypersurfaces in the $(n+1)-$dimensional Minkowski  space.
In this case the differential equation for $g$ reduces to

$$g'(t)^2=f(g)=C + H^2 g^2 + g^{2 - 2 n} + 2 H g^{2 - n}$$

In this case, for any $H\ne0$, the function $f$ has only one critical point $q_1(H)$. We have,

$$q_1,r_1:(-\infty,0)\cup(0,\infty)\to \R  $$

given by

$$q_1(H)=\begin{cases}  (-\frac{1}{h})^{\frac{1}{n}} & \hbox{if}\quad  H<0\\
\left(\frac{n-1}{h}\right)^{1/n} &\hbox{if}\quad  H>0
\end{cases},\quad r_1(H)=\begin{cases}  0 &\hbox{if}\quad  H\le0\\
-\frac{h^2 n^2 \left(\frac{n-1}{h}\right)^{2/n}}{(n-1)^2}&\hbox{if}\quad  H>0
\end{cases}  $$

\begin{figure}[hbtp]
\begin{centering}\includegraphics[width=.4\textwidth]{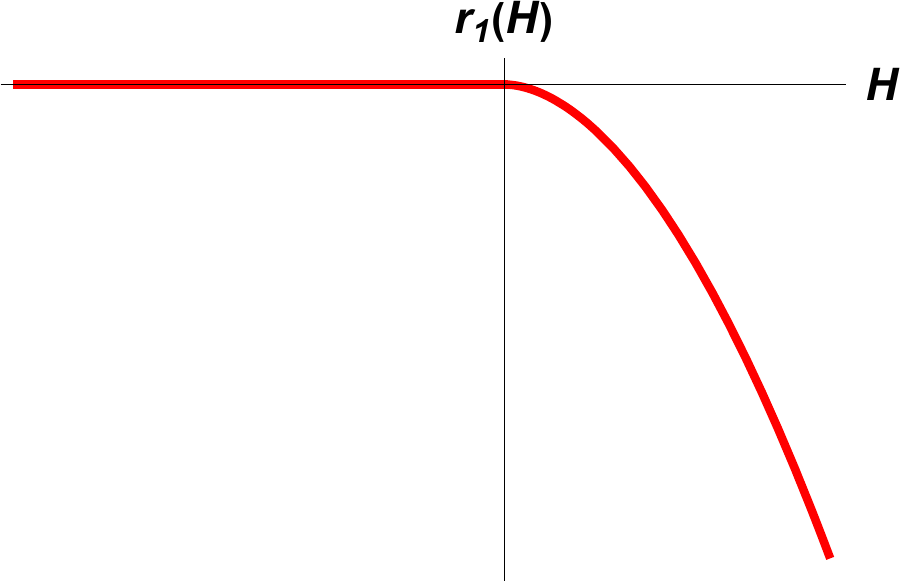} 
\end{centering}
\caption{Graph of the function $r_1(H)$.}\label{aeq0bdeqn1r}
\end{figure}

The following theorem summarizes the values of $H$ and $C$ that are achieved as values of non trivial complete hypersurfaces with two principal curvatures and cmc.

\begin{thm}\label{Mink} Let $M\subset \sf$ be a complete hypersurface with two principal curvatures. Assume that $M$ has constant mean curvature $H$, that $C$ is the constant associated with $M$ and that the principal curvatures of $M$  are not constant.  Also assume that the sign of $H$ is determined by the condition $\lambda-\mu>0$.  If $n>2$, $a=0$ and $bd=-1$, then  
$H$ can be any real number different from zero. Moreover  $C\le r_1(H)$ and $g(t)$ is a {\it type 2} solution when $C=r_1(H)$, otherwise $g(t)$ is a {\it type 3} solution.

\end{thm}

\begin{proof}
Once again the proof follows the same arguments as those on Theorem \ref{hyp}. In this case we have that the limit of the function $f$ when $v\to 0^+$ is positive infinity. When $H=0$, $f$ is always decreasing and therefore, either $f$ is positive everywhere or $f$ is positive on a interval $(0,v_1)$, in either case, $f$ dos not define a solution $g(t)$ in the whole real line. When $H\ne 0$,  $f$ has only one critical point $q_1(H)$ and also, we have that if $C> r_1(H)$ then $f$ is always positive, when $C=r_1(H)$ then $f$ one root $v_1$ which is a minimum and it is positive on the intervals $(0,v_1)$ and $(v_1,\infty)$.  Figure \ref{aeq0beqn1ex1} shows particular choices for the parameters that produces an example of a function $f$ for this case. Finally, when $C<r_1(H)$. $f$ has two roots $v_1$ and $v_2$ and it is positive on the intervals $(0,v_1)$ and $(v_2,\infty)$.

\end{proof}

%% SECTION a=0 and bd=1

\subsection{Space of solutions in case: $a=0$ and $bd=1$} In this section we describe the values of $(H,C)$ that provide complete immersions in $\sf$ when $a=0$ and $bd=1$. This case includes  cmc hypersurfaces in the $(n+1)-$dimensional Euclidean space.
In this case the differential equation for $g$ reduces to

$$g'(t)^2=f(g)=C - H^2 g^2 - g^{2 - 2 n} - 2 H g^{2 - n}$$

In this case, for any $H\ne0$, the function $f$ has only one critical point $q_1(H)$. We have,

$$q_1,r_1:(-\infty,0)\cup(0,\infty)\to \R  $$

given by

$$q_1(H)=\begin{cases}  (-\frac{1}{h})^{\frac{1}{n}} & \hbox{if}\quad  H<0\\
\left(\frac{n-1}{h}\right)^{1/n} &\hbox{if}\quad  H>0
\end{cases},\quad r_1(H)=\begin{cases}  0 &\hbox{if}\quad  H\le0\\
\frac{h^2 n^2 \left(\frac{n-1}{h}\right)^{2/n}}{(n-1)^2}&\hbox{if}\quad  H>0
\end{cases}  $$

\begin{figure}[hbtp]
\begin{centering}\includegraphics[width=.4\textwidth]{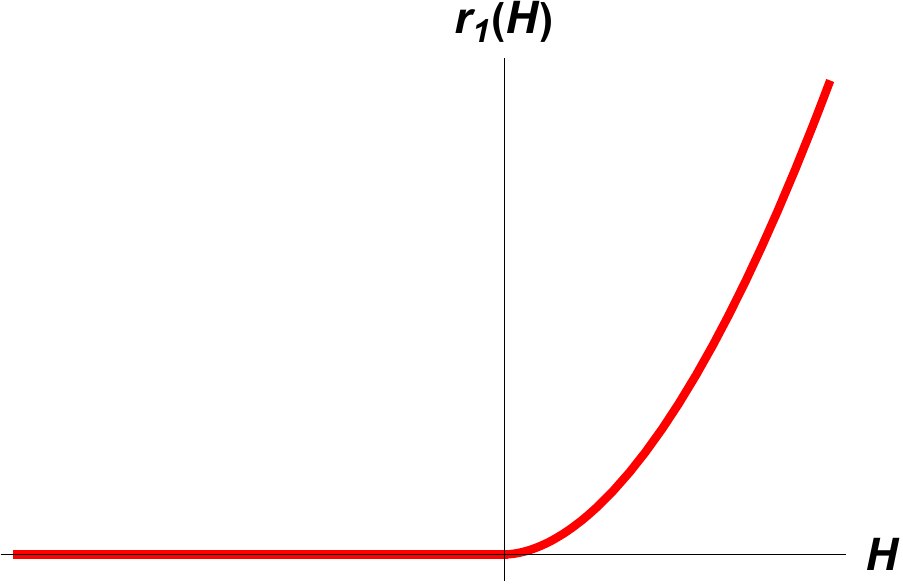} 
\end{centering}
\caption{Graph of the function $r_1(H)$. For this graph $n=12$.}\label{aeq0bdeq1r}
\end{figure}

The following theorem summarizes the values of $H$ and $C$ that are achieved as values of non trivial complete hypersurfaces with two principal curvatures and cmc.

\begin{thm}\label{Euc} Let $M\subset \sf$ be a complete hypersurface with two principal curvatures. Assume that $M$ has constant mean curvature $H$, that $C$ is the constant associated with $M$ and that the principal curvatures of $M$  are not constant.  Also assume that the sign of $H$ is determined by the condition $\lambda-\mu>0$.  If $n>2$, $a=0$ and $bd=1$, then  
$H$ can be any real number. Moreover, when $H\ne0$, $C> r_1(H)$ and $g(t)$ is a {\it type 1} solution and when $H=0$, $C>0$ and   $g(t)$ is a {\it type 3} solution.

\end{thm}

\begin{proof}
Once again the proof follows the same arguments as those on Theorem \ref{hyp}. In this case we have that the limit of the function $f$ when $v\to 0^+$ is negative infinity. When $H=0$, $f$ is always increasing and, either $f$ is negative everywhere or $f$ is positive on a interval $(v_1,\infty)$. When $H\ne 0$,  $f$ has only one critical point $q_1(H)$ and also, we have that if $C\le r_1(H)$ then $f$ is never positive, and  when $C> r_1(H)$. $f$ has two roots $v_1$ and $v_2$ and it is positive on the intervals $(v_1,v_2)$.

\end{proof}

%%  THE SHAPE OPERATOR.

\section{The norm square of the second fundamental form} \label{s2 results}

We start this section by computing  the norm of the second fundamental form, the norm of the traceless  second fundamental form and the scalar curvature in terms of the function $g(t)$. As expected, there are linear relation between them and therefore when we find estimates for  one of them, we find estimates for all of them. For convenience, we present estimates on the norm of the traceless second fundamental form.

\subsubsection{The norm of the second fundamental form and the traceless second fundamental form for riemannian hypersurfaces} Since $\lambda=H+g^{-n}$ and $\mu=H-(n-1) H$, we can easily deduce that the square of the norm of the shape operator is given by 

$$|A|^2=(n-1)\lambda^2+\mu^2=n(n-1)g^{-2n}+nH^2$$

If we consider the traceless second fundamental form, which is defined by $\Phi=A-nI$ where $I$ is the identity operator, then we have that 

$$ |\Phi|^2= |A|^2-nH^2= n(n-1)g^{-2n}$$

\subsubsection{The scalar curvature for riemannian hypersurfaces} As shown in Section \ref{basic equations}, the curvature tensor has the following expression

$$
\<R(w_1,w_2)w_3,w_4\>=a(\<w_1,w_3\>\<w_2,w_4\>-\<w_1,w_4\>\<w_2,w_3\>)+b(\<A(w_3),w_1\>\<A(w_2),w_4\>-\<A(w_3),w_2\>\<A(w_1),w_4\>)
$$

in the case that the $M$ is Riemmanian, and the orhonormal frame  $\{e_1,\dots, e_n\}$ is defined as in Definition (\ref{two principal curvatures general}), then $\<e_i,e_j\>=\delta_{ij}$ and then the Ricci curvature is given by 

\begin{eqnarray*}
Ric(e_1)&=&\frac{1}{n-1}(R(e_1,e_2,e_1,e_2)+\dots+R(e_1,e_{n-1},e_1,e_{n-1})+R(e_1,e_n,e_1,e_n))\\
  &=& \frac{1}{n-1}((n-2) (a+b\lambda^2)+a+b\lambda \mu)=a+\frac{n-2}{n-1}b\lambda^2+\frac{b}{n-1} \lambda\mu
\end{eqnarray*} 
 
 Likewise, for any $i<n$ we have that 

$$Ric(e_i)=a+\frac{n-2}{n-1}b\lambda^2+\frac{b}{n-1} \lambda\mu\com{and} Ric(e_n)=a+b\lambda\mu
$$

Therefore we have that the scalar curvature $S$ is given by 

$$S=\frac{1}{n}\sum_{j=1}^n Ricc(e_j) = a+\frac{b}{n}((n-1)\lambda^2+2\lambda\mu)=a+b H^2-bg^{-2n}$$

Our complete classification of hypersurfaces with two principal curvatures allows us to exactly find the superior and the inferior of the norm of the second fundamental form in terms of the positive  roots of a polynomial that depends on the constant $C$ of the cmc hypersurfaces.  With the intension of providing estimates on the norm of the traceless second fundamental form that are independent of $C$, we decided to improved and extend the main result by Alias and Garcia-Mart\'{\i}nez in \cite{A-M}. The proof of our estimate easily follows from the  following remark,

\begin{rem} In this remark, $v_1$, $v_2$ and $v_3$ denote roots of an explicit  polynomial given in terms of $H$, $C$ and $n$. For any complete cmc hypersurface with exactly two (non-contant) principal curvatures, the function  $g(t)=\omega(\gamma(t))$ is a function of type 1,2,3,4 or 5. Moreover we have:
\begin{itemize}
\item
If $g(t)$ is of type 1, then $g(t)$ has range $[v_1,v_2]$ where $v_1$ and $v_2$ with $0<v_1<q(H)<v_2$ where $q(H)$ can be easily computed. Therefore we can easily say that $\max(g(t))>q(H)>\min(g(t))$.

\item
If $g(t)$ is of type 2, then $g(t)$ has range $(q(H),\infty)$  where $q(H)$ can be easily computed. In this case the infimum of $g(t)$ is $q(H)$ and it is never achieved and the the infimum of $g(t)^{-n}$ is zero.

\item
If $g(t)$ is of type 3, then $g(t)$ has range $[v_3,\infty)$ . In this case  $\min(g(t))=v_3$  and  the infimum of $g(t)^{-n}$ is zero.

\item
If $g(t)$ is of type 4, then $g(t)$ has range $[v_1,q(H))$, where $q(H)$ can be easily computed. In this case  $\min(g(t))=v_1$ and $\sup(g(t))=q(H)$, this superior is never achieved.
\item
If $g(t)$ is of type 5, then $g(t)$ has range $(q(H),v_1]$, where $q(H)$ can be easily computed. In this case  $\max(g(t))=v_1$ and $\inf(g(t))=q(H)$, this inferior is never achieved.

\end{itemize}
\end{rem}

Recall that the sign of $H$ is important in our classification due to the following remark.

\begin{rem}\label{defH} Let us assume that $M\subset \sf$ is  hypersurfaces with two principal curvatures, $\lambda$ and  $\mu$ with multiplicity $n-1$ and $1$ respectively and with $n>2$. There is a unique way to define the Gauss map so that the function $\lambda-\mu$ is positive. This choice of the Gauss map defines a unique $H$. In this way the hypersurfaces with positive $H$ and negative $H$ are different. 
\end{rem}

In order to present the result it is convenient to define the following three functions

\begin{mydef} We define

\begin{eqnarray*}
b_1(H)&=&\frac{\sqrt{n} \left(\sqrt{H^2 n^2+4ab(n-1)}-H (n-2)\right)}{2 \sqrt{(n-1)}}\\
b_2(H)&=& -\frac{n \left(\sqrt{H^2 n^2+4ab(n-1)}+H (n-2)\right)}{2 \sqrt{(n-1) n}}\\
b_3(H)&=&\begin{cases}  -\sqrt{n(n-1)} H & \hbox{if}\quad  H<0\\
& \\
 \frac{\sqrt{n}}{\sqrt{n-1}} H &\hbox{if}\quad  H\ge0
\end{cases}\\
\end{eqnarray*}

\end{mydef}

\begin{rem} Notice that the expression $b_1(H)$ agrees with the expression $\beta_{|H|,c}$ in the paper \cite{A-M} when $H<-1$ but $ b_1(H)<\beta_{|H|,c}$ when $H>1$. For this reason the bounds in this paper in the case $H>1$ represent an improvement of those in the paper \cite{A-M}.
\end{rem}

%% Shape Hyperbolic space

\subsection{Hypersurfaces in Hyperbolic spaces: case $ad=-1$ and $bd=1$}

\begin{thm} \label{hypshape} Let $M\subset \sf$ be a complete hypersurface with two principal curvatures, one of them with multiplicity one. Assume that $M$ has constant mean curvature $H$.  Also assume that the sign of $H$ is determined by the condition $\lambda-\mu>0$. If $n>2$, $ad=-1$ and $bd=1$ then,

\begin{enumerate}
\item
If $|H|>1$, then $ |\Phi|$ is periodic and $0<\min( |\Phi|)\le b_1(H)\le \max( |\Phi|)$ with equality holding in any of the two inequalities  if and only if the principal curvatures are constant. 

\item
When $H=-1$, then either 
\begin{itemize}
\item
$\inf( |\Phi|)=0$ and $b_1(H)< \sup( |\Phi|)$ or 
\item
$\inf( |\Phi|)>0$,  $|\Phi|$ is periodic and $\min( |\Phi|)\le b_1(H)\le \max( |\Phi|)$ with equality holding in any of the two inequalities  if and only if the principal curvatures are constant. 
\end{itemize}
Both types of examples exist.  

\item
When $-1<H<-\frac{2\sqrt{n-1}}{n} $ then either
\begin{itemize}
\item
$\inf( |\Phi|)=0$ and $\sup (|\Phi|)\le b_1(H)$

\item
$b_2(H)\le \inf( |\Phi|)\le b_1(H)\le \sup( |\Phi|)$. In this case there exist examples with non-constant principal curvatures that satisfy $b_2(H)= \inf( |\Phi|)$. Moreover, $b_1(H)= \sup( |\Phi|)$ holds  if and only if the principal curvatures are constant.

 \end{itemize}
 
 There are examples with $\inf( |\Phi|)>0$ and $\inf( |\Phi|)=0$
 \item
When $-\frac{2\sqrt{n-1}}{n}\le H\le1$, then $\inf( |\Phi|)=0$. There are examples with non constant principal curvatures with $H=-\frac{2\sqrt{n-1}}{n}$ that satisfy $\sup( |\Phi|)=b_2(H)=b_1(H)$.
\end{enumerate} 
\end{thm}

\begin{proof} The proof follows directly from Theorem \ref{hyp}. 

\end{proof}

Figures \ref{aeqn1beq1ex3}, \ref{aeqn1beq1ex2}, \ref{aeqn1beq1ex4} and \ref{aeqn1beq1ex1}  show some the graph of the function $|\Phi|=\sqrt{n(n-1)} \omega^{-n}(t)$ for some particular interesting hypersurfaces discussed in this Theorem.

\begin{rem} If we use the same values for $C$ and $H$ as in the Examples displayed in Figures \ref{aeqn1beq1ex3} and \ref{aeqn1beq1ex2} and the initial condition $g(0)=q_2(H)=\sqrt[4]{\frac{30}{9-2 \sqrt{6}}}$, 
then we would obtain an example with constant principal curvatures. In  \cite{W5}, Theorem 3.1, it is wrongly stated that the only examples with $C=r_2(H)$ are those with constant principal curvatures.

\end{rem}

%
%shape de Sitter
%

\subsection{Hypersurfaces of the de Sitter Spaces: case $ad=1$ and $bd=-1$} 
\begin{thm} \label{deSittershape} Let $M\subset \sf$ be a complete hypersurface with two principal curvatures, one of them with multiplicity one. Assume that $M$ has constant mean curvature $H$.  Also assume that the sign of $H$ is determined by the condition $\lambda-\mu>0$. If $n>2$, $ad=1$ and $bd=-1$, 

\begin{enumerate}
\item
$H$ cannot take values between $[-2\frac{\sqrt{n-1}}{n},1]$
\item
When either $|H|>1$ or $H=-1$, then $\inf(|\Phi|)=0$ and $\sup(|\Phi|)\le b_1(H)$. There are examples with non constant principal curvature that satisfy $\sup(|\Phi|) = b_1(H)$.

\item
When $-1<H<-\frac{2\sqrt{n-1}}{n} $, then $\inf( |\Phi|)\le b_2(H)\le \sup( |\Phi|)\le b_1(H)$. Moreover, $\inf( |\Phi|)=b_2(H)$ if and only if the principal curvatures are constant. Moreover, there are examples with non constant principal curvatures that satisfy $\sup( |\Phi|)=b_1(H)$. 
%Theorem 4.2 (3) is incorrect. There are non-periodic solutions.

\end{enumerate} 

\end{thm}

\begin{proof} The proof follows directly from Theorem \ref{deSit}. 

\end{proof}

Figures \ref{aeq1beqn1ex1}, \ref{aeq1beqn1ex2} and \ref{aeq1beqn1ex3}  show some the graph of the function $|\Phi|=\sqrt{n(n-1)} \omega^{-n}(t)$ for some particular interesting hypersurfaces discussed in this Theorem.

\begin{rem} The Example displayed in Figure \ref{aeq1beqn1ex3} shows an example where $g(t)$ is not periodic, $H$ lies between $-1$ and $-2\frac{\sqrt{n-1}}{n}$. In \cite{W3}, Theorem 4.2 part (3), it is wrongly stated that for all the examples with $H$ between $-1$ and $-2\frac{\sqrt{n-1}}{n}$ $g(t)$  must be periodic.

\end{rem}

%
% %% Shape sphere
%
\subsection{Hypersurfaces of the sphere: case $ad=1$ and $bd=1$} 

\begin{thm} \label{Sphereshape} Let $M\subset \sf$ be a complete hypersurface with two principal curvatures, one of them with multiplicity one. Assume that $M$ has constant mean curvature $H$.  Also assume that the sign of $H$ is determined by the condition $\lambda-\mu>0$.  If $n>2$, $ad=1$ and $bd=1$. We have that for any $H$,  $\inf( |\Phi|)\le b_1(H)\le \sup( |\Phi|)$ with equality holding in any of the two inequalities  if and only if the principal curvatures are constant. 

\end{thm}

\begin{proof} The proof follows directly from Theorem \ref{Sph}. 

\end{proof}
%%
%%ANTIDESITTER
%%
\subsection{Hypersurfaces of the anti-de Sitter space: case $ad=-1$ and $bd=-1$}

\begin{thm}  \label{Anti-deSittershape} Let $M\subset \sf$ be a complete hypersurface with two principal curvatures, one of them with multiplicity one. Assume that $M$ has constant mean curvature $H$.  Also assume that the sign of $H$ is determined by the condition $\lambda-\mu>0$.  If $n>2$, $ad=-1$ and $bd=-1$. We have that for any $H$, $\inf( |\Phi|)=0$ and  $ \sup(|\Phi|)\le b_1(H)$. There are examples with non constant principal curvature that satisfy $\sup(|\Phi|)\le b_1(H)$.

\end{thm}

\begin{proof} The proof follows directly from Theorem \ref{AntideSit}. 
\end{proof}

Figures \ref{aeqn1beqn1ex1}  shows  the graph of the function $|\Phi|=\sqrt{n(n-1)} \omega^{-n}(t)$ for a particular interesting hypersurface discussed in this Theorem.

\begin{rem} If we use the same values for $C=r_1(H)$ and $H$ as in the Example displayed in Figure \ref{aeqn1beqn1ex1} and the initial condition $g(0)=q_1(H)=\sqrt[4]{\frac{3}{\sqrt{19}+2}}$, 
then we would obtain an example with constant principal curvatures. In  \cite{W1}, Lemma 4.2, it is wrongly stated that the only examples with $C=r_1(H)$ are those with constant principal curvatures.

\end{rem}
%%
%%Euclidean
%%

\subsection{Hypersurfaces of the Euclidean space: case $a=0$ and $bd=1$} 

\begin{thm} \label{Euclideanshape} Let $M\subset \sf$ be a complete hypersurface with two principal curvatures, one of them with multiplicity one. Assume that $M$ has constant mean curvature $H$.  Also assume that the sign of $H$ is determined by the condition $\lambda-\mu>0$.  If $n>2$, $a=0$ and $bd=1$. We have that for any $H\ne0$,   $0<\inf( |\Phi|)\le b_3(H)\le \sup( |\Phi|)$ with equality holding in any of the two inequalities  if and only if the principal curvatures are constant. Moreover, if $H=0$ then,  $\inf( |\Phi|)=0$

\end{thm}

\begin{proof} The proof follows directly from Theorem \ref{Euc}. 

\end{proof}
%%
%%Minkowski
%%

\subsection{Hypersurfaces of the Minkowski space: case $a=0$ and $bd=-1$} 

\begin{thm} \label{Minkowskianshape} Let $M\subset \sf$ be a complete hypersurface with two principal curvatures, one of them with multiplicity one. Assume that $M$ has constant mean curvature $H$.  Also assume that the sign of $H$ is determined by the condition $\lambda-\mu>0$.  If $n>2$, $a=0$ and $bd=-1$ then, for any $H$,   $\inf( |\Phi|)=0$ and  $\sup( |\Phi|)\le b_3(H)$.  There are examples with non constant principal curvature that satisfy $\sup(|\Phi|) =  b_3(H)$.
\end{thm}

\begin{proof} The proof follows directly from Theorem \ref{Mink}. 

\end{proof}

\begin{rem} The example displayed in Figure \ref{aeq0beqn1ex1} shows cmc hypersurface with two non-constant principal curvatures with $C=0$.  In  \cite{W2}, Theorem 4.2, it is wrongly stated that these type of examples (with $C=0$) do not exist.

\end{rem}

\section{Images showing some particular examples}

\begin{figure}[hbtp]
\begin{centering}\includegraphics[width=.4\textwidth]{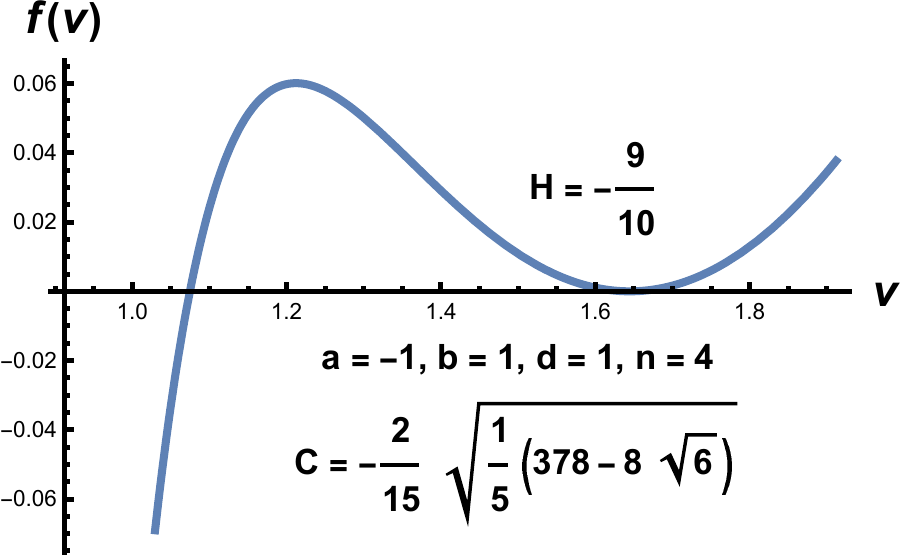} \includegraphics[width=.4\textwidth]{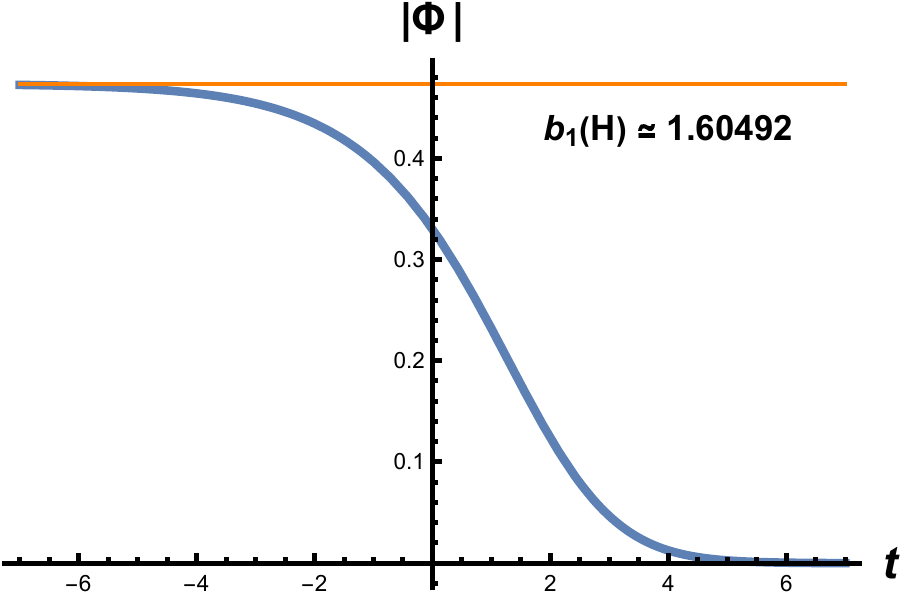}
\end{centering}
\caption{The second image shows the graph of the function $|\Phi |=\sqrt{12}\, g(t)^{-4}$ for the cmc hypersurface in the 5 dimensional hyperbolic space $SF(5,0,-1)$. For this example $C=r_2(H)=-\frac{2}{15}  \sqrt{\frac{1}{5} \left(378-8 \sqrt{6}\right)}$, $H=-9/10$, $\sup(|\Phi |)=b_1(H)=\frac{2 \sqrt{6}+9}{5 \sqrt{3}}\approx 1.60492$. The image on the left shows the graph of the function $f(v)$. Recall that $(g^\prime(t))^2=f(g(t))$. We used the initial condition $g(0)=1.8$.
}\label{aeqn1beq1ex3}
\end{figure}

\begin{figure}[hbtp]
\begin{centering}\includegraphics[width=.4\textwidth]{aeqn1beq1ffex2.eps} \includegraphics[width=.4\textwidth]{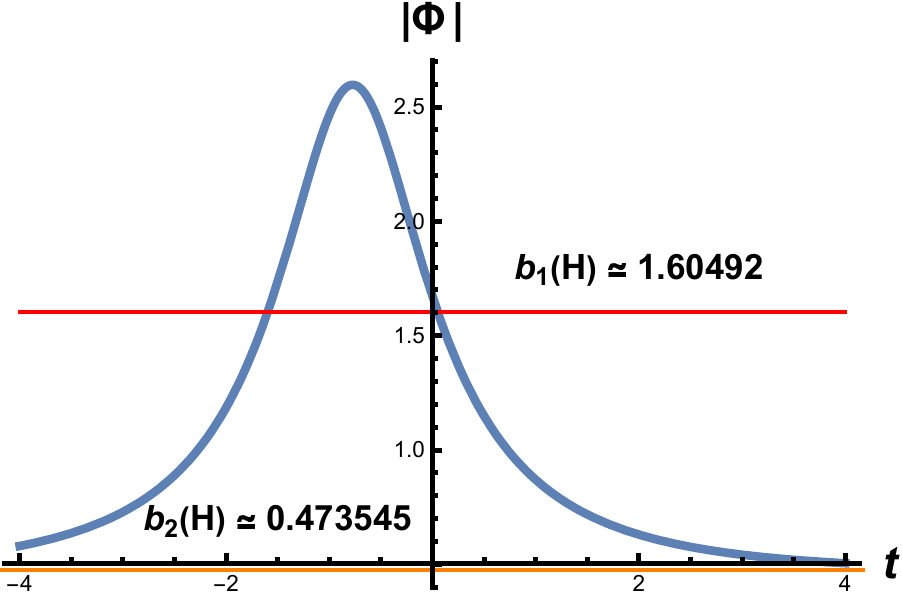}
\end{centering}
\caption{The second image shows the graph of the function $|\Phi |=\sqrt{12}\, g(t)^{-4}$ for the cmc hypersurface in the 5 dimensional hyperbolic space $SF(5,0,-1)$. For this example $C=r_2(H)=-\frac{2}{15}  \sqrt{\frac{1}{5} \left(378-8 \sqrt{6}\right)}$, $H=-9/10$, $\inf(|\Phi |)=b_2(H)=\frac{9-2 \sqrt{6}}{5 \sqrt{3}}\approx 0.473545$. The image on the left shows the graph of the function $f(v)$. Recall that $(g^\prime(t))^2=f(g(t))$. We used the initial condition $g(0)=1.2$.
}\label{aeqn1beq1ex2}
\end{figure}

\begin{figure}[hbtp]
\begin{centering}\includegraphics[width=.4\textwidth]{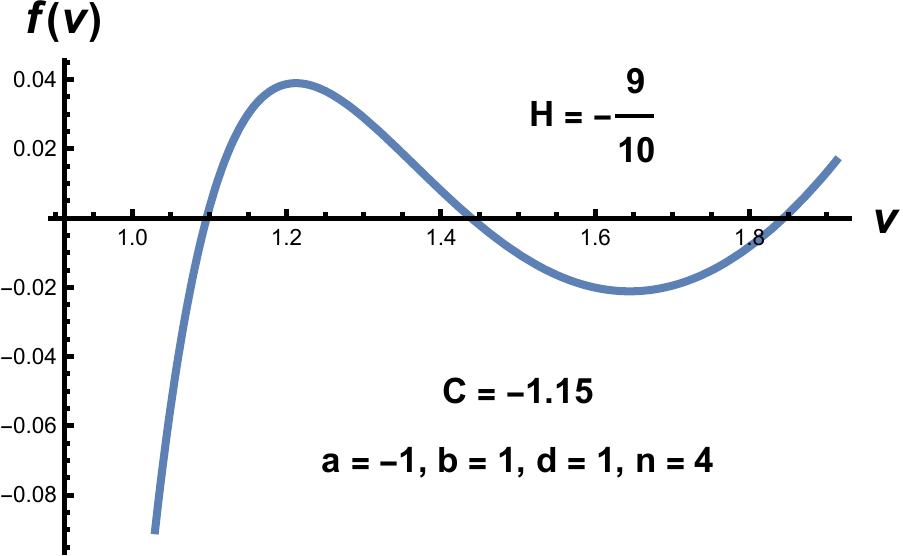} \includegraphics[width=.4\textwidth]{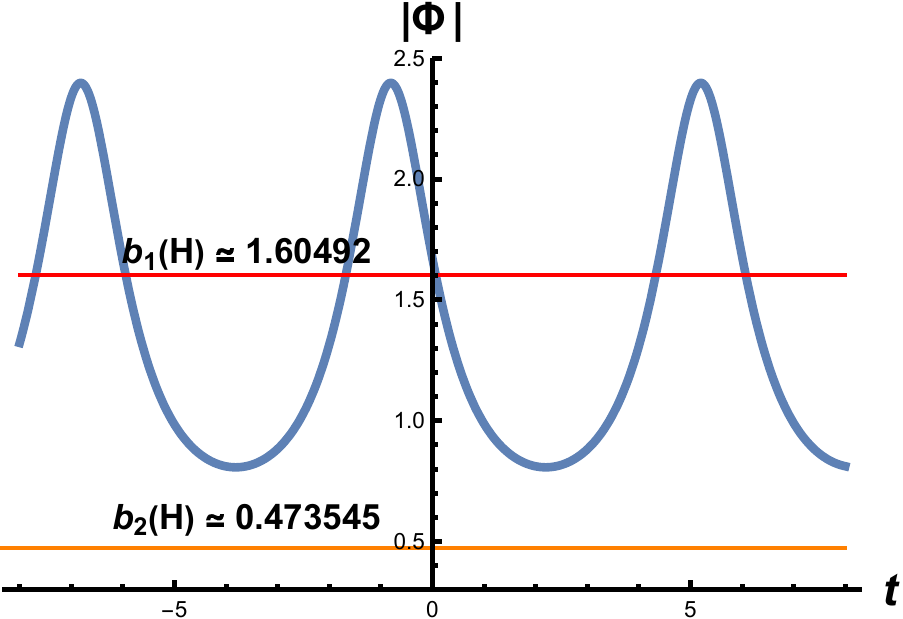}
\end{centering}
\caption{The second image shows the graph of the function $|\Phi |=\sqrt{12}\, g(t)^{-4}$ for the cmc hypersurface in the 5 dimensional hyperbolic space $SF(5,0,-1)$. For this example $C=-1.15$, $H=-9/10$. We used the initial condition $g(0)=1.2$.
}\label{aeqn1beq1ex4}
\end{figure}

\begin{figure}[hbtp]
\begin{centering}\includegraphics[width=.4\textwidth]{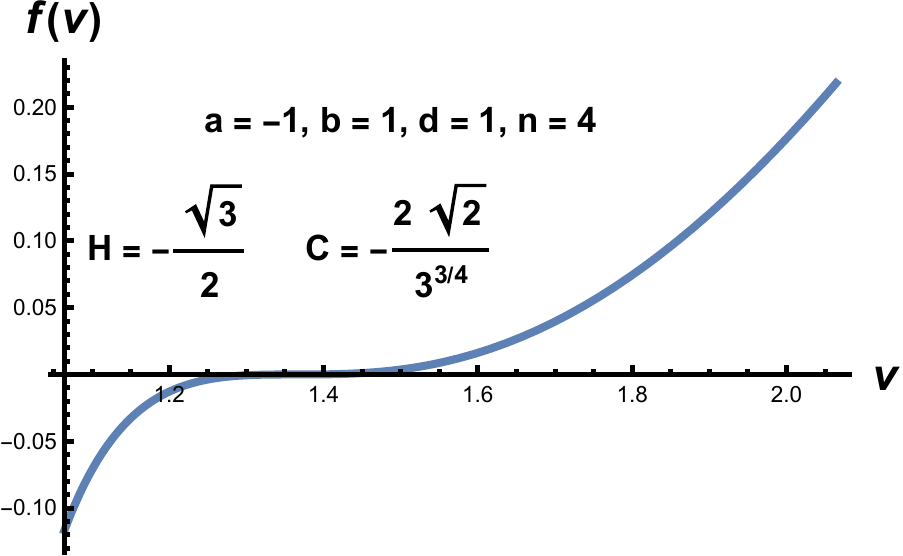} \includegraphics[width=.4\textwidth]{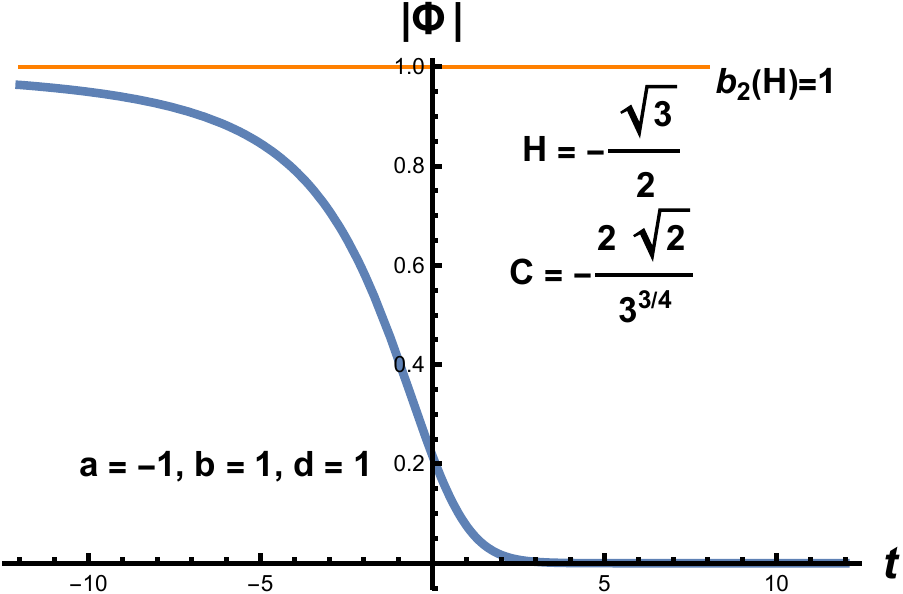}
\end{centering}
\caption{The second image shows the graph of the function $|\Phi |=\sqrt{12}\, g(t)^{-4}$ for the cmc hypersurface in the 5 dimensional hyperbolic space $SF(5,0,-1)$. For this example $C=r_1(H)$, $H=-2\frac{\sqrt{n-1}}{n}$, $\sup(|\Phi |)=b_1(H)=b_2(H)=1$. The image on the left shows the graph of the function $f(v)$. Recall that $(g^\prime(t))^2=f(g(t))$. We used the initial condition $g(0)=2$.
}\label{aeqn1beq1ex1}
\end{figure}

%%de Sitter

\begin{figure}[hbtp]
\begin{centering}\includegraphics[width=.4\textwidth]{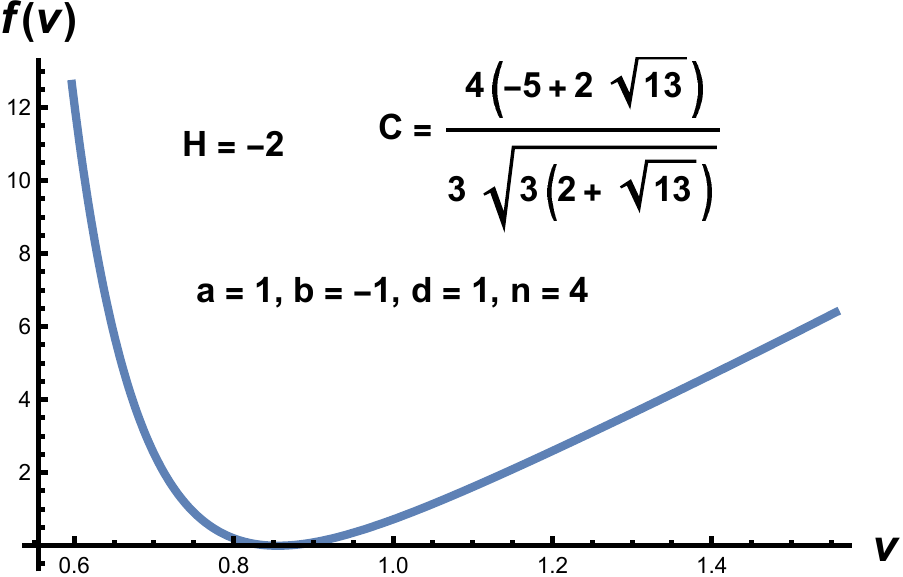} \includegraphics[width=.4\textwidth]{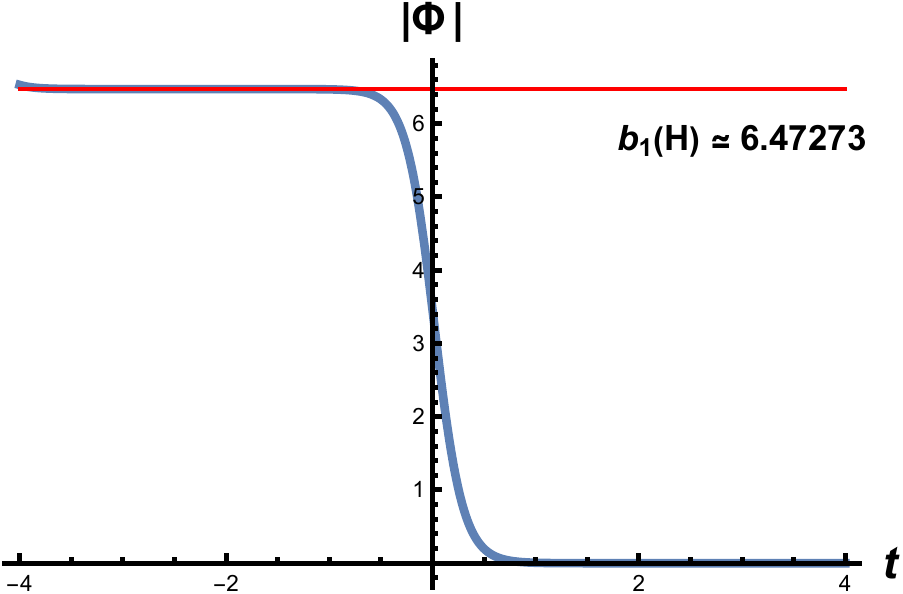}
\end{centering}
\caption{The second image shows the graph of the function $|\Phi |=\sqrt{12}\, g(t)^{-4}$ for the cmc hypersurface in the 5 dimensional de Sitter space $SF(5,1,1)$. For this example  $H=-2$,  $C=r_1(H)=\frac{4 \left(2 \sqrt{13}-5\right)}{3 \sqrt{3 \left(\sqrt{13}+2\right)}}$, $\sup(|\Phi |)=b_1(H)=\frac{2 \left(\sqrt{13}+2\right)}{\sqrt{3}}\approx 6.47273$. The image on the left shows the graph of the function $f(v)$. Recall that $(g^\prime(t))^2=f(g(t))$. We used the initial condition $g(0)=1$.
}\label{aeq1beqn1ex1}
\end{figure}

\begin{figure}[hbtp]
\begin{centering}\includegraphics[width=.4\textwidth]{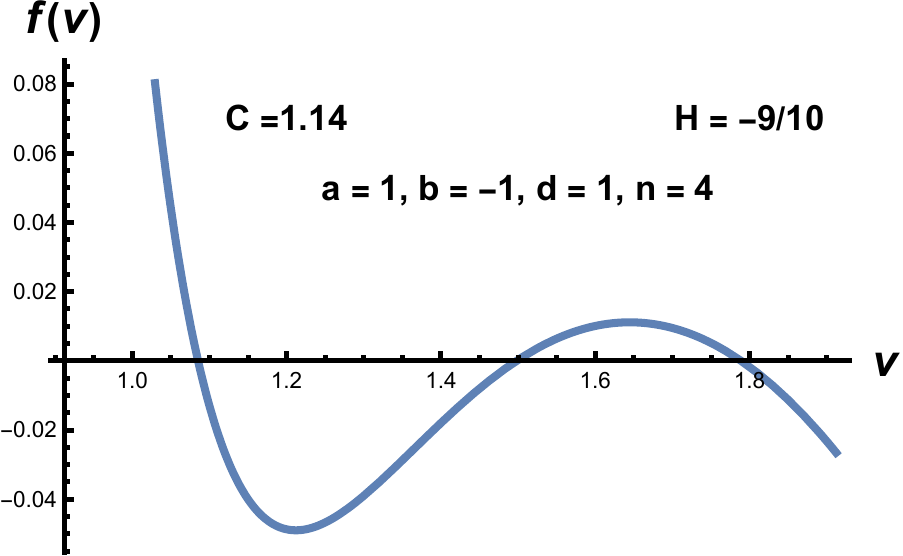} \includegraphics[width=.4\textwidth]{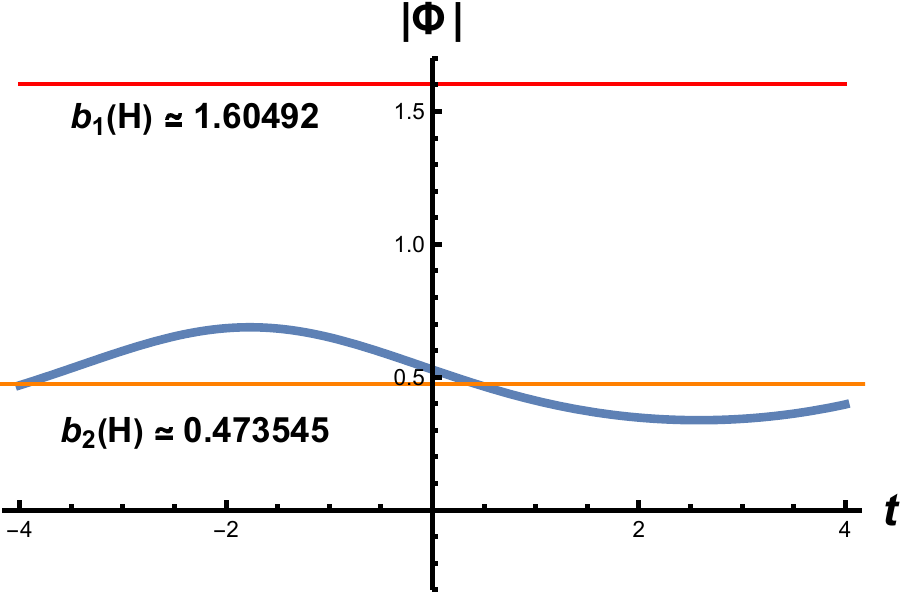}
\end{centering}
\caption{The second image shows the graph of the function $|\Phi |=\sqrt{12}\, g(t)^{-4}$ for the cmc hypersurface in the 5 dimensional de Sitter space $SF(5,1,1)$. For this example  $H=-9/10$,  $C=1.14$. The image on the left shows the graph of the function $f(v)$. Recall that $(g^\prime(t))^2=f(g(t))$. We used the initial condition $g(0)=1.6$.
}\label{aeq1beqn1ex2}
\end{figure}

\begin{figure}[hbtp]
\begin{centering}\includegraphics[width=.4\textwidth]{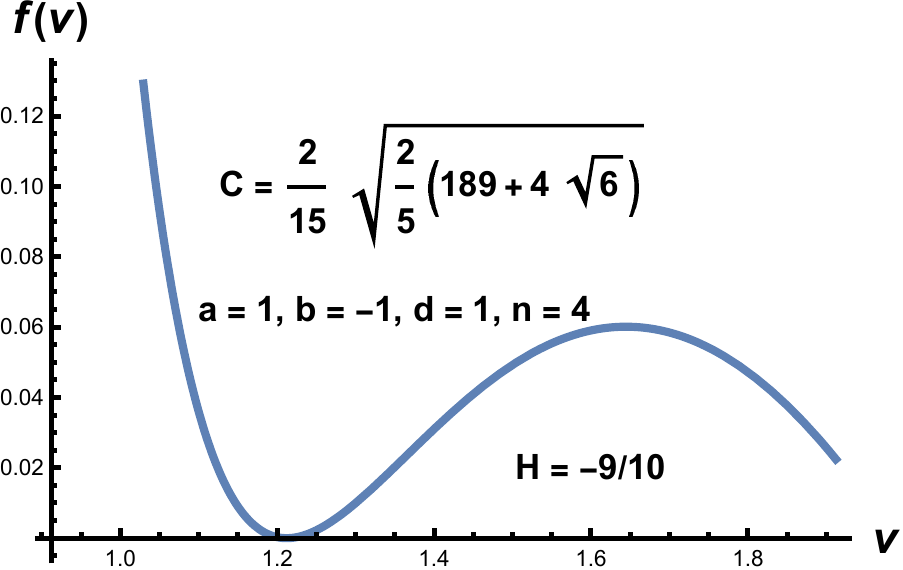} \includegraphics[width=.4\textwidth]{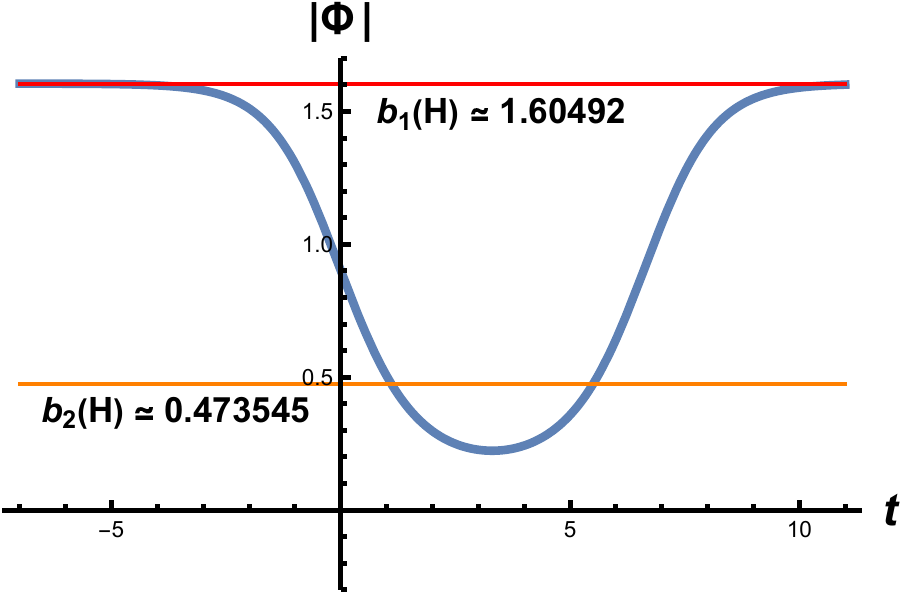}
\end{centering}
\caption{The second image shows the graph of the function $|\Phi |=\sqrt{12}\, g(t)^{-4}$ for the cmc hypersurface in the 5 dimensional de Sitter space $SF(5,1,1)$. For this example  $H=-9/10$,  $C=r_1(H)=\frac{2}{15} \sqrt{\frac{2}{5} \left(4 \sqrt{6}+189\right)}$, $\sup(|\Phi |)=b_1(H)=\frac{1}{5} \sqrt{12 \sqrt{6}+35}\approx 1.60492$. The image on the left shows the graph of the function $f(v)$. Recall that $(g^\prime(t))^2=f(g(t))$. We used the initial condition $g(0)=1.4$.
}\label{aeq1beqn1ex3}
\end{figure}

\begin{figure}[h]
\begin{centering}\includegraphics[width=.4\textwidth]{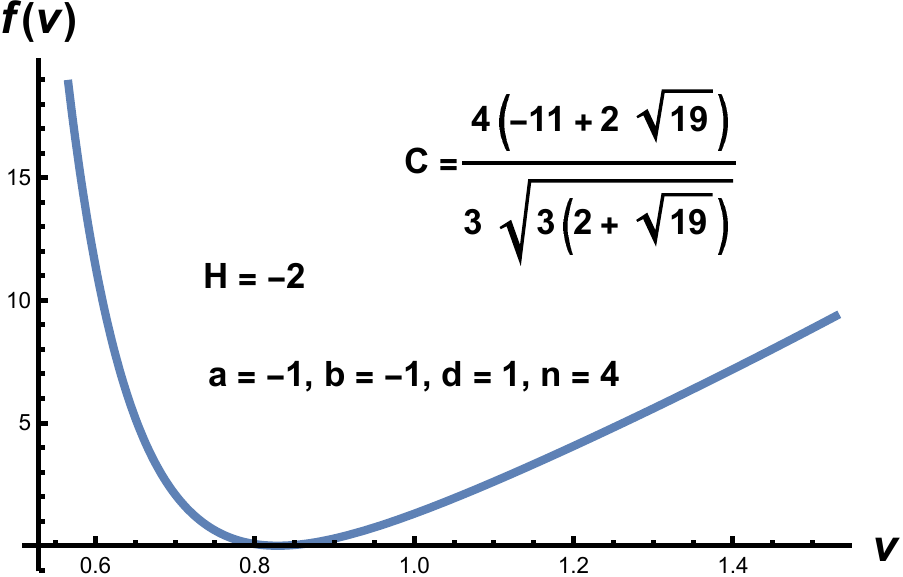} \includegraphics[width=.4\textwidth]{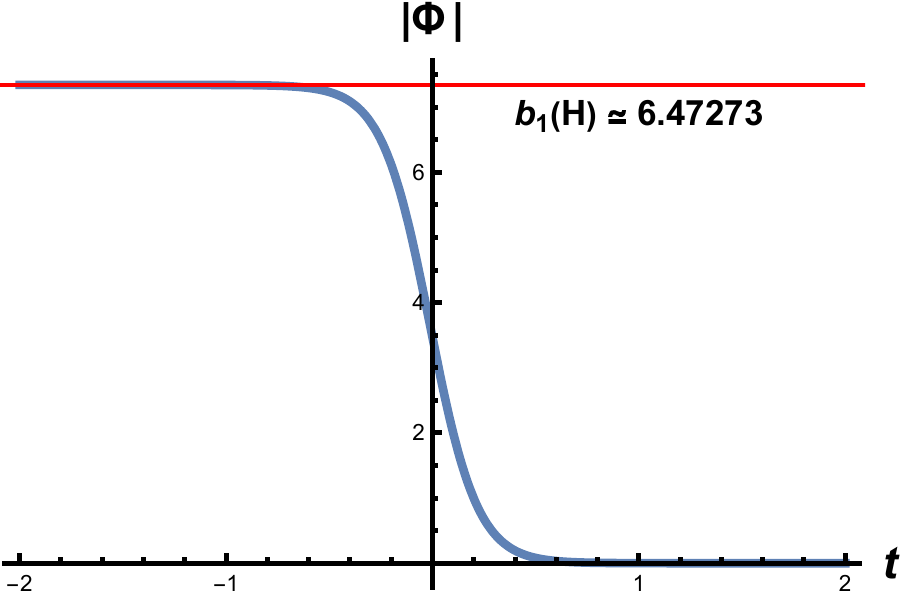}
\end{centering}
\caption{The second image shows the graph of the function $|\Phi |=\sqrt{12}\, g(t)^{-4}$ for the cmc hypersurface in the 5 dimensional Anti de Sitter space $SF(5,1,-1)$. For this example  $H=-2$,  $C=r_1(H)=\frac{4 \left(2 \sqrt{19}-11\right)}{3 \sqrt{3 \left(\sqrt{19}+2\right)}}$, $\sup(|\Phi |)=b_1(H)=\frac{2 \left(\sqrt{13}+2\right)}{\sqrt{3}}\approx 6.47273$. The image on the left shows the graph of the function $f(v)$. Recall that $(g^\prime(t))^2=f(g(t))$. We used the initial condition $g(0)=1$.
}\label{aeqn1beqn1ex1}
\end{figure}

\begin{figure}[hbtp]
\begin{centering}\includegraphics[width=.4\textwidth]{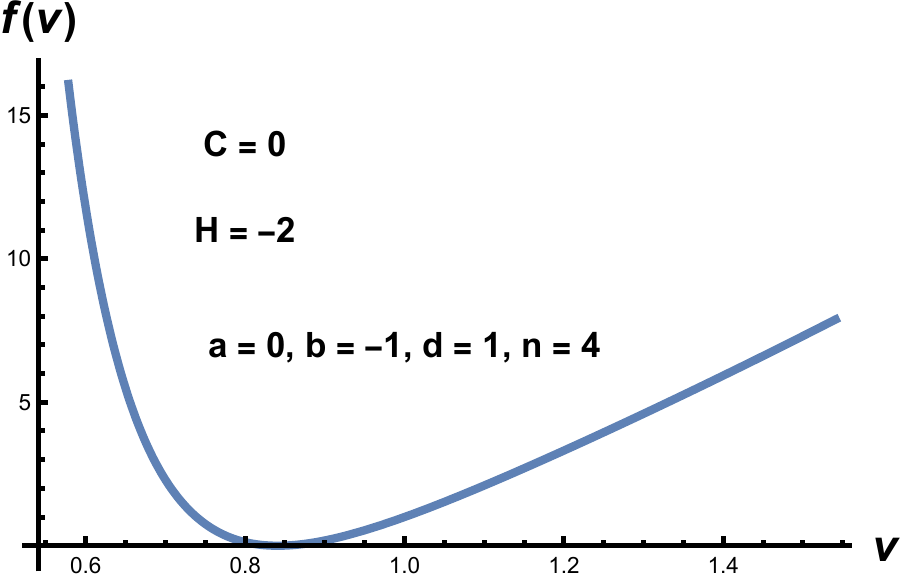} \includegraphics[width=.4\textwidth]{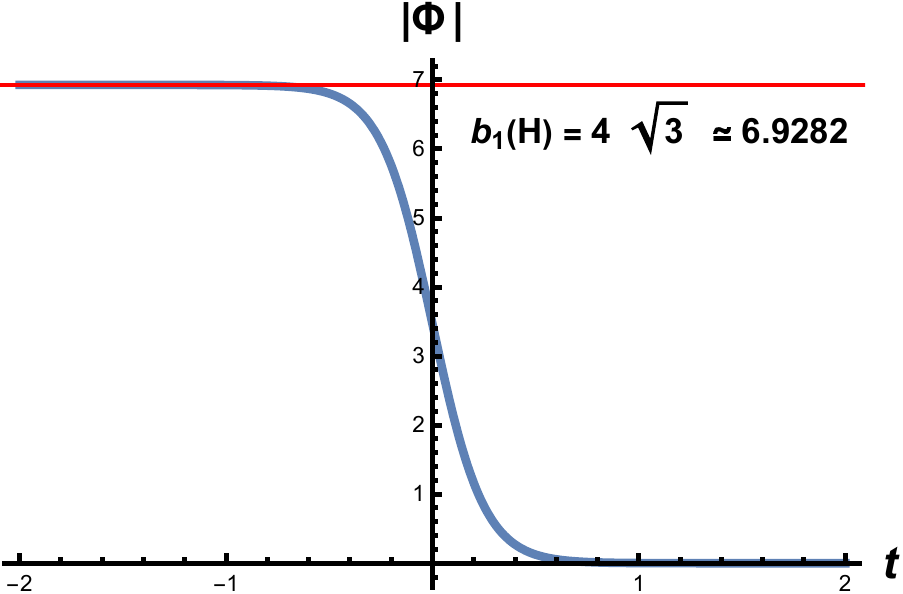}
\end{centering}
\caption{The second image shows the graph of the function $|\Phi |=\sqrt{12}\, g(t)^{-4}$ for the cmc hypersurface in the 5 dimensional Lorentz-Minkowski space $SF(5,1,0)$. For this example  $H=-2$,  $C=r_3(H)=0$, $\sup(|\Phi |)=b_3(H)=4 \sqrt{3}\approx 6.9282$. The image on the left shows the graph of the function $f(v)$. Recall that $(g^\prime(t))^2=f(g(t))$. We used the initial condition $g(0)=1$.
}\label{aeq0beqn1ex1}
\end{figure}

\vfil
\eject

\end{document}